\documentclass[10pt]{article}
\pdfoutput=1
\usepackage{fullpage}
\usepackage[utf8]{inputenc}
\usepackage{amsmath,amsbsy,amsgen,amscd,amssymb,amsthm,amsfonts,stmaryrd}
\usepackage{mathtools}
\usepackage{url}
\usepackage{hyperref}
\usepackage{xcolor}
\hypersetup{
    colorlinks,
    linkcolor={red},
    citecolor={blue},
    urlcolor={blue}
}

\makeatletter
\newtheorem*{rep@theorem}{\rep@title}
\newcommand{\newreptheorem}[2]{%
\newenvironment{rep#1}[1]{%
 \def\rep@title{#2 \ref{##1}}%
 \begin{rep@theorem}}%
 {\end{rep@theorem}}}
\makeatother

\newtheorem{theorem}{Theorem}
\newreptheorem{theorem}{Theorem}
\newtheorem{lemma}[theorem]{Lemma}
\newreptheorem{lemma}{Lemma}

\usepackage{natbib}

\def\Lcal{{\mathcal{L}}}
\newcommand{\E}{\mathbb{E}}
\newcommand{\R}{\mathbb{R}}
\newcommand{\cL}{\mathcal{L}}
\newcommand{\Zeta}{Z}
\newcommand{\dist}{\mathsf{d}}

\usepackage{algorithm}
\usepackage{algorithmic}

\newtheorem{corollary}[theorem]{Corollary}
\newtheorem{fact}[theorem]{Fact}

\theoremstyle{definition}
\newtheorem{remark}[theorem]{Remark}

\numberwithin{equation}{section}
\numberwithin{theorem}{section}

\usepackage{todonotes}

\title{Complexity of Single Loop Algorithms for Nonlinear Programming with Stochastic Objective and Constraints}
\date{}
\author{Ahmet Alacaoglu\footnote{Wisconsin Institute for Discovery, University of Wisconsin--Madison, \url{alacaoglu@wisc.edu}} \and Stephen J. Wright\footnote{Department of Computer Sciences, University of Wisconsin--Madison, \url{swright@cs.wisc.edu}}}

\begin{document}

\maketitle
\begin{abstract}
We analyze the complexity of single-loop quadratic penalty and augmented Lagrangian algorithms for solving nonconvex optimization problems with functional equality constraints. 
We consider three cases, in all of which the objective is stochastic and smooth, that is, an expectation over an unknown distribution that is accessed by sampling.
The nature of the equality constraints differs among the three cases: deterministic and linear in the first case, deterministic, smooth and nonlinear in the second case, and stochastic, smooth and nonlinear in the third case.
Variance reduction techniques are used to improve the complexity.
To find a point that satisfies $\varepsilon$-approximate first-order conditions, we require $\widetilde{O}(\varepsilon^{-3})$ complexity in the first case, $\widetilde{O}(\varepsilon^{-4})$ in the second case, and $\widetilde{O}(\varepsilon^{-5})$ in the third case. 
For the first and third cases, they are the first algorithms of ``single loop'' type (that also use $O(1)$ samples at each iteration) that still achieve the best-known complexity guarantees.
\end{abstract}

\section{Introduction}\label{sec: intro}
Augmented Lagrangian and quadratic penalty algorithms have been a mainstay for solving nonlinear optimization problems for several decades~\citep{hestenes1969multiplier,powell1969method,fiacco1968nonlinear,bertsekas2014constrained}. 
We consider first the following  nonlinear programming template:
\begin{equation}\label{eq: prob_gen1}
\min_{x\in X} f(x)\text{~subject to~} c(x) = 0,
\end{equation}
where $f:\R^d\to\R$ and $c:\R^d \to \R^m$.
Historically, algorithms for \eqref{eq: prob_gen1} are analyzed for the case of nonlinear and nonconvex $f$ and $c$. 
Typical results show \emph{asymptotic} iterate convergence with \emph{local linear} or \emph{superlinear} rate guarantees. 
In the last two decades, with the influence of the emerging field of data science, there has been wider interest in global convergence rates, including \emph{sublinear} rates.
Nonasymptotic convergence rate analyses of augmented Lagrangian method (ALM) and quadratic penalty method (QPM) for nonlinear programming in both convex~\citep{lan2013iteration,lan2016iteration,xu2017accelerated,xu2021iteration} and nonconvex cases~\citep{hong2016decomposing,xie2021complexity,li2021rate,lin2022complexity,lu2022single,huang2023single,kong2023iteration} are surprisingly recent.

With large data sets fueling many recent advances in machine learning and data sciences, \emph{stochastic} algorithms for solving \eqref{eq: prob_gen1} have become a necessity. 
These algorithms generally work with a single-sample or a mini-batch of the full dataset at each iteration. 
In many applications, even one pass over the data can be prohibitive. 
\emph{Unconstrained} and \emph{simple-constrained} versions of~\eqref{eq: prob_gen1}, where $c(x)$ is absent, have been solved with  stochastic \emph{projected} gradient algorithms for convex or nonconvex $f$~\citep{lan2020first,davis2019stochastic,cutkosky2019momentum}. 
In this paper, we focus on three subclasses of~\eqref{eq: prob_gen1} in which $c(x)$ is nontrivial, so that projection onto the feasible set of \eqref{eq: prob_gen1} is too expensive to be practical.

In all problems considered in this paper, the objective $f$ has expectation form, so we restate \eqref{eq: prob_gen1} more narrowly as follows:
\begin{equation} \label{eq:gen}
\min_{x\in X} \big\{ f(x):= \mathbb{E}_{\xi}[\tilde{f}(x, \xi)]\big\} \text{~subject to~} c(x) = 0,
\end{equation}
where $f\colon \mathbb{R}^d \to \mathbb{R}$ and $c\colon \mathbb{R}^d \to \mathbb{R}^m$ are functions with Lipschitz continuous gradients that can be nonconvex, while the set $X \subseteq \mathbb{R}^d$ is closed and convex. The function $\tilde{f}$ maps $\R^d \times \Xi$ to $\R$, where $\Xi$ is the sample space with $\xi \in \Xi$ and $\xi$ is distributed according to $P_{\Xi}$.
$\E_\xi$ is the expectation over the distribution of $\xi$. 
(We sometimes abbreviate $\E_{\xi}$ as $\E$ when the context is clear.)

\paragraph{Three cases. } We consider three instances of the template \eqref{eq:gen}, motivated by applications in machine learning and data science.

\begin{enumerate}
\item[I.] Let $c(x):= Ax-b$ and $X = \mathbb{R}^d$, yielding
\begin{equation}
\label{eq:I} \tag{I}
\begin{aligned} 
\min_{x\in \mathbb{R}^d} \big\{ f(x):= \mathbb{E}[\tilde{f}(x, \xi)]\big\} 
\text{~subject to~} Ax=b,
\end{aligned}
\end{equation}
where $A \in \R^{m \times d}$.
This problem arises in the context of distributed optimization where the linear constraints enforce consensus~\citep{hong2018gradient,hong2016decomposing}.
It also arises in resource allocation~\cite[Section~7.3]{boyd2011distributed}, reformulations of problems involving the composition of convex or nonconvex functions with linear operators, and other contexts (see also \cite[Section 1.1]{hong2016decomposing}).
Oracle accesses in this case require stochastic gradients of $f$ and matrix multiplications by $A$ and  $A^\top$.
We denote by $\delta > 0$ the smallest nonzero eigenvalue of $A^\top A$, so that \begin{equation} \label{eq:def.delta}
\|A^\top \lambda \| \ge \sqrt\delta \| \lambda \| \quad \mbox{for all $\lambda \in \mbox{Range}(A)$.}
\end{equation}
\item[II.] 
In the second case, $c: \mathbb{R}^d \to \mathbb{R}^m$ is a deterministic nonlinear function, yielding
\begin{equation} 
\label{eq:II} \tag{II}
\begin{aligned}
\min_{x\in X} \big\{ f(x):= \mathbb{E}[\tilde{f}(x, \xi)]\big\} 
\text{~subject to~} c(x) = 0,
\end{aligned}
\end{equation}
with $X \subseteq \mathbb{R}^d$ a closed, convex set.
Problems of this type arise in optimization problems with partial differential equation (PDE) constraints~\citep{kupfer1992numerical, rees2010optimal,curtis2021worst}.
The oracle access requires stochastic gradients of $f$, and evaluations of constraint function and gradient, $c$ and $\nabla c$.
\item[III.]
The third problem has the constraint $c$ as a nonlinear function defined as an expectation
\begin{equation} \label{eq:III} \tag{III}
\begin{aligned}
\min_{x\in X} \big\{ f(x):= \mathbb{E}_\xi[\tilde{f}(x, \xi)]\big\}, \text{subject to~} c(x) = \mathbb{E}_{\zeta}[\tilde{c}(x, \zeta)] = 0,
\end{aligned}
\end{equation}
where $\tilde{c}:\R^d \times \Zeta \to \R^m$, with $\Zeta$ being the sample space and $X\subseteq \mathbb{R}^d$ is closed and convex.
The oracle access requires stochastic gradients and stochastic function evaluations for both objective and constraint. 
This problem is motivated by recent applications of neural network training with output constraints, in such tasks as out-of-distribution detection or fair machine learning~\citep{katz2022training, liu2020chance, dener2020training, zafar2019fairness}.
\end{enumerate}
Problems \eqref{eq:I}, \eqref{eq:II}, \eqref{eq:III} are studied in Sections~\ref{sec: main_linear}, \ref{subsec: main_det_nonlinear}, and \ref{sec: qp_stoc_const}, respectively.
Inequality constraints can be accommodated into \eqref{eq:II} and \eqref{eq:III} via the use of slack variables, which can be constrained to be nonnegative by membership in an appropriately defined set $X$. 

\paragraph{First-order stationarity.} We say that $\bar x$ is {\em $\varepsilon$-stationary} for~\eqref{eq:gen} if there exists $\bar \lambda\in\mathbb{R}^m$ such that
\begin{equation}\label{eq:1o.eps}
\begin{aligned}
\dist(\nabla f(\bar x) + \nabla c(\bar x)^\top \bar\lambda, -N_X(\bar x)) &\leq \varepsilon, \\ \| c(\bar x) \| &\leq \varepsilon,
\end{aligned}
\end{equation}
where $\dist(x, C) =\min_{y\in C}\|x-y\|$ is the distance function and $N_X(\bar{x})$ is the normal cone to $X$ at $\bar{x}$. This is the same first-order stationarity definition as in~\cite{sahin2019inexact,li2021rate,lin2022complexity,li2023stochastic} and generalizes~\cite{xie2021complexity} to the case when $X$ is present in the problem formulation.

We say that $\bar x$ is an {\em $\varepsilon$-stationary point in expectation} if~\eqref{eq:1o.eps} holds in expectation.

\paragraph{Augmented Lagrangian and Quadratic Penalty.} Augmented Lagrangian methods were proposed to overcome both the theoretical and practical drawbacks of quadratic penalty (QP) approaches \citep{hestenes1969multiplier}. 
Instances of ALM are traditionally equipped with stronger guarantees than QPM; they incorporate dual updates that help with feasibility guarantees and subproblem conditioning \citep{bertsekas2014constrained}. 
However, for existing nonasymptotic analyses in the nonconvex cases, the literature does not reflect these advantages. 
In the deterministic case, the existing analyses for ALM with best-known guarantees require the penalty parameters to increase rapidly to infinity, so the dual step size effectively needs to decay~\citep{sahin2019inexact,li2021rate}. 
The only work to our knowledge with a constant dual step size and penalty parameter is~\cite{xie2021complexity}, which has worse complexity than the methods with growing penalty parameters and decaying dual step sizes.

Since constant penalty parameter and constant dual step sizes are the main features of ALM, we focus on a version of ALM for problem \eqref{eq:I} that uses constant penalty parameters and constant dual step sizes. 
For the more general problems, we focus on QP-based algorithms and their extensions to ALM-type algorithms with small dual step sizes.

\paragraph{Oracle model.} 
We assume throughout to have access to an unbiased oracle for gradients, a standard setting used for example in  \citep{arjevani2022lower,cutkosky2019momentum}. That is, there exists $\tilde\nabla f(x, \xi)$ such that
\begin{equation}\label{eq: oracle}
\begin{aligned}
    \mathbb{E}_\xi[\tilde \nabla f(x, \xi)] &= \nabla f(x), \\
    \mathbb{E}_\xi\| \tilde \nabla f(x, \xi) - \tilde \nabla f(y, \xi) \|^2 &\leq L^2 \| x-y\|^2.
    \end{aligned}
\end{equation}

\paragraph{Algorithmic Approaches and Contributions.}

We handle the functional constraints via two classical approaches: quadratic penalty and augmented Lagrangian~\citep{bertsekas2014constrained,numopt2}. 
These algorithmic frameworks normally require solution of subproblems at every iteration. 
Instead of solving these subproblems exactly, we perform one stochastic gradient descent step on each subproblem, yielding an overall approach that is single-loop in nature. 
This technique is also known as \emph{linearization} in the context of ALM and QPM~\citep{ouyang2015accelerated}. 
To obtain improved sample complexity guarantees, we also use variance reduction techniques (see~\cite{cutkosky2019momentum}). Our main aim in the paper is to provide \emph{simple} and \emph{easily implementable} single-loop algorithms for the described problem classes with optimal or best-known complexity results.\footnote{Such a goal is nontrivial and not always achievable. See e.g., \cite{ji2022will} and \cite{zhang2020single} for different settings where one currently needs multiple loops for best complexity and research for single-loop methods is active.}

In the case of linear constraints studied in Section~\ref{sec: main_linear}, we use constant penalty parameter and constant dual step sizes for an ALM with variance reduction. In this case, we show the complexity $\widetilde{O}(\varepsilon^{-3})$ which is optimal (up to a log factor) even for unconstrained, smooth, stochastic optimization~\citep{arjevani2022lower}.

With functional constraints,  consistent with the literature on deterministic instances of our template, we use increasing penalty parameters with quadratic penalty (and decreasing dual step sizes with ALM in Section~\ref{subsec: main_alm_stoc}). We show $\widetilde{O}(\varepsilon^{-4})$ complexity with deterministic constraints in Section~\ref{subsec: main_det_nonlinear} and  $\widetilde{O}(\varepsilon^{-5})$ complexity with stochastic constraints in Section~\ref{sec: qp_stoc_const}.

Besides being single-loop and requiring only a single sample at each iteration, each iteration of our algorithms requires only simple projections and simple vector operations.
We do not require complicated auxiliary subproblems to be solved.
As a consequence, our sample complexity and computational complexity results are essentially the same.

\subsection{Related Works}\label{subsec: relworks}

Algorithms for nonlinear programming have been studied for many decades, but recent years have seen more focus on the case of functions defined as expectations.
There is focus also on algorithms that identify an approximate solution in some finite time, expressed in terms of a parameter $\varepsilon > 0$ that quantifies the inexactness in the solution of the problem.
For ease of presentation we discuss the related works separately for each of our three special cases.

\paragraph{Problem \eqref{eq:I}: Nonconvex stochastic optimization with linear constraints. } 
Complexity of ALM in the  case of deterministic $f$ is studied in many  works; see for example \cite{zhang2020proximal,zhang2022global,hong2016decomposing,hong2018gradient}.
A complexity of $O(\varepsilon^{-2})$ is typical for identifying a point $\bar{x}$ that satisfies first-order conditions $\varepsilon$-approximately in the sense of \eqref{eq:1o.eps}, for some $\bar\lambda$.
Among the works mentioned, \cite{hong2016decomposing,hong2018gradient} focus on the unconstrained case and~\cite{zhang2020proximal,zhang2022global,zhang2022iteration} focus on the case in which additional polyhedral constraints (or more general nonlinear functional constraints with further assumptions) are present,  requiring error bounds to estimate distances to the optimal set.
One feature of the methods in these works is that both the penalty parameter and the dual step size in ALM are constant.

With nonconvex and stochastic objective,~\cite{huang2019faster} obtained complexity ${O}(\varepsilon^{-3})$ with additional assumptions such as $A$ having a full rank, large batch sizes depending on accuracy $\varepsilon$, and a uniform upper bound on  $\|\nabla f(x)\|^2$.
(The latter often does not hold for problems in the form~\eqref{eq:I}.)
See Sec.~\ref{subsec: analysis_linconst} for the details and how we address these shortcomings\footnote{The same limitations are present in~\cite[Thm. 5.6]{lin2022stochastic}.}. 

The work of~\cite{zhang2021gt} focuses on a consensus-optimization instance of \eqref{eq:I} with nonconvex stochastic objective, and uses gradient tracking and the variance-reduction approach  from~\cite{cutkosky2019momentum} to obtain $\widetilde{O}(\varepsilon^{-3})$ complexity.  
We achieve the same rate for a more general problem than consensus optimization; see for instance \cite[Sec.~1]{hong2016decomposing},~\cite[Section~7]{boyd2011distributed} for a ``sharing problem" example or \cite{boct2020proximal}, for standard splitting approaches to represent composite optimization problems as linearly constrained optimization.
Our ALM type method is more general and different from the problem-specific method of~\cite{zhang2021gt}.

\paragraph{Problem \eqref{eq:II}: Nonconvex stochastic optimization with nonlinear deterministic constraints.} 
For this problem, \cite{shi2022momentum} analyzes an algorithm similar to ours except that the penalty parameter is fixed (ours is variable) and depends on the predefined number of iterations $K$. 
Their approach involves an initial stage of finding a  feasible point of the nonconvex constraint and has complexity $O(\varepsilon^{-4})$. 
We show that the initial stage is unnecessary when we use variable parameters that depend on the current iterate $k$. 
The  sequential quadratic programming (SQP) method of \cite{curtis2021worst} has sample complexity $\widetilde{O}(\varepsilon^{-4})$,  but this paper does not address iteration complexity or computational complexity directly. 
In addition, each iteration requires solution of a linear system, a more expensive operation than the vector operations required at each iteration of an ALM.

\paragraph{Problem \eqref{eq:III}: Nonconvex stochastic optimization with nonlinear stochastic constraints. } 
The two works building on a \emph{regularization} idea for solving problem \eqref{eq:III}, with inequality constraints instead of equalities, are~\cite{boob2022stochastic} and~\cite{ma2020quadratically}. 
Both assume the existence of a strictly feasible solution, so their applicability to equality constraints is not clear. 
Both describe a  complexity of $ O(\varepsilon^{-6})$ on a slightly weaker assumption on Lipschitz continuity of the gradient. The algorithms of these papers have a double loop structure, compared to the single-loop algorithm that we analyze.

The recent independent work of~\cite{li2023stochastic} considered a similar idea of using STORM estimator for this problem to obtain complexity $O(\varepsilon^{-5})$.
In contrast to us, they analyzed an inexact ALM. 
Apart from the complicated structure of a double loop method, an important drawback of this approach is that termination rule of the inner loop is generally not implementable. 
This is because the number of required iterations of the inner loop depends on the optimal value of the subproblems, variance upper bounds or other unknown values\footnote{\cite[Lemma 5]{li2023stochastic} suggests that optimal value of subproblems can be replaced by other values such as the diameter of balls containing the iterates, upper bound of function values or the parameter of regularity-condition ($\delta$ in~\eqref{eq: asp_gen4} in our notation, $v$ in \cite[Assumption 3]{li2023stochastic}) to set the number of inner iterations. Unfortunately, these values are also normally unknown.}.
The other alternative for termination of the inner loop requires computing first order stationarity, which in turn requires the computation of full gradients, an operation that is not practical with stochastic algorithms. 
By contrast, single-loop any-time algorithms like ours have a  straightforward implementation both conceptually and in practice.

\paragraph{Notation. } To improve readability, we use standard asymptotic notations such as $O$, $\Omega$, $\asymp$ in the main text by suppressing universal constants. The distance between a point $x\in\mathbb{R}^n$ and a set $C\subseteq\mathbb{R}^n$ is denoted as $\dist(x, C)=\min_{y\in C}\|x-y\|$. \emph{Any-time} refers to an algorithm that does not require setting $\varepsilon$ in advance.
\section{Linear Constraints: Problem \eqref{eq:I}}\label{sec: main_linear}
\begin{algorithm*}
		\small
		\caption{Stochastic Linearized Augmented Lagrangian Method with Variance Reduction for~\eqref{eq:I}}
		\label{algorithm:lalm_lincons}
		\begin{algorithmic}[1]
			\STATE {\textbf{Input}:} Initialize $\lambda_0, x_0, g_0$ arbitrarily and $\alpha_k, \rho, \eta_k$ as in~\eqref{subsec: analysis_linconst}. 
   \FOR{$k=0, 1, \dots$}
   \STATE $x_{k+1} = x_k - \eta_{k+1} (g_k + A^\top \lambda_k + \rho A^\top(Ax_k - b))$
   \STATE $\lambda_{k+1} = \lambda_k + \rho (Ax_{k+1} - b)$
   \STATE Sample $\xi_{k+1}\sim P_{\Xi}$ and set $g_{k+1} = \tilde \nabla {f}(x_{k+1}, \xi_{k+1}) + (1-\alpha_{k+1})(g_k - \tilde \nabla {f}(x_k, \xi_{k+1}))$
   \ENDFOR
		\end{algorithmic}
	\end{algorithm*}
\subsection{Algorithm and the Main Result}
In this section, we address \eqref{eq:I}, restated here as
\begin{equation} \tag{I}
\min_{x\in \mathbb{R}^n} \big\{ f(x):= \mathbb{E}[\tilde{f}(x, \xi)]\big\} \text{~subject to~} Ax=b,
\end{equation}
for which the augmented Lagrangian is
\begin{equation*}
\mathcal{L}_\rho(x, \lambda) := f(x) + \langle \lambda, Ax-b\rangle + \frac{\rho}{2} \| Ax - b\|^2,
\end{equation*}
for parameter $\rho>0$. 
We describe and analyze a linearized ALM given in Algorithm~\ref{algorithm:lalm_lincons}, in which a single step of stochastic gradient descent replaces the minimization  of augmented Lagrangian with respect to the primal variable.  
This algorithm can be seen as a variance-reduced version of ALM with constant step sizes, studied in the deterministic setting by \cite{hong2016decomposing}.
Due to stochasticity in the objective, we use a variance-reduced estimator of $\nabla f(x_k)$ based on sampling of the oracle $\tilde \nabla {f}(x_i,\xi_i)$ for $i=0,1,\dotsc,k$; see~\eqref{eq: oracle} for the oracle description.
The output of the algorithm (denoted as $\bar x$ in Thm \ref{thm: informal_lincons}), after running for $K$ iterations, is a randomly selected primal-dual pair, i.e., $(x_{\hat k}, \lambda_{\hat k})$ where $\hat k$ is selected uniformly at random from $\{1,2,\dots, K\}$.

In this section, we obtain optimal complexity results with constant penalty parameter / dual step size $\rho$ in ALM. 
The latter feature of the algorithm is the main challenge in the analysis, and is the reason for our separate focus on the linearly constrained case.

We make the following assumptions in this case (see also~\eqref{eq: oracle}):
\begin{equation}\label{eq: asp_case1}\tag{A1}
\begin{aligned}
\mathbb{E}_\xi \| \tilde\nabla {f}(u, \xi) - \tilde\nabla {f}(v, \xi) \|^2 &\leq L_f^2\|u - v \|^2, \\
\mathbb{E}_\xi\|\tilde\nabla {f}(x, \xi) - \nabla f(x) \|^2 &\leq V^2, \\
f(x) &\geq 0, ~~\forall x.
\end{aligned}
\end{equation}
The first assumption in \eqref{eq: asp_case1} is Lipschitz continuity of the gradients on average (also called mean-square smoothness, see~\cite[eq. (4)]{arjevani2022lower}) while the second is a standard variance bound. 
By Jensen's inequality, the first inequality in~\eqref{eq: asp_case1} also implies that $
\|\nabla f(u) - \nabla f(v) \|\leq L_f\|u-v\|$.
The last assumption in \eqref{eq: asp_case1}, also made in~\cite{hong2016decomposing} is without loss of generality\footnote{As mentioned in \cite[footnote 1, pg 5]{hong2016decomposing} this assumption is equivalent to lower boundedness of $f$. }.

In the following subsection, we prove the following result, stated informally here for simplicity. 
The choices of parameters $\rho$, $\alpha_k$, and $\eta_k$, and the full result appear as Theorem~\ref{th: main1} and Corollary~\ref{cor: sor4}. 
\begin{theorem}[Informal]\label{thm: informal_lincons}
    With the assumptions in~\eqref{eq: asp_case1} and suitable choices of $\eta_k, \rho, \alpha_{k+1}$ as in~\eqref{eq: params_sec1}, Algorithm~\ref{algorithm:lalm_lincons} outputs $(\bar x, \bar\lambda)$ such that
\begin{equation*}
\mathbb{E}\| \nabla f(\bar x) + A^\top \bar \lambda \| \leq \varepsilon \text{~~~and~~~} \mathbb{E}\| A\bar x - b\| \leq \varepsilon,
\end{equation*}
after $K=\widetilde{O}(\varepsilon^{-3})$ iterations, thus requiring $\widetilde{O}(\varepsilon^{-3})$ evaluations of stochastic gradients of $f$.
\end{theorem}
\begin{remark}
An important aspect of this result, which is critical for ALM, is that the penalty parameter and the dual step size $\rho$ is a constant and independent of final accuracy $\varepsilon$ or iteration counter $k$. 
The choices $\alpha_k, \eta_k$ are independent of $\varepsilon$ and they depend on $k$ because of the stochastic setting we focus on. 
The independence from $\varepsilon$ is critical to ensuring that a certain potential function can increase only by a controlled amount at every iteration, which in turn is important for lower boundedness of the expected potential.
Further explanations appear in Section~\ref{subsec: analysis_linconst}.

\end{remark}
\subsection{Analysis}\label{subsec: analysis_linconst}
We use the following parameters which are written with asymptotic notations for readability (recall \eqref{eq:def.delta} and \eqref{eq: asp_case1}). 
We suppress only universal constants; full specifications are provided in~\eqref{eq: params_sec1_app}.
\begin{equation}\label{eq: params_sec1}
\begin{aligned}
\rho &\asymp \frac{L_f}{\delta},~~~k_0 =\Omega(\mathrm{poly}(\delta^{-1})),~~~
\eta = \frac{1}{11(L_f + \rho \| A\|^2)}, \\
\eta_k &= \frac{\eta}{(k+k_0)^{1/3}\log(k+k_0)},~~~ \alpha_k = 121L_f^2\eta_k^2. 
\end{aligned}
\end{equation}
We start with a lemma that analyzes a single iteration of the algorithm. 
This lemma constructs a potential function $Y_k$ that we show later to be non-increasing in expectation, up to a small error. 
Due to the combination of ALM with constant dual step sizes/penalty parameters and the use of variance-reduction techniques, each with complicated constants, the coefficients in this lemma are rather involved. 
We only provide the orders of some terms, which suffice to convey the central ideas in the lemma.
\begin{lemma}\label{lem: skl4}
Let the assumptions in~\eqref{eq: asp_case1} hold. Set $\rho$, $\eta_k$, $\alpha_{k}$ as~\eqref{eq: params_sec1}. For the iterates of Alg.~\ref{algorithm:lalm_lincons}, we have
\begin{equation}\label{eq: gjt7}
\begin{aligned}
\mathbb{E}Y_{k+1} \leq \mathbb{E}Y_k  - \frac{\eta_{k+1}}{2} \mathbb{E} \| g_k - \nabla f(x_k)\|^2  + (\beta_{1, k+1} + \beta_{2, k+1}) \mathbb{E} \| x_{k+1} - x_k\|^2 + w_k + v_k V^2,
\end{aligned}
\end{equation}
where
    \begin{equation*}
\begin{aligned}
Y_{k+1} &= L_{\rho}(x_{k+1}, \lambda_{k+1}) + \frac{\rho m}{2\eta_{k+2}} \| Ax_{k+1} - b\|^2  + \frac{m}{2\eta_{k+1}} \| x_{k+1} - x_k\|^2_{Q_{k+1}} + \beta_{1, k+1} \| x_{k+1} - x_k\|^2 \\
    &\quad+ \frac{2}{c\eta_{k+1}} \| g_{k+1} - \nabla f(x_{k+1})\|^2 + \left(\frac{6(1+c_1)}{\rho\delta} + \frac{4m}{L_f\eta_k} \right)  \| g_{k} - \nabla f(x_{k})\|^2
\end{aligned}
\end{equation*}
with $c_1 > 0$, $c =121L_f^2$ , $m \asymp \frac{1}{L_f}$, and
\begin{align*}
    \beta_{1, k} &\asymp L_f+\eta_{k+1}^{-1},~~~
    \beta_{2, k} = O\left(\frac{L_f}{\delta}+\eta_{k+1}^{-1}\right) - \frac{1}{2\eta_{k}}, \\
    w_k &= \frac{3(1+c_1)}{\delta\rho} \mathbb{E}\| x_{k+1} - 2x_k + x_{k-1}\|^2_{Q_{k+1}^\top Q_{k+1}} - \frac{m}{2\eta_{k+1}}\mathbb{E} \| x_{k+1} - 2x_k + x_{k-1} \|^2_{Q_{k+1}},\\
    v_k &= O\left(\frac{\alpha_k^2}{\eta_k}\right),~~~ Q_k = \eta_k^{-1} I - \rho  A^\top A \succeq 0.
\end{align*}
\end{lemma}
At a high level, the lemma requires us  to show that \emph{(i)} $\beta_{1, k+1} + \beta_{2,k+1} < 0$, \emph{(ii)} $w_k \leq 0$ and \emph{(iii)} $\sum_{k=1}^\infty v_k < \infty$ to obtain that the function $Y_k$ is non-increasing in expectation up to a small error (see~\eqref{lem: skl4}). 
The parameter choices of \eqref{eq: params_sec1} with constants chosen as in in~\eqref{eq: params_sec1_app} can be shown to achieve the required properties, by a tedious but straightforward analysis.

The main theorem of this section utilizes the single-iteration inequality described in Lemma~\ref{lem: skl4} to show that both scaled iterate differences and variance term are small. 
Approximate stationarity follows in view of Theorem \ref{thm: informal_lincons} via standard reductions described in Corollary \ref{cor: sor4}. 
We provide a proof sketch to illustrate the main ideas; details are deferred to Section~\ref{app: lincons_main}.
\begin{theorem}\label{th: main1}
Let the assumptions in~\eqref{eq: asp_case1} hold and suppose that $\eta_k$ and the other algorithmic parameters are chosen as in~\eqref{eq: params_sec1} (see also~\eqref{eq: params_sec1_app}).
Then, for the iterates of Algorithm~\ref{algorithm:lalm_lincons}, we have for any $K>1$ that 
\begin{align*}
&\frac{1}{K}\sum_{k=1}^{K+1} \mathbb{E} \left[ \| \eta_{k}^{-1}(x_{k} - x_{k-1})\|^2 + \| g_{k-1} - \nabla f(x_{k-1})\|^2 \right]\\&= \widetilde{O}(K^{-2/3}).
\end{align*}
\end{theorem}
\begin{proof}[Proof sketch]
In the result of Lemma~\ref{lem: skl4}, we use the parameter choices given in~\eqref{eq: params_sec1} to obtain
\begin{align}
\mathbb{E}Y_{k+1} &\leq \mathbb{E}Y_k  -\frac{1}{16\eta_{k+1}}\mathbb{E} \| x_{k+1} - x_k\|^2 - \frac{\eta_{k+1}}{2} \mathbb{E} \| g_k - \nabla f(x_k)\|^2  + v_k V^2.\label{eq: nuy4}
\end{align}
Note that by the choices of $\eta_k, \alpha_k$ and the definition of $v_k$, we have that $\sum_{k=1}^\infty v_k = O(1)$.
By adjusting  \eqref{eq: nuy4} and summing for $k\geq 1$, we obtain
\begin{align*}
&\frac{\eta_{K+1} }{32}\sum_{k=1}^{K+1} \left(\mathbb{E} \| x_{k} - x_{k-1} \|^2 + \mathbb{E} \| g_{k-1} - \nabla f(x_{k-1})\|^2\right)\\
&\leq  \mathbb{E}Y_1 + \frac{1}{32\eta_1}\| x_1-x_0\|^2 + \frac{\eta_1}{4}\|g_0 - \nabla f(x_0)\|^2  -\mathbb{E} Y_{K+1} +O(1).
\end{align*}
To ensure that the the right-hand side is upper bounded by a constant, we need to show $\mathbb{E}Y_{k+1}$ is lower bounded. This is not immediate, because our use of a constant dual step size blocks the derivation of a uniform upper bound on the norm of dual variable $\lambda_k$. 
Lack of monotonicity of $\mathbb{E}Y_k$ also prevents us from using the estimates available in deterministic cases; see \cite{hong2016decomposing}. 
In Lemma~\ref{lem: potential_lb}, we show that the \emph{almost} monotonicity of $\mathbb{E}Y_k$ given in~\eqref{eq: nuy4} is sufficient to show  lower boundedness. This fact leads to the result.
\end{proof}
With this result, we can use the standard reductions of Corollary \ref{cor: sor4} to prove Theorem~\ref{thm: informal_lincons}. 
It is worth noting that most of the estimations in the analysis would simplify if we were to fix all $\eta_k, \alpha_k$ at values that depend on the final iterate $K$ (or equivalently $\varepsilon$). 
However, such choices  would not suffice to show lower boundedness of $\mathbb{E}Y_k$ mentioned above, which is necessary to obtain the right constants. 
Our use of variable step sizes also allows us to  derive an ``any-time" algorithm with no need to set an accuracy $\varepsilon$ in advance. 
For example,~\cite{huang2019faster} needed to assume uniform boundedness of $\|\nabla f(x_k)\|^2$ which trivially implies a lower bound for the potential; see \cite[eq. (69)]{huang2019faster}. 
However, this assumption does not hold normally and the limitations of bounded gradient assumption are well-known; see, for example, \cite[Section 3]{yang2023two}, \cite{faw2022power}. 
Our analysis does not need this restriction to obtain a lower bound for the potential. 

\section{Stochastic Constraints: Problem~\eqref{eq:III}}\label{sec: qp_stoc_const}
\begin{algorithm*}
		\small
		\caption{Stochastic Linearized Quadratic Penalty Method with Variance Reduction for~\eqref{eq:III}}
		\label{algorithm:stoc_lalm}
		\begin{algorithmic}[1]
			\STATE {\textbf{Input}:} Initialize $x_1\in X$ and $g_1 = \tilde\nabla Q_{\rho_1}(x_1,B_1)$ and $\alpha_k, \rho_k, \eta_k$ as in Theorem \ref{eq: asb4}
   \FOR{$k=1, 2, \dots$}
   \STATE $x_{k+1} = P_{X}(x_k - \eta_k g_k)$
   \STATE $\text{Sample $\xi^0_{k+1}$, $\zeta^{1}_{k+1}$, $\zeta^{2}_{k+1}$ to get }B_{k+1} = (\xi^0_{k+1}, \zeta^{1}_{k+1}, \zeta^{2}_{k+1}) \in \Xi \times \Zeta^{2}$ where $\zeta^{1}_{k+1}$ and $\zeta^{2}_{k+1}$ are i.i.d.
   \STATE $\tilde{\nabla}Q_{\rho}(x, B_{k+1}) = \tilde{\nabla}f(x, \xi^{0}_{k+1}) + \rho \sum_{i=1}^m \tilde \nabla c_{i}(x, \zeta^{1}_{k+1}) \tilde c_i(x, \zeta^{2}_{k+1})$
    \STATE $g_{k+1} = \tilde\nabla Q_{\rho_{k+1}}(x_{k+1}, B_{k+1}) + (1-\alpha_{k+1})(g_k - \tilde\nabla Q_{\rho_{k}}(x_{k}, B_{k+1}))$
   \ENDFOR
		\end{algorithmic}
	\end{algorithm*}
 \subsection{Algorithm and the Main Result}\label{subsec: alg_stoc_qp}
 In this section, we address \eqref{eq:III}, restated here as
 \begin{equation*} 
\begin{aligned}
\min_{x\in X} \big\{ f(x):= \mathbb{E}[\tilde{f}(x, \xi)]\big\} \text{~subject to~} c_i(x) := \mathbb{E}_{\zeta}[\tilde{c}_i(x, \zeta)] = 0, ~~~\forall i\in\{1,\dots,m\},
\end{aligned}
\end{equation*}
where $X$ is convex and closed, $f$ and  $c_i$, $i=1,2,\dotsc,m$  are smooth functions, $c(x)=(c_1(x),\dots,c_m(x))^\top$, and $\nabla c$ is the Jacobian.\footnote{One can consider $m=1$ in the first reading for simplicity.} The notation $P_X$ denotes projection onto $X$. In addition to the oracle model described in~\eqref{eq: oracle}, in this section we also have access to $\tilde \nabla c_i$ such that $\mathbb{E}[\tilde \nabla c_i(x, \zeta)] = \nabla c_i(x)$.
We assume that there are constants $\tilde{L}_{\nabla f}$, $\tilde{L}_{\nabla c}$, and $\tilde{L}_c$ such that for all $x, y$, we have
\begin{equation}\label{eq: htr_eqs}\tag{A2}
\begin{aligned}
\mathbb{E}_{\xi}\| \tilde \nabla f(x, \xi) - \tilde \nabla f(y, \xi) \|^2 &\leq \tilde L_{ \nabla f}^2 \| x - y\|^2, 
\\
\mathbb{E}_{\zeta}\| \tilde \nabla c(x, \zeta) - \tilde \nabla c(y, \zeta) \|^2 &\leq \tilde L_{ \nabla c}^2 \| x - y\|^2, 
\\
\mathbb{E}_{\zeta}\| \tilde c(x, \zeta) - \tilde c(y, \zeta) \|^2 &\leq  \tilde L_{c}^2 \| x - y\|^2.
\end{aligned}
\end{equation}
These conditions are stronger than mere smoothness of $f, c$ but are necessary for variance reduction in general~\citep{arjevani2022lower}. 
Other recent works (for example, \citep{boob2022stochastic}) do not use this assumption but obtain a slightly worse complexity result (see Section~\ref{subsec: relworks}).

Algorithm \ref{algorithm:stoc_lalm} is based on the quadratic penalty function
\begin{equation*}
Q_{\rho}(x) = f(x) + \frac{\rho}{2} \sum_{i=1}^m (c_i(x))^2.
\end{equation*}
The variance reduced estimator $g_{k+1}$ is based on STORM ~\citep{cutkosky2019momentum}.
The quadratic terms $(c_i(x))^2$ are not in the suitable form to apply SGD due to their \emph{compositional} structure.
However, it is a special form for which simply using independent samples can give an unbiased sample for the gradient. (This observation appeared in the recent independent work~\citep{li2023stochastic}.)
Let us define the stochastic oracle
\begin{align}
\tilde{\nabla}Q_{\rho}(x, B)= \tilde{\nabla}f(x, \xi^{0}) + \rho \sum_{i=1}^m \tilde \nabla c_{i}(x, \zeta^{1}) \tilde c_i(x, \zeta^{2})\label{eq: sor1}
\end{align}
 where $B = (\xi^0, \zeta^{1}, \zeta^{2}) \in \Xi \times \Zeta^2$ with $\zeta^{1}$ and $\zeta^{2}$ i.i.d. 
 We then have that $\mathbb{E}[\tilde{\nabla}f(x, \xi^0)] = \nabla f(x_k)$. Additionally, $\forall i =1,2,\dotsc,m$, we have
\begin{align*}
\mathbb{E}_{\zeta^{1}, \zeta^{2}} [\tilde \nabla c_i(x, \zeta^{1}) \tilde c_i(x, \zeta^{2})] 
&= \mathbb{E}_{\zeta^{1}}[\tilde \nabla c_i(x, \zeta^{1})] \mathbb{E}_{\zeta^{2}}[\tilde c_i(x, \zeta^{2})] \\
&=\nabla c_i(x)c_i(x),
\end{align*}
where the first step is by independence of $\zeta^{1}$ and $\zeta^{2}$. Hence, we have $\mathbb{E}\tilde \nabla Q_{\rho}(x) = \nabla Q_\rho(x)$.
We assume that there are positive constants $\sigma_f$, $\sigma_{\nabla c}$, and $\sigma_c$ such that
\begin{equation}\label{eq: asp_gen2}\tag{A3}
\begin{aligned}
\mathbb{E} \| \tilde \nabla f(x, \xi) - \nabla f(x) \|^2 &\leq \sigma_f^2,\\
\mathbb{E} \| \tilde \nabla c(x, \xi) - \nabla c(x) \|^2 &\leq \sigma_{\nabla c}^2, \\\mathbb{E} \|  c(x, \xi) - c(x) \|^2 &\leq \sigma_c^2,
\end{aligned}
\end{equation}
a set of assumptions also made in~\cite[eq. (2.9)]{boob2022stochastic}.
Other assumptions include the following:
\begin{equation}\label{eq: asp_gen3}\tag{A4}
\begin{aligned}
&\| \nabla c(x) \| \leq C_{\nabla c},~~~ \|c(x) \| \leq C_{c},\\
&\| \tilde \nabla c(x, \zeta) \| \leq \tilde C_{\nabla c},~~~ \|c(x, \zeta) \| \leq \tilde C_{c},\\
&|f(x_k)|\leq B_f,~~~ \|\nabla f(x)\|\leq C_{\nabla f},\\
&Q_{\rho}(x)\geq \underline Q > -\infty ~~\forall \rho, x.
\end{aligned}
\end{equation}
These boundedness assumptions are widespread for nonconvex constrained problems, even with deterministic objective and constraints, see for example \citep{sahin2019inexact,lin2022complexity,li2021rate}.

Under these assumptions, we have that $x\mapsto Q_{\rho_k}(x)$ is $L_{\rho_k}$-smooth with $L_{\rho_k} = \rho_k(L_{\nabla f} + m(C_c L_{\nabla c} + C_{\nabla c} L_c))$ 
which, for example, we can see by direct calculation on the gradient.
For variance reduction, we use 
\begin{equation}\label{eq: stoc_lips}
\begin{aligned}
\mathbb{E}_B \| \tilde \nabla Q_{\rho_k}(x, B) - \tilde \nabla Q_{\rho_k}(y, B) \|^2 \leq \tilde L_{\rho_k}^2 \| x-y\|^2,
\end{aligned}
\end{equation}
where $\tilde L_{\rho}^2  = \tilde{L} \rho^2, \; \tilde{L} := 4\tilde L_{\nabla f}^2 + 4m^2 (\tilde C_c^2 \tilde L_{\nabla c}^2 + \tilde C_{\nabla c}^2 \tilde L_c^2)$, with $\rho_k\geq 1$.
This can be shown the same way as the $L_{\rho}$, by using \eqref{eq: htr_eqs} and \eqref{eq: asp_gen3} (see also \eqref{eq: uew3} and \eqref{eq: vpt4}). 

We also use a generalization of the full rank assumption on the Jacobian. Recall that without the set inclusion constraint $x \in X$, this means $\| \nabla c(x_k)^\top c(x_k) \| \geq \delta |c(x_k)|$.  
With constraints, we assume
\begin{equation}\label{eq: asp_gen4}\tag{A5}
\dist(\nabla c(x_k)^\top c(x_k), -N_X(x_k)) \geq \delta \|c(x_k)\|.
\end{equation}
This assumption, despite being strong, is common in the existing literature of deterministic or stochastic algorithms with nonconvex functional constraints, see e.g.,~\cite{sahin2019inexact,li2021rate,lin2022complexity,li2023stochastic}. Let us note that this is implied by assuming LICQ in the whole space. 
\cite{bolte2018nonconvex} make a similar assumption that they term uniform regularity.~\cite{lin2022complexity} considered the relationship of this assumption with Kurdyka-{\L}ojasiewicz and constraint qualifications.

We state now the main result for this section. 
\begin{theorem}\label{eq: asb4}
Let the assumptions in~\eqref{eq: htr_eqs}, \eqref{eq: asp_gen2}, \eqref{eq: asp_gen3}, \eqref{eq: asp_gen4} hold. Set the parameters of Algorithm~\ref{algorithm:stoc_lalm} as
\[
\eta_k = \frac{1}{9\tilde L \rho(k+1)^{3/5}},
\rho_k = \rho k^{1/5}, \alpha_{k+1} = \frac{72}{81(k+1)^{4/5}},
\]
for some $\rho>1$. Then, there exists $\lambda$ such that
\begin{align*}
    \mathbb{E}\dist(\nabla f(x_{\hat k+1}) + \nabla c(x_{\hat k+1})^\top \lambda, -N_X(x_{\hat k +1})) &\leq \varepsilon,\\
    \mathbb{E}\|c(x_{\hat k + 1})\|&\leq\varepsilon.
\end{align*}
with number of iterations $K$ of Algorithm~\ref{algorithm:stoc_lalm} bounded by $\widetilde{O}(\varepsilon^{-5})$ and $\hat k$ selected uniformly at random from $\{1,\dots, K\}$.
\end{theorem}
\begin{remark}
    This is an iteration complexity result that directly translates to $\widetilde{O}(\varepsilon^{-5})$ sample complexity and arithmetic complexity. This is because each iteration of Algorithm~\ref{algorithm:stoc_lalm} requires one sample of each stochastic function $f, 
    (c_i)_{i=1}^m$ and each iteration only involves simple projections to $X$ and vector operations.
\end{remark}
\begin{remark}
    It is worth noting that it is straightforward to get rid of the logarithmic terms in the above bound by using parameters $\eta_k, \rho_k$ that depend on the final iterate. However, this would require an initial preprocessing stage to get a near-feasible point for getting the best complexity. \cite{shi2022momentum} used this approach to study the deterministic constraints setting.
\end{remark}
\begin{remark}
    We state our results for the constrained case for simplicity. The extension to the \emph{proximal case}, where we have the additional proper convex lower semicontinuous function instead of the constraint $x\in X$,  is straightforward with our analysis template.
\end{remark}
\subsection{Analysis}
As in the previous section, we start with the one iteration analysis of the algorithm for which we suppress some of the universal constants for readability. Statement of the lemma with details appears in Sec. \ref{subsec: one_it_stoch}.
\begin{lemma}\label{eq: one_it_ineq1}
Under the assumptions in~\eqref{eq: htr_eqs}, \eqref{eq: asp_gen2}, \eqref{eq: asp_gen3}, \eqref{eq: asp_gen4} and the parameters (see also \eqref{eq: stoc_lips}, \eqref{eq: asp_gen3}) 
\begin{equation*}
\begin{aligned}
\eta_k &= \frac{1}{9\tilde L \rho(k+1)^{3/5}},~~~
\rho_k = \rho k^{1/5},\\ \alpha_{k+1} &= \frac{72}{81(k+1)^{4/5}},
\end{aligned}
\end{equation*}
for some constant $\rho>1$, we have that
\begin{align*}
 &\frac{\eta_k}{72}\E \dist^2(\nabla f(x_{k+1}) + \rho_{k}\nabla c(x_{k+1})^\top c({x_{k+1}}), - N_X(x_{k+1})) \\
 &\leq  \E[Y_{k}- Y_{k+1} + |Q_{\rho_k}(x_{k+1}) - Q_{\rho_{k+1}}(x_{k+1})|] + \mathcal{E}_{k+1},
\end{align*}
where 
\begin{align*}
Y_{k+1} &= Q_{\rho_{k+1}}(x_{k+1}) + \frac{1}{72\tilde L^2 \rho_{k+1}^2 \eta_k} \| g_{k+1} - \nabla Q_{\rho_{k+1}}(x_{k+1}) \|^2,\\
\mathcal{E}_{k+1} &= O\left(\frac{(\rho_k-\rho_{k+1})^2}{\rho_{k+1}^2\eta_k}\right) + O\left( \frac{\alpha_{k+1}^2}{\eta_k}\right).
\end{align*}
\end{lemma}
\begin{remark}\label{rem: cons1}
The first term of $\mathcal{E}_{k+1}$ has the order $O(k^{-7/5})$ and the second term of $\mathcal{E}_{k+1}$ has the order $O(k^{-1})$, therefore $\sum_{k=1}^K \mathcal{E}_{k+1} = O(\log(K+1))$.
\end{remark}
\begin{proof}[Proof sketch of Theorem~\ref{eq: asb4}]
    In view of Remark \ref{rem: cons1}, it is easy to see that the only remaining piece we need on top of Lemma~\ref{eq: one_it_ineq1} is the control over the penalty parameter changes. For this, we show in Lemma~\ref{lem: sps2} that 
    \begin{equation}\label{eq: soi4}
        \sum_{k=1}^\infty \mathbb{E}|Q_{\rho_k}(x_{k+1}) - Q_{\rho_{k+1}}(x_{k+1})| = O(1).
    \end{equation}
    The main idea in this lemma is to use the estimate of Lemma~\ref{eq: one_it_ineq1} with the uniform upper bound on $\|c(x_k)\|^2$ from~\eqref{eq: asp_gen3} and take advantage of the decay of $|\rho_k - \rho_{k+1}|$ and \eqref{eq: asp_gen4} to obtain \eqref{eq: soi4}.
    Using this estimate in Lemma \ref{eq: one_it_ineq1} gives the result.
\end{proof}

\section{Extensions}\label{sec: main_ext}
In this section, we consider two extensions (with details and proofs given in Sec. \ref{subsec: ext1} and \ref{subsec: ext2}) and show how they follow by minor adjustments on our analysis. 
\subsection{Dual Variable Updates}\label{subsec: main_alm_stoc}
In the context of nonconvex optimization with nonconvex functional constraints and ALM, the standard way of incorporating dual updates is to use small step sizes and large penalty parameters to ensure boundedness of the dual variable, see~\cite{li2021rate,sahin2019inexact,shi2022momentum}. Rapidly increasing, unbounded, penalty parameter is then used to obtain feasibility guarantees.
One exception is~\cite{xie2021complexity} which, unfortunately comes with worse complexity guarantees for first-order stationarity, compared to~\cite{li2021rate,lin2022complexity}.
We considered the quadratic penalty method in the previous section for simplicity, but we show in this section that dual updates can be incorporated as done in~\cite{li2021rate,sahin2019inexact,shi2022momentum}, with small step sizes.
The modification compared to Algorithm~\ref{algorithm:stoc_lalm} consists of changing the definition of $g_{k+1}$ and incorporating a dual update step. In particular, we will change the step for $g_{k+1}$ as
\begin{align*}
    g_{k+1} = \tilde\nabla Q_{\rho_{k+1}}(x_{k+1}, \lambda_{k+1}, B_{k+1}) + (1-\alpha_{k+1})(g_k - \tilde\nabla Q_{\rho_{k}}(x_{k}, \lambda_k, B_{k+1})),
\end{align*}
where
\begin{align*}
\tilde{\nabla}Q_{\rho}(x, \lambda, B) = \tilde{\nabla}f(x, \xi^{0}) + \sum_{i=1}^m \tilde \nabla c_i(x, \zeta^{1}) \lambda_{i}  + \rho \sum_{i=1}^m \tilde \nabla c_{i}(x, \zeta^{1}) \tilde c_i(x, \zeta^{2}).
\end{align*}
We also add the dual update step as
\begin{align}
    \lambda_{k+2, i} = \lambda_{k+1, i} + \gamma_{k+1, i} \tilde c_i(x_{k+1}, \zeta^{2}_{k+1}),\label{eq: alg_vr2_step4}
\end{align}
for all $i\in\{1,\dots,m\}$.
As alluded earlier, dual steps generally require a decaying step size as $\gamma_{k+1}$ (or clipping the contribution of the previous dual parameter by a constant amount as in \cite{lu2022single}) for getting the best-known guarantees. 
\begin{theorem}\label{th: th_ext_1_main}
For the algorithm described in Section~\ref{subsec: main_alm_stoc}, let 
\begin{align}
\eta_k &= \frac{1}{9\tilde L \rho(k+1)^{3/5}}, ~~
\rho_k = \rho k^{1/5},  \notag \\
\gamma_{k,i} &= \frac{\gamma}{k(\log(k+1))^2|\tilde c_{i}(x_{k}, \zeta^{2}_k)|},~~\alpha_{k+1} =  \frac{72}{81(k+1)^{4/5}}\notag
\end{align}
for some constant $\rho>1$, $\gamma>0$. Also let the assumptions in \eqref{eq: htr_eqs}, \eqref{eq: asp_gen2}, \eqref{eq: asp_gen3}, \eqref{eq: asp_gen4} hold. We have that there exists $\lambda$ such that
\begin{equation*}
    \mathbb{E}\dist(\nabla f(x_{\hat k+1}) +  \nabla c(x_{\hat k+1})^\top \lambda, -N_X(x_{\hat k +1})) \leq \varepsilon,
\end{equation*}
with number of iterations bounded by $\widetilde{O}(\varepsilon^{-5})$. Moreover, we also have that $\mathbb{E}\|c(x_{\hat k + 1})\|\leq\varepsilon$.
\end{theorem}
\subsection{Deterministic Functional Constraints}\label{subsec: main_det_nonlinear}
In this section, we consider the case when the constraints are deterministic. In this case, we set the parameters accordingly to get the complexity $\tilde O(\varepsilon^{-4})$.

\begin{theorem}\label{th: th_determ_nonl}
For Algorithm~\ref{algorithm:stoc_lalm}, set
\begin{align}
\eta_k &= \frac{1}{9\tilde L \rho(k+1)^{1/2}}, \quad
\rho_k = \rho k^{1/4}, \notag \\
\alpha_{k+1} &=  \frac{72}{81(k+1)^{1/2}},\notag
\end{align}
for some constant $\rho>1$. Also let the assumptions in \eqref{eq: htr_eqs}, \eqref{eq: asp_gen2}, \eqref{eq: asp_gen3}, \eqref{eq: asp_gen4} hold with a deterministic $c(x)$. We have that there exists $\lambda$ such that
\begin{align*}
    \mathbb{E}\dist(\nabla f(x_{\hat k+1}) +  \nabla c(x_{\hat k+1})^\top \lambda, -N_X(x_{\hat k +1})) &\leq \varepsilon,\\
    \mathbb{E}\|c(x_{\hat k + 1})\|&\leq\varepsilon,
\end{align*}
with number of iterations bounded by $\widetilde{O}(\varepsilon^{-4})$.
\end{theorem}
We note that a similar result with a single-loop algorithm is obtained in~\cite{shi2022momentum} with parameters depending on the last iteration (or equivalently, on the final accuracy). This results requires a pre-processing step to get an almost feasible point to get the complexity ${O}(\varepsilon^{-4})$, which deteriorates to ${O}(\varepsilon^{-5})$ otherwise. Hence, obtaining the more favorable complexity leads to a two-stage approach and also needing to set the final accuracy. Our approach leads to an algorithm that is both single stage and any-time.

\section*{Acknowledgments}
This research was supported in part by the NSF grant 2023239, the NSF grant 2224213, the AFOSR award FA9550-21-1-0084. 
\bibliographystyle{abbrvnat}
\bibliography{single_loop.bib}

\appendix

\section{Preliminaries}
\paragraph{Remarks regarding the stochastic oracle model in~\eqref{eq: oracle}.} One possible way to obtain $\tilde \nabla f$ as characterized in~\eqref{eq: oracle} is to compute $\nabla_x \tilde f(x, \xi)$ where $\tilde f$ is defined in~\eqref{eq:gen}.
For this quantity to satisfy the requirement in \eqref{eq: oracle}, its expectation over $\xi$ must be the full gradient.
However, this would require the 
gradient and expectation operations to be interchangeable, which is not always true, especially in our nonconvex setting.
It does however hold for functions in which the support of $\xi$ is finite (that is, finite-sum functions).
It also holds  under fairly general conditions for smooth $f$.
We make similar assumptions and apply similar conventions for the constraint functions $c$ and $\tilde{c}$ and use the oracle model described in this paragraph. It is worth noting that this oracle access is standard in stochastic optimization with nonconvexity, see e.g.,~\cite{ghadimi2013stochastic,arjevani2022lower,cutkosky2019momentum,lan2020first}.

\paragraph{A preliminary lemma.} The variance reduction technique introduced in~\cite{cutkosky2019momentum} uses a vector $g_k$ at each $k$ as a proxy for a gradient  of an expectation function. The variable
$g_k$ accumulates information from earlier iterations and from many values of the random variables, thus has lower variance than the gradient evaluated at $x_k$ and a single value of the random variable.

We need a lemma to bound the difference between $g_k$ and the corresponding gradient, as it evolves across iterations.
Here we state this lemma in a general form that can be applied in all the problem formulations considered in this paper.
The lemma includes possibly iteration \-dependent functions to accommodate situations in which $g_k$ represents the gradient of the augmented Lagrangian or quadratic penalty function with an iteration-dependent penalty parameter.
This lemma builds on~\cite[Lemma 5, Theorem 2]{cutkosky2019momentum}; see also~\cite{ruszczynski1987linearization,mokhtari2020stochastic,yang2016parallel} for conceptually similar derivations in other settings.

\begin{lemma}\label{lem: storm_estimator_genform}
Let $G_k\colon \mathbb{R}^n \to \mathbb{R}^n$ and $\tilde G_{k}(x, \xi)$ be such that $\mathbb{E}_\xi[\tilde G_k(x_k, \xi)] = G_k(x_k)$, $\mathbb{E}_\xi[\tilde G_{k+1}(x_{k+1}, \xi)] = G_{k+1}(x_{k+1})$. Define $g_{k+1} = \tilde G_{k+1}(x_{k+1}, \xi_{k+1}) + (1-\alpha_{k+1})(g_k - \tilde G_{k}(x_k, \xi_{k+1}))$ for $k\geq0$ and $g_0 \in \mathbb{R}^n$. We then have
\begin{align*}
   \mathbb{E}_k \|g_{k+1} - G_{k+1}(x_{k+1})\|^2 &\leq (1-\alpha_{k+1})^2\|g_k - G_{k}(x_{k})\|^2 \notag\\
    &\quad+ 7\mathbb{E}_k\| \tilde G_{k+1}(x_{k+1}, \xi_{k+1}) - \tilde G_{k+1}(x_k, \xi_{k+1}) \|^2 \notag\\
    &\quad+ 42 \mathbb{E}_k\| \tilde G_{k+1}(x_k, \xi_{k+1}) - \tilde G_{k}(x_k, \xi_{k+1})\|^2 \notag\\
    &\quad+3\alpha_{k+1}^2\mathbb{E}_k \| G_{k}(x_k) - \tilde G_{k}(x_k, \xi_{k+1}) \|^2,
\end{align*}
where $\mathbb{E}_k$ is the expectation conditioned on all the history up to and including $x_{k+1}, \xi_k$.
\end{lemma}
\begin{proof}
We have, by subtracting $G_{k+1}(x_{k+1})$ from both sides of the definition of $g_{k+1}$, that
\[
    g_{k+1} -G_{k+1}(x_{k+1}) = \tilde G_{k+1}(x_{k+1}, \xi_{k+1}) + (1-\alpha_{k+1})(g_k - \tilde G_{k}(x_{k}, \xi_{k+1})) - G_{k+1}(x_{k+1}).
\]
On this identity, we use the simple decomposition
\begin{equation*}
    (1-\alpha_{k+1})(g_k - \tilde G_k(x_k, \xi_{k+1})) = (1-\alpha_{k+1})(g_k - G_k(x_k)) + (1-\alpha_{k+1})(G_k(x_k)-\tilde G_k(x_k, \xi_{k+1})),
\end{equation*}
to obtain
\begin{align*}
    g_{k+1} - G_{k+1}(x_{k+1}) &= (1-\alpha_{k+1})(g_k - G_{k}(x_{k})) + (\tilde G_{k+1}(x_{k+1}, \xi_{k+1}) - G_{k+1}(x_{k+1})) \\
    &\quad+ (1-\alpha_{k+1})(G_{k}(x_k)- \tilde G_{k}(x_k, \xi_{k+1})).
\end{align*}
We next take the squared norm of both sides and expand the right-hand side
\begin{align}
    \|g_{k+1} - G_{k+1}(x_{k+1})\|^2 &= (1-\alpha_{k+1})^2\|g_k - G_{k}(x_{k})\|^2 \notag\\
    &\quad+2(1-\alpha_{k+1})\langle g_k - G_{k}(x_{k}), \tilde G_{k+1}(x_{k+1}, \xi_{k+1}) - G_{k+1}(x_{k+1}) \rangle\notag \\
    &\quad+2(1-\alpha_{k+1})^2 \langle g_k - G_{k}(x_k), G_{k}(x_k)- \tilde G_{k}(x_k, \xi_{k+1}) \rangle \notag \\
    &\quad+ \| \tilde G_{k+1}(x_{k+1}, \xi_{k+1}) - G_{k+1}(x_{k+1}) + (1-\alpha_{k+1})(G_{k}(x_k)- \tilde G_{k}(x_k, \xi_{k+1}))\|^2.\label{eq: jfh3}
\end{align}
We now take conditional expectation of this equality, where $\mathbb{E}_k$ is as defined in the lemma statement:
\begin{equation} \label{eq: sep3}
\begin{aligned}
    & \mathbb{E}_k \|g_{k+1} - G_{k+1}(x_{k+1})\|^2 = (1-\alpha_{k+1})^2   \|g_k - G_{k}(x_{k})\|^2  \\
    & \quad + \mathbb{E}_k\|\tilde G_{k+1}(x_{k+1}, \xi_{k+1}) - G_{k+1}(x_{k+1}) + (1-\alpha_{k+1})(G_{k}(x_k)- \tilde G_{k}(x_k, \xi_{k+1}))\|^2.
\end{aligned}
\end{equation}
This equality is because the inner product terms on \eqref{eq: jfh3} disappear after taking condition expectation: we have $g_k - G_{k}(x_k)$ is deterministic when we condition on $x_{k+1}$ and also $\mathbb{E}_k [\tilde G_{k+1}(x_{k+1}, \xi_{k+1}) - G_{k+1}(x_{k+1})] = 0$ and $\mathbb{E}_k [G_{k}(x_k)- G_{k}(x_k, \xi_{k+1}))] = 0$, by the requirements on $\tilde G_k$ and $\tilde G_{k+1}$ given in the lemma.

For the last term on the right-hand side of \eqref{eq: sep3}, we use Young's inequalities to obtain
\begin{align}
    &\mathbb{E}_k\| \tilde G_{k+1}(x_{k+1}, \xi_{k+1}) - G_{k+1}(x_{k+1}) + (1-\alpha_{k+1})(G_{k}(x_k)- \tilde G_{k}(x_k, \xi_{k+1})) \|^2 \notag \\
    &\leq 3 \mathbb{E}_k\| G_{k+1}(x_{k+1}) - G_{k}(x_k) \|^2 + 3 \mathbb{E}_k\| \tilde G_{k+1}(x_{k+1}, \xi_{k+1}) - \tilde G_{k}(x_k, \xi_{k+1}) \|^2 \notag\\
    &\quad+ 3\alpha_{k+1}^2 \mathbb{E}_k\| G_{k}(x_k) - \tilde G_{k}(x_k, \xi_{k+1}) \|^2 \notag\\
    &\leq 6  \mathbb{E}_k\| \tilde G_{k+1}(x_{k+1}, \xi_{k+1}) - \tilde G_{k}(x_k, \xi_{k+1}) \|^2 + 3\alpha_{k+1}^2 \mathbb{E}_k\| G_{k}(x_k) - \tilde G_{k}(x_k, \xi_{k+1}) \|^2,\label{eq: sep4}
\end{align}
where the final bound joins the first two terms in the previous bound, by Jensen's inequality, since 
\[
\mathbb{E}_k [\tilde G_{k+1}(x_{k+1}, \xi_{k+1}) - \tilde G_{k}(x_k, \xi_{k+1})] = G_{k+1}(x_{k+1}) - G_{k}(x_k).
\]
For the first term on the right-hand side of~\eqref{eq: sep4}, we add and subtract $\tilde G_{k+1}(x_k, \xi_{k+1})$ and use Young's inequality to find that
\begin{align*}
    & \| \tilde G_{k+1}(x_{k+1}, \xi_{k+1}) - \tilde G_{k}(x_k, \xi_{k+1}) \|^2 \\
    & \leq \frac{7}{6}\| \tilde G_{k+1}(x_{k+1}, \xi_{k+1}) - \tilde G_{k+1}(x_k, \xi_{k+1}) \|^2 
    + 7 \| \tilde G_{k+1}(x_k, \xi_{k+1}) - \tilde G_{k}(x_k, \xi_{k+1})\|^2.
\end{align*}
Thus \eqref{eq: sep4} becomes
\begin{align*}
    &\mathbb{E}_k\| \tilde G_{k+1}(x_{k+1}, \xi_{k+1}) - G_{k+1}(x_{k+1}) + (1-\alpha_{k+1})(G_{k}(x_k)- \tilde G_{k}(x_k, \xi_{k+1})) \|^2 \\ 
    &\quad\leq 7\mathbb{E}_k\| \tilde G_{k+1}(x_{k+1}, \xi_{k+1}) - \tilde G_{k+1}(x_k, \xi_{k+1}) \|^2 \\
    &\quad\quad+ 42 \mathbb{E}_k\| \tilde G_{k+1}(x_k, \xi_{k+1}) - \tilde G_{k}(x_k, \xi_{k+1})\|^2 \\
    &\quad\quad+3\alpha_{k+1}^2 \mathbb{E}_k\| G_{k}(x_k) - \tilde G_{k}(x_k, \xi_{k+1}) \|^2.
\end{align*}
By using this inequality to bound the last term on the right-hand side of~\eqref{eq: sep3}, we obtain the result.
\end{proof}

\section{Linear Constraints: Problem \eqref{eq:I} }
\subsection{One-step recursion on augmented Lagrangian}\label{sec: subsec_al_recursion}
We recall the augmented Lagrangian function for \eqref{eq:I} as
\begin{equation}\label{eq: sw3}
\mathcal{L}_\rho(x, \lambda) := f(x) + \langle \lambda, Ax-b\rangle + \frac{\rho}{2} \| Ax - b\|^2,
\end{equation}
for $\rho>0$. 
In this section, we prove the following result concerning the change in $\E \cL_\rho(x_k, \lambda_k)$ over one iteration. 
Note that the expectation $\E$ is with respect to the randomness of $\xi_k$ for $k=1,2,\dotsc$.
Our result makes use of the  Lipschitz constant of $\nabla_x \mathcal{L}_\rho(\cdot, \lambda_k)$, which is 
\begin{equation} \label{eq:Lrho}
    L_\rho := L_f + \rho \| A\|^2.
\end{equation}
In the rest of this section, we make frequent use of the matrix $Q_k$ defined by
\begin{equation} \label{eq:defQ}
    Q_k := \eta_k^{-1} I - \rho  A^\top A.
\end{equation}
\begin{lemma}\label{lem: gjr4}
Let the assumptions in~\eqref{eq: asp_case1} hold. 
Then for any $k\geq 1$, we have that
\begin{equation}\label{eq: sht4}
\begin{aligned}
    \mathbb{E}\Lcal_{\rho}(x_{k+1}, \lambda_{k+1}) & \leq \mathbb{E}\Lcal_{\rho}(x_{k}, \lambda_k) + \left(\frac{L_{\rho}}{2} -\frac{1}{2\eta_{k+1}} \right) \mathbb{E}\| x_{k+1} - x_k\|^2 + \frac{\eta_{k+1}}{2} \mathbb{E}\| g_k - \nabla f(x_k) \|^2 \\
    & \quad + \frac{1}{\rho} \mathbb{E}\| \lambda_{k+1} - \lambda_k \|^2.
    \end{aligned}
\end{equation}
With the definition of $Q_{k+1}$ as in \eqref{eq:defQ}, and assuming that $\eta_{k+1}$ is chosen  to ensure that $Q_{k+1} \succ 0$, we also have
\begin{align*}
    \mathbb{E}\|\lambda_{k+1} - \lambda_k\|^2 &\leq \frac{1}{\delta}\Bigg( 6 L_f^2 \mathbb{E}\| x_{k} - x_{k-1}\|^2 + 6\alpha_{k}^2 V^2 +6\alpha_{k}^2\mathbb{E}\| g_{k-1} - \nabla f(x_{k-1})\|^2 \\
   &\quad+ 3\mathbb{E}\left\| x_{k+1} - 2x_k + x_{k-1} \right\|^2_{Q_{k+1}^\top Q_{k+1}}  +3\left(\eta_k^{-1} - \eta_{k+1}^{-1}\right)^2\mathbb{E}\|x_k - x_{k-1}\|^2\Bigg).
\end{align*}
\end{lemma}

Before proving this result, we state and prove an immediate corollary.
\begin{corollary}\label{cor: gjr5}
   Under the same assumptions as Lemma~\ref{lem: gjr4}, for any positive constants $c_1$ and $c_4$, and with $\eta_k$ defined as in~\eqref{eq: params_sec1_app}  (where this definition depends on $c_1$, $c_5$, and other constants), we have 
\begin{equation*}
\begin{aligned}
    \mathbb{E}\Lcal_{\rho}(x_{k+1}, \lambda_{k+1}) & \leq \mathbb{E}\Lcal_{\rho}(x_{k}, \lambda_k) + \frac{\eta_{k+1}}{2} \mathbb{E}\| g_k - \nabla f(x_k) \|^2 + \frac{6(1+c_1)\alpha_{k}^2}{\rho\delta} \mathbb{E} \| g_{k-1} - \nabla f(x_{k-1})\|^2 \\
    &\quad  + \left(\frac{L_{\rho}}{2} -\frac{1}{2\eta_{k+1}} \right) \mathbb{E}\| x_{k+1} - x_k\|^2 +   \frac{(6+c_4)(1+c_1)L_f^2}{\rho\delta} \mathbb{E}\| x_k - x_{k-1}\|^2 \\
    & \quad + \frac{3(1+c_1)}{\rho\delta} \mathbb{E}\| x_{k+1} - 2x_k + x_{k-1}\|^2_{Q_{k+1}^\top Q_{k+1}} + \frac{6(1+c_1)\alpha_{k}^2 V^2}{\rho\delta} - \frac{c_1}{\rho}\mathbb{E}\| \lambda_{k+1} - \lambda_k\|^2.
    \end{aligned}
\end{equation*} 
\end{corollary}
\begin{proof}[Proof of Corollary~\ref{cor: gjr5}]
The proof is immediate after adding and subtracting $\frac{c_1}{\rho} \mathbb{E}\|\lambda_{k+1} - \lambda_k\|^2$, using the upper bound of $\mathbb{E}\|\lambda_{k+1} - \lambda_k\|^2$ from Lemma~\ref{lem: gjr4} and Fact~\ref{fact: fact1} to bound $3(\eta_k^{-1}-\eta_{k+1}^{-1})^2\leq c_4L_f^2$.
\end{proof}

We now return to the proof of Lemma~\ref{lem: gjr4}.
\begin{proof}[Proof of Lemma~\ref{lem: gjr4}]
By Lipschitz continuity of $\nabla \Lcal_\rho(\cdot, \lambda_k)$, we have
\begin{align}\label{eq: ute4}
    \Lcal_{\rho}(x_{k+1}, \lambda_k) &\leq \Lcal_{\rho}(x_{k}, \lambda_k) + \langle \nabla_x \Lcal_{\rho}(x_{k}, \lambda_k), x_{k+1} - x_k \rangle + \frac{L_{\rho}}{2} \| x_{k+1} - x_k\|^2.
\end{align}
By the definitions of $\nabla_x \Lcal_\rho$ and $x_{k+1}$ in Algorithm \ref{algorithm:lalm_lincons}, we have that
\begin{align*}
\langle \nabla_x\Lcal_{\rho}(x_k, \lambda_k), x_{k+1} - x_k \rangle &= \langle g_k + A^\top \lambda_k + \rho A^\top(Ax_k- b), x_{k+1} - x_k \rangle + \langle \nabla f(x_k) - g_k, x_{k+1} - x_k \rangle \\
&\leq -\frac{1}{\eta_{k+1}} \| x_{k+1} - x_k\|^2 +\frac{\eta_{k+1}}{2}\|\nabla f(x_k) - g_k\|^2 + \frac{1}{2\eta_{k+1}} \| x_{k+1} - x_k\|^2 \\
&= -\frac{1}{2\eta_{k+1}} \| x_{k+1} - x_k\|^2 +\frac{\eta_{k+1}}{2}\|\nabla f(x_k) - g_k\|^2,
\end{align*}
where the last two terms on the second line are from Young's inequality.
By substituting in~\eqref{eq: ute4} and collecting like terms, we obtain
\begin{equation}\label{eq: wee3}
    \Lcal_{\rho}(x_{k+1}, \lambda_k) \leq \Lcal_{\rho}(x_{k}, \lambda_k) + \left(\frac{L_\rho}{2} -\frac{1}{2\eta_{k+1}} \right) \| x_{k+1} - x_k\|^2 + \frac{\eta_{k+1}}{2} \| g_k - \nabla f(x_k) \|^2.
\end{equation}
We also have by the definition of $\Lcal_\rho$ in~\eqref{eq: sw3} and $\lambda_{k+1}$ in Algorithm~\ref{algorithm:lalm_lincons} that
\begin{equation*}
    \Lcal_{\rho}(x_{k+1}, \lambda_{k+1}) - \Lcal_{\rho}(x_{k+1}, \lambda_{k}) = \langle \lambda_{k+1} - \lambda_k, Ax_{k+1} - b \rangle = \frac{1}{\rho}\|\lambda_{k+1} - \lambda_k\|^2.
\end{equation*}
By using this identity in~\eqref{eq: wee3}, we obtain the first result~\eqref{eq: sht4} after taking expectation. 

We now bound $\mathbb{E}\|\lambda_{k+1} - \lambda_k\|^2$ to get the second result. By using the definitions of $x_{k+1}$ and $\lambda_{k+1}$ in Algorithm~\ref{algorithm:lalm_lincons}, we obtain
\begin{align}
    x_{k+1} &= x_k - \eta_{k+1}(g_k + A^\top \lambda_k + \rho A^\top(Ax_k - b)) \notag \\
    &= x_k - \eta_{k+1}(g_k + A^\top \lambda_{k+1} + \rho A^\top A(x_k-x_{k+1}))\label{eq: yui4.5}\\
    \iff \left(\eta_{k+1}^{-1} I-\rho A^\top A\right)(x_{k+1} - x_k) & = -g_k - A^\top \lambda_{k+1}.\label{eq: yui5}
\end{align}
Replacing $k$ by $k-1$ in this identity, we have
\begin{align}
    \left(\eta_k^{-1} I -\rho A^\top A\right)(x_{k} - x_{k-1}) &= -g_{k-1} - A^\top \lambda_{k} \nonumber \\
    \iff \left(\eta_{k+1}^{-1} I -\rho A^\top A\right)(x_{k} - x_{k-1}) +\left(\eta_k^{-1} - \eta_{k+1}^{-1}\right) (x_k - x_{k-1}) &= -g_{k-1} - A^\top \lambda_{k}.\label{eq: yui6}
\end{align}
Subtracting~\eqref{eq: yui6} from~\eqref{eq: yui5} (to use a similar technique to one used  in~\cite{hong2016decomposing,hong2018gradient} with the change of identifying the error coming from iteration-dependent parameters) gives
\begin{align}
&\left(\eta_{k+1}^{-1} I  - \rho A^\top A \right) \left( x_{k+1} - x_k - (x_k - x_{k-1}) \right) - \left(\eta_k^{-1}-\eta_{k+1}^{-1}\right)(x_k - x_{k-1}) \notag \\
&= g_{k-1} - g_k + A^\top(\lambda_k- \lambda_{k+1}). \label{eq: wer2}
    \end{align}
    By rearranging, taking squared norm of each side, and using Young's inequality, we obtain
    \begin{multline}
 \| A^\top (\lambda_k - \lambda_{k+1}) \|^2 \leq 3\| g_k - g_{k-1}\|^2 + 3\left\| \left( \eta_{k+1}^{-1} I  - \rho A^\top A \right)(x_{k+1} - x_k - (x_k - x_{k-1})) \right\|^2 \\
    + 3\left(\eta_k^{-1} - \eta_{k+1}^{-1}\right)^2\|x_k - x_{k-1}\|^2. \label{eq: wer3}
\end{multline}
For the first term in this bound, by the definition of $g_{k+1}$ from Algorithm~\ref{algorithm:lalm_lincons}, we have that
\begin{equation*}
    g_{k+1} - g_{k} = \tilde \nabla f(x_{k+1}, \xi_{k+1}) - \tilde\nabla f(x_k, \xi_{k+1}) + \alpha_{k+1}(\tilde\nabla f(x_k, \xi_{k+1}) - g_k).
\end{equation*}
By taking squared norms, using Young's inequality and Lipschitzness of $\tilde \nabla f(\cdot, \xi)$ from \eqref{eq: asp_case1}, we obtain
\begin{align}
    \mathbb{E}\|g_{k+1} - g_{k}\|^2 &\leq 2 L_f^2 \mathbb{E}\| x_{k+1} - x_k\|^2 + 2\alpha_{k+1}^2 \mathbb{E}\| \tilde\nabla f(x_k,\xi_{k+1}) - g_k\|^2 \notag \\
    &= 2L_f^2\mathbb{E} \| x_{k+1} - x_k\|^2 + 2\alpha_{k+1}^2  \mathbb{E}\| \tilde\nabla f(x_k, \xi_{k+1}) - \nabla f(x_k) \|^2 + 2\alpha_{k+1}^2 \mathbb{E}\|g_k- \nabla f(x_k) \|^2 \notag \\
    &\leq 2L_f^2\mathbb{E} \| x_{k+1} - x_k\|^2 + 2\alpha_{k+1}^2 V^2 + 2\alpha_{k+1}^2 \mathbb{E}\|g_k- \nabla f(x_k) \|^2,
    \label{eq: hjy3}
\end{align}
where $V$ is the bound on variance from \eqref{eq: asp_case1}.
To obtain the equality in the derivation above, we used the tower rule: since $\mathbb{E}_{\xi_{k+1}} [\tilde\nabla f(x_k,\xi_{k+1})] = \nabla f(x_k)$ and $g_k, \nabla f(x_k)$ are deterministic when we are taking the conditional expectation, thus $\mathbb{E}\langle \tilde \nabla f(x_k, \xi_{k+1}) - \nabla f(x_k), g_k - \nabla f(x_k) \rangle = 0$.

Using the inequality \eqref{eq: hjy3} in~\eqref{eq: wer3} (with index $k-1$ instead of $k$), we have after taking expectation that 
\begin{align}
   \mathbb{E} \| A^\top (\lambda_k - \lambda_{k+1}) \|^2 &\leq 6 L_f ^2    \mathbb{E}\| x_{k} - x_{k-1}\|^2 + 6\alpha_{k}^2 V^2 + 6\alpha_{k}^2    \mathbb{E}\| g_{k-1} - \nabla f(x_{k-1})\|^2\notag  \\
   &\quad + 3\mathbb{E}\left\| \left( \eta_{k+1}^{-1} I - \rho A^\top A \right)(x_{k+1} - x_k - (x_k - x_{k-1})) \right\|^2\notag \\
   &\quad +3\left(\eta_k^{-1} - \eta_{k+1}^{-1}\right)^2\mathbb{E}\|x_k - x_{k-1}\|^2.\label{eq: hjy4}
   \end{align}
   Since $\lambda_{k+1}-\lambda_k$ is in the range of $A$ and $\delta$ is the smallest nonzero eigenvalue of $A^\top A$ (see also \eqref{eq:def.delta}), we have the result after using the definition  of $Q_{k+1}$ from \eqref{eq:defQ}.
\end{proof}

\subsection{One-step recursion on feasibility and iterate difference}
\begin{lemma}\label{lem: gjr5}
Let the assumptions in~\eqref{eq: asp_case1} hold. 
For all $k \ge 0$,  assume that $\eta_{k+1}$ is chosen  to ensure that $Q_{k+1} \succ 0$ (where $Q_{k+1}$ is defined in \eqref{eq:defQ}).
Then for any $k\geq 1$, we have that
\begin{align*}
& \frac{\rho}{2\eta_{k+2}} \mathbb{E}\| A x_{k+1} - b \|^2 + \frac{1}{2\eta_{k+1}} \mathbb{E}\|x_{k+1} - x_k\|_{Q_{k+1}}^2 \\
&\leq \frac{\rho}{2\eta_{k+1}} \mathbb{E}\| A x_{k} - b \|^2 + \frac{1}{2\eta_k} \mathbb{E}\|x_{k} - x_{k-1}\|^2_{Q_k}\\
&\quad - \frac{1}{2\eta_{k+1}} \mathbb{E}\| x_{k+1} - 2x_k+x_{k-1} \|^2_{Q_{k+1}} +\frac{\eta_{k+2}^{-1}-\eta_{k+1}^{-1}}{2\rho} \mathbb{E}\| \lambda_{k+1} - \lambda_k\|^2\\
&\quad + \left(\frac{L_f}{2\eta_{k+1}} +\frac{\eta_{k+1}^{-1} - \eta_k^{-1}}{2\eta_{k+1}}\right) \mathbb{E}\| x_{k+1} - x_k\|^2   + \left( \frac{L_f}{\eta_{k+1}} + \frac{\eta_{k+1}^{-1} - \eta_k^{-1}}{2\eta_{k+1}} + \frac{\eta_{k+1}^{-2} - \eta_k^{-2}}{2} \right) \mathbb{E}  \| x_k - x_{k-1}\|^2 \\
&\quad +\frac{\alpha_k^2}{L_f\eta_{k+1}}\mathbb{E}\| g_{k-1} - \nabla f(x_{k-1})\|^2  + \frac{\alpha_k^2 V^2}{L_f\eta_{k+1}}.
\end{align*}
\end{lemma}

Before proving this result, we state and prove an immediate corollary.
\begin{corollary}\label{cor: wr4}
    Suppose the assumptions of Lemma~\ref{lem: gjr5} hold and that $c_1$, $c_2$, $c_3$ are arbitrary positive constants, with $\eta_k$ and $m>0$ defined from
    \eqref{eq: params_sec1_app} (where the definitions of $\eta_k$ and $m$ depend on $c_1$, $c_2$, $c_3$, and other constants). We have 
    \begin{align*}
    & \frac{\rho}{2\eta_{k+2}} \mathbb{E}\| A x_{k+1} - b \|^2 + \frac{1}{2\eta_{k+1}} \mathbb{E}\|x_{k+1} - x_k\|_{Q_{k+1}}^2 \\
    &\leq \frac{\rho}{2\eta_{k+1}} \mathbb{E}\| A x_{k} - b \|^2 + \frac{1}{2\eta_k} \mathbb{E}\|x_{k} - x_{k-1}\|^2_{Q_k}  \\
    &\quad - \frac{1}{2\eta_{k+1}} \mathbb{E}\| x_{k+1} - 2x_k+x_{k-1} \|^2_{Q_{k+1}} +\frac{c_1}{\rho m} \mathbb{E}\| \lambda_{k+1} - \lambda_k\|^2\\
    &\quad + \frac{(1+2c_3)L_f}{2\eta_{k+1}}  \mathbb{E}\| x_{k+1} - x_k\|^2 + \frac{(1+c_2+c_3)L_f}{\eta_{k+1}} \mathbb{E}\|x_k-x_{k-1}\|^2  \\
    &\quad  + \frac{\alpha_k^2}{L_f\eta_{k+1}} \mathbb{E}\| g_{k-1} - \nabla f(x_{k-1})\|^2+\frac{\alpha_{k}^2 V^2}{L_f\eta_{k+1}}.
\end{align*}
\end{corollary}
\begin{proof}[Proof of Corollary~\ref{cor: wr4}]
The result is immediate after using \eqref{eq: fact_ineq1},~\eqref{eq: fact_ineq2} and~\eqref{eq: fact_ineq3}  in Fact~\ref{fact: fact1} on the result of Lemma~\ref{lem: gjr5} and rearranging terms.
\end{proof}

We now return to Lemma~\ref{lem: gjr5}
\begin{proof}[Proof of Lemma~\ref{lem: gjr5}]
For this recursion, we use a similar technique to~\cite{hong2016decomposing,hong2018gradient}, extended to use variable step sizes. The additional error terms will appear due to using stochastic gradients and variance reduction in our case, which were not considered in~\cite{hong2016decomposing,hong2018gradient}. We start from~\eqref{eq: wer2} and use the definition of $\lambda_{k+1}$ in Algorithm~\ref{algorithm:lalm_lincons} (that is, $\lambda_{k+1} - \lambda_k = \rho (Ax_{k+1} - b)$) and the definition \eqref{eq:defQ} of $Q_{k+1}$  to obtain
\begin{equation*}
g_k - g_{k-1} + \rho A^\top(Ax_{k+1} - b) + Q_{k+1} (x_{k+1} - 2x_k + x_{k-1})+ (\eta_{k+1}^{-1} - \eta_k^{-1})(x_k - x_{k-1}) = 0.
\end{equation*}
We next take the inner product of this expression with $x_{k+1} - x_k$, divide by $\eta_{k+1}$ and rearrange to get
\begin{align}
    \frac{\rho}{\eta_{k+1}} \langle Ax_{k+1} - b, A(x_{k+1} - x_k)  \rangle &= \frac{1}{\eta_{k+1}}\langle g_{k-1} - g_k - Q_{k+1}(x_{k+1} - 2x_{k} + x_{k-1}), x_{k+1} - x_k \rangle \notag \\ 
    &\quad+ \frac{\eta_{k+1}^{-1}-\eta_k^{-1}}{\eta_{k+1}}\langle x_{k-1} - x_{k}, x_{k+1} - x_k \rangle.\label{eq: htu5}
\end{align}
We now obtain bounds on four parts of this expression in turn.
\begin{enumerate}
    \item By using the identity $2\langle a, b \rangle = \|a\|^2 + \| b\|^2 - \|a-b\|^2$ for any vectors $a$ and $b$, we have that
\begin{align}
    &\frac{\rho}{\eta_{k+1}} \langle Ax_{k+1} - b, A(x_{k+1} - x_k)  \rangle\notag \\
    &= \frac{\rho}{2\eta_{k+1}}\left( \| Ax_{k+1} - b\|^2 + \| A(x_{k+1} - x_k) \|^2 - \|Ax_k - b\|^2 \right) \notag \\
    &\geq\frac{\rho}{2\eta_{k+1}}\left( \| Ax_{k+1} - b\|^2 - \|Ax_k - b\|^2 \right) \notag \\
    &= \frac{\rho}{2\eta_{k+2}} \| Ax_{k+1} - b\|^2 - \frac{\rho}{2\eta_{k+1}} \|Ax_k - b\|^2 + \frac{\eta_{k+1}^{-1} - \eta_{k+2}^{-1}}{2\rho}\|\lambda_{k+1} - \lambda_k\|^2,\label{eq: htu6}
\end{align}
where the last step used the definition of $\lambda_{k+1}$ from Algorithm~\ref{algorithm:lalm_lincons} again, together with addition and subtraction of a term involving $\eta_{k+2}^{-1}$.
\item By using the same identity  and the definition \eqref{eq:defQ} of $Q_{k+1} \succ 0$, we also have
\begin{align}
&-\frac{1}{\eta_{k+1}}\langle Q_{k+1}(x_{k+1} - 2x_{k} + x_{k-1}), x_{k+1} -x_k \rangle\notag \\
&= -\frac{1}{\eta_{k+1}}\langle x_{k+1} - 2x_{k} + x_{k-1}, x_{k+1} -x_k \rangle_{Q_{k+1}}\notag \\
&=-\frac{1}{2\eta_{k+1}}\bigg( \|x_{k+1} - 2x_k + x_{k-1}\|_{Q_{k+1}}^2 + \|x_{k+1} - x_k\|_{Q_{k+1}}^2 - \| x_k - x_{k-1}\|_{Q_{k+1}}^2 \bigg)\notag \\
&\leq -\frac{1}{2\eta_{k+1}}\|x_{k+1} - 2x_k + x_{k-1}\|_{Q_{k+1}}^2 - \frac{1}{2\eta_{k+1}}\|x_{k+1} - x_k\|_{Q_{k+1}}^2 + \frac{1}{2\eta_{k}}\| x_k - x_{k-1}\|_{Q_{k}}^2 \notag \\
&\quad+ \frac{1}{2}\left(\eta_{k+1}^{-2} - \eta_k^{-2}\right)\| x_k - x_{k-1}\|^2,\label{eq: htu7}
\end{align}
where the last step uses \eqref{eq:defQ} together with $0< \eta_{k+1} \leq \eta_k$, i.e., 
\begin{align*}
    &\frac{1}{2\eta_{k+1}} \| x_k - x_{k-1} \|^2_{Q_{k+1}} - \frac{1}{2\eta_{k}} \| x_k - x_{k-1} \|^2_{Q_{k}} \\
    &= \frac{1}{2}\left(\eta_{k+1}^{-2} - \eta_k^{-2}\right)\| x_k - x_{k-1}\|^2 + \langle x_k - x_{k-1}, \left( -\eta_{k+1}^{-1} + \eta_k^{-1} \right) \rho A^\top A(x_k - x_{k-1})\rangle\\
    & \leq \frac{1}{2}\left(\eta_{k+1}^{-2} - \eta_k^{-2}\right)\| x_k - x_{k-1}\|^2.
\end{align*}
\item By Young's inequality, we have
\begin{align}
\frac{\eta_{k+1}^{-1} - \eta_k^{-1}}{\eta_{k+1}}\langle x_{k-1} - x_k, x_{k+1} - x_k \rangle &\leq \frac{\eta_{k+1}^{-1} - \eta_k^{-1}}{2\eta_{k+1}} \left(  \| x_k - x_{k-1}\|^2 + \| x_{k+1} - x_k\|^2\right).\label{eq: htu9}
\end{align}
\item By Young's inequality and~\eqref{eq: hjy3} (replacing $k$ by $k-1$), we get
\begin{align}
    &\frac{1}{\eta_{k+1}}\mathbb{E}\langle g_{k-1} - g_k, x_{k+1} - x_k \rangle \notag \\
    &\leq \frac{1}{2L_f\eta_{k+1}} \mathbb{E}\| g_k - g_{k-1} \|^2 + \frac{L_f}{2\eta_{k+1}} \mathbb{E}\| x_{k+1} - x_k\|^2\notag\\
    &\leq \frac{1}{2L_f\eta_{k+1}} \left( 2L_f^2\mathbb{E} \| x_{k} - x_{k-1}\|^2 + 2\alpha_{k}^2 V^2 + 2\alpha_{k}^2 \mathbb{E}\|g_{k-1}- \nabla f(x_{k-1}) \|^2 \right) \notag \\&\quad+ \frac{L_f}{2\eta_{k+1}} \mathbb{E}\| x_{k+1} - x_k\|^2.\label{eq: htu8}
\end{align}
\end{enumerate}
We get the result by taking expectation of~\eqref{eq: htu5} and substituting from \eqref{eq: htu6},~\eqref{eq: htu7},~\eqref{eq: htu9}, and~\eqref{eq: htu8}.
\end{proof}

\subsection{Main result}\label{app: lincons_main}
In this section, for arbitrary positive values of $c_1,c_2,c_3,c_4$ we define the parameters $\eta_k$, $m$, $\rho$, and $\alpha_k$ in the following way (recall the definitions of $L_f$ and $\delta$ from \eqref{eq: asp_case1} and \eqref{eq:def.delta}): 
\begin{equation}\label{eq: params_sec1_app}
\begin{aligned}
c& =121L_f^2,  \\
m &= \min\left( \frac{1}{448L_f}, \frac{1}{32(1+c_2+c_3)L_f}, \frac{1}{8(1+2c_3)L_f} \right)\\
\rho &= \max\left( \frac{7(1+c_1)}{m\delta}, \frac{4(6+c_4)(1+c_1)L_f}{\delta},\frac{168(1+c_1)L_f}{\delta} \right),\\
\eta &= \frac{1}{11(L_f+\rho\|A\|^2)}, \\
k_0 &= \max\left( \left(\frac{10m}{3c_1\eta}\right)^2, \left(\frac{20}{3\eta c_2L_f}\right)^2 , \left(\frac{10}{3c_3L_f}\right)^2, \frac{400}{3\eta^2 c_4 L_f^2}, \left(\frac{20}{c\eta^2}\right)^4, \left(\frac{50}{3c\eta^2} \right)^6, 2 \right) \\
\eta_k &= \frac{\eta}{(k+k_0)^{1/3}\log(k+k_0)} \\    
\alpha_k &= c\eta_k^2. 
\end{aligned}
\end{equation}
We should remark that the constants of these parameters and also constants in other bounds in the paper are not optimized, since we focus on the dependence on $\varepsilon$ in this work. 
The constants certainly could  be improved at the expense of even great complexity in the analysis.
\begin{replemma}{lem: skl4}
Let the assumptions in~\eqref{eq: asp_case1} hold and $\eta_k$ be set as~\eqref{eq: params_sec1_app}. For the iterates of Algorithm~\ref{algorithm:lalm_lincons}, we have for $k \geq 1$ that
\begin{align}
\mathbb{E}Y_{k+1} &\leq \mathbb{E}Y_k  - \frac{\eta_{k+1}}{2} \mathbb{E} \| g_k - \nabla f(x_k)\|^2  + (\beta_{1, k+1} + \beta_{2, k+1}) \mathbb{E} \| x_{k+1} - x_k\|^2 + w_k + v_k V^2,\label{eq: gjt7}
\end{align}
where
    \begin{equation}
\begin{aligned}
Y_{k+1} &= L_{\rho}(x_{k+1}, \lambda_{k+1}) + \frac{\rho m}{2\eta_{k+2}} \| Ax_{k+1} - b\|^2 + \frac{m}{2\eta_{k+1}} \| x_{k+1} - x_k\|^2_{Q_{k+1}} + \beta_{1, k+1} \| x_{k+1} - x_k\|^2 \\
    &\quad+ \frac{2}{c\eta_{k+1}} \| g_{k+1} - \nabla f(x_{k+1})\|^2 + \left(\frac{6(1+c_1)}{\rho\delta} + \frac{4m}{L_f\eta_k} \right)  \| g_{k} - \nabla f(x_{k})\|^2\label{eq: gjf4}
\end{aligned}
\end{equation}
and
\begin{align*}
    \beta_{1, k} &= \frac{(1+c_2+c_3)L_f m}{\eta_{k+1}} + \frac{(6+c_4)(1+c_1)L_f^2}{\rho\delta} + \frac{42(1+c_1)L_f^2}{\rho\delta} + \frac{28mL_f}{\eta_k},\\
    \beta_{2, k} &= \frac{L_\rho}{2} + \frac{(1+2c_3)mL_f}{2\eta_{k}} + \frac{14L_f^2}{c\eta_{k}} - \frac{1}{2\eta_{k}}, \\
    w_k &= \frac{3(1+c_1)}{\delta\rho} \mathbb{E}\| x_{k+1} - 2x_k + x_{k-1}\|^2_{Q_{k+1}^\top Q_{k+1}} - \frac{m}{2\eta_{k+1}}\mathbb{E} \| x_{k+1} - 2x_k + x_{k-1} \|^2_{Q_{k+1}},\\
    v_k &= \frac{6(1+c_1)\alpha_{k}^2}{\delta\rho} + \frac{\alpha_k^2 m}{L_f\eta_{k+1}} + \frac{6\alpha_{k+1}^2}{c\eta_{k+1}} + \frac{18(1+c_1)\alpha_k^2}{\rho\delta} + \frac{12m\alpha_k^2}{L_f\eta_k}.
\end{align*}
\end{replemma}
\begin{proof}
    We combine the bounds of Corollaries~\ref{cor: gjr5} and \ref{cor: wr4} after multiplying the latter by the scalar $m$ of \eqref{eq: params_sec1} (by noting cancellation of the terms with $\|\lambda_{k+1}-\lambda_k\|^2$ and after adding and subtracting $\frac{\eta_{k+1}}{2} \mathbb{E} \| g_k - \nabla f(x_k)\|^2$):
\begin{align}
    &\mathbb{E}L_\rho(x_{k+1}, \lambda_{k+1}) + \frac{\rho m}{2\eta_{k+2}} \mathbb{E}\| Ax_{k+1} - b\|^2 + \frac{m}{2\eta_{k+1}} \mathbb{E}\| x_{k+1} - x_k\|^2_{Q_{k+1}} \notag \\
    &
    \leq \mathbb{E}L_{\rho}(x_k, \lambda_k) + \frac{\rho m}{2\eta_{k+1}}\mathbb{E} \| Ax_{k} - b\|^2 + \frac{m}{2\eta_k} \mathbb{E}\| x_{k} - x_{k-1}\|^2_{Q_k}-\frac{\eta_{k+1}}{2}  \mathbb{E}\| g_k - \nabla f(x_k) \|^2\notag \\
    &\quad + \left( \frac{L_{\rho}}{2} + \frac{(1+2c_3)mL_f}{2\eta_{k+1}} - \frac{1}{2\eta_{k+1}} \right) \mathbb{E}\| x_{k+1} - x_k\|^2 \notag \\&\quad+ \left( \frac{(1+c_2+c_3)L_f m}{\eta_{k+1}} + \frac{(6+c_4)(1+c_1)L_f^2}{\rho\delta}\right) \mathbb{E}\| x_k - x_{k-1} \|^2\notag \\
    &\quad +  \eta_{k+1} \mathbb{E}\| g_k - \nabla f(x_k) \|^2 + \left(\frac{m\alpha_k^2}{L_f\eta_{k+1}}+\frac{6(1+c_1)\alpha_{k}^2}{\rho\delta}\right)\mathbb{E} \| g_{k-1} - \nabla f(x_{k-1}) \|^2 \label{eq: seg4}\\
    &\quad +\frac{3(1+c_1)}{\rho\delta} \mathbb{E}\| x_{k+1} - 2x_k + x_{k-1}\|^2_{Q_{k+1}^\top Q_{k+1}} - \frac{m}{2\eta_{k+1}}\mathbb{E} \| x_{k+1} - 2x_k + x_{k-1} \|^2_{Q_{k+1}}\notag \\
    &\quad + \left( \frac{6(1+c_1)\alpha_{k}^2}{\rho\delta} + \frac{\alpha_k^2 m}{L_f\eta_{k+1}} \right) V^2.\label{eq: ghj3}
\end{align}
Since this expression looks rather daunting, we take a closer look at the different terms. 
On a high level, we see that the terms in the first and second lines telescope. 
For the terms in the third and fourth lines, and the terms in the sixth line, we show that our  parameter choices make them nonpositive or telescoping. 
The size of the terms in the last line can be controlled  by the choice of $\alpha_k$. The details will be spelled out in Theorem~\ref{th: main1}. What remains is to handle the fifth line by variance reduction recursion from Lemma~\ref{lem: var_bd}.

We now  use the definitions of $\beta_{1,k}$, $\beta_{2,k}$, and $v_k$ in the statement of the lemma. 
These terms  include not only the coefficients of $\mathbb{E} \| x_k - x_{k-1}\|^2$, $\mathbb{E}\|x_{k+1} -x_k\|^2$, and $V^2$ but also the contribution for the corresponding terms that will come when we use Lemma~\ref{lem: var_bd} to bound the terms on~\eqref{eq: seg4}, \emph{cf.} \eqref{eq: lso3}. 
In particular, with Lemma~\ref{lem: var_bd} to bound the terms on~\eqref{eq: seg4}, and the definitions of $\beta_{1, k}, \beta_{2, k+1}, v_k, w_k$ in the statement of the lemma, we have 
\begin{align}
    &\mathbb{E}L_{\rho}(x_{k+1}, \lambda_{k+1}) + \frac{\rho m}{2\eta_{k+2}} \mathbb{E}\| Ax_{k+1} - b\|^2 + \frac{m}{2\eta_{k+1}} \mathbb{E}\| x_{k+1} - x_k\|^2_{Q_{k+1}}  \notag \\
    &\quad + \frac{2}{c\eta_{k+1}}  \mathbb{E} \| g_{k+1} - \nabla f(x_{k+1})\|^2 + \left(\frac{6(1+c_1)}{\rho\delta} + \frac{4m}{L_f\eta_{k}}\right)\mathbb{E} \| g_{k} - \nabla f(x_{k})\|^2 \notag \\
    &\leq \mathbb{E}L_{\rho}(x_k, \lambda_k) + \frac{\rho m}{2\eta_{k+1}} \mathbb{E}\| Ax_{k} - b\|^2 + \frac{m}{2} \mathbb{E}\| x_{k} - x_{k-1}\|^2_{Q_k}-\frac{\eta_{k+1}}{2} \mathbb{E}\| g_k - \nabla f(x_k)\|^2 \notag \\
    &\quad+ \frac{2}{c\eta_{k}}  \mathbb{E} \| g_{k} - \nabla f(x_{k})\|^2 + \left(\frac{6(1+c_1)}{\rho\delta} + \frac{4m}{L_f\eta_{k-1}}\right)\mathbb{E} \| g_{k-1} - \nabla f(x_{k-1})\|^2 \notag \\
    &\quad + \beta_{2, k+1} \mathbb{E}\| x_{k+1} - x_k\|^2 + \beta_{1, k} \mathbb{E}\| x_k - x_{k-1} \|^2 + w_k  + v_k V^2.\notag
\end{align}
By adding $\beta_{1, k+1} \mathbb{E}\|x_{k+1} - x_k\|^2$ to both sides and using the definition of $Y_k$, we have the result.
\end{proof}

We continue with the restatement and proof of Theorem~\ref{th: main1}.

\begin{reptheorem}{th: main1}
Let the assumptions in~\eqref{eq: asp_case1} hold and supppose that $\eta_k$ and the other algorithmic parameters are chosen as in~\eqref{eq: params_sec1_app}.
Then, for the iterates of Algorithm~\ref{algorithm:lalm_lincons}, we have for any $K\geq 1$ that 
\begin{align*}
\frac{1}{K}\sum_{k=1}^{K+1} \mathbb{E} \left[ \| \eta_{k}^{-1}(x_{k} - x_{k-1})\|^2 + \| g_{k-1} - \nabla f(x_{k-1})\|^2 \right]=O\left(\frac{1}{\eta_{K+1} K} \right) = \tilde O(K^{-2/3}).
\end{align*}
\end{reptheorem}

\begin{proof}
We start with the result of Lemma~\ref{lem: skl4}.
To show that the terms involving $w_k$, $\beta_{1,k+1}$, and $\beta_{2,k+1}$ in the right-hand side of~\eqref{eq: gjt7} sum up to a nonpositive constant, we will show that 
\begin{subequations}\label{eq: hee4}
\begin{align}
    &Q_{k+1} = \eta_{k+1}^{-1}I-\rho A^\top A \succ 0,\label{eq: hee4_2}\\
    & \frac{m}{2\eta_{k+1}} Q_{k+1} - \frac{3(1+c_1)}{\rho\delta}Q_{k+1}^\top Q_{k+1} \succeq 0,\label{eq: hee4_1}\\
    &\beta_{1, k+1}+\beta_{2,k+1}=\frac{L_{\rho}}{2} + \frac{(1+2c_3)mL_f}{2\eta_{k+1}} + \frac{14L_f^2}{c\eta_{k+1}} - \frac{1}{2\eta_{k+1}} \notag \\
    &\quad+ \frac{(1+c_2+c_3)L_fm}{\eta_{k+2}} + \frac{(6+c_4)(1+c_1)L_f^2}{\rho\delta} + \frac{42(1+c_1)L_f^2}{\rho\delta} + \frac{28mL_f}{\eta_{k+1}} \leq -\frac{1}{16\eta_{k+1}},\label{eq: hee4_3}
\end{align}
\end{subequations}
where the definitions of $\beta_{1, k}, \beta_{2, k}$ are given in Lemma~\ref{lem: skl4}.

First, we know that~\eqref{eq: hee4_2} is satisfied since $\eta_{k+1} < \eta \le 1/(\rho \|A\|^2)$.
The positive semidefiniteness condition \eqref{eq: hee4_1} is satisfied when $m \geq \frac{7(1+c_1)}{\rho\delta}$ (which is ensured by the definition of $\rho$) since $\frac{m}{2} \geq \frac{3(1+c_1)}{\rho\delta}$ and since
\begin{equation*}
Q_{k+1} = \frac{1}{\eta_{k+1}}I - \rho A^\top A \succ 0  \implies \frac{1}{\eta_{k+1}}Q_{k+1} - Q_{k+1}^\top Q_{k+1}\succeq 0,
\end{equation*}
as can be verified by direct substitution of $Q_{k+1}$. 
We thus have from the definition of $w_k$ in Lemma~\ref{lem: skl4}  that  $w_k \leq 0$.

Finally, we verify~\eqref{eq: hee4_3}. 
With $m, \rho, c$ defined in \eqref{eq: params_sec1}, and using $\eta_{k+2}^{-1}\leq 2\eta_{k+1}^{-1}$, see for example Fact~\ref{fact: fact1}, we have that 
\begin{align*}
    -\frac{1}{2\eta_{k+1}} + \frac{(1+2c_3)mL_f}{2\eta_{k+1}}+ \frac{14L_f^2}{c\eta_{k+1}} + \frac{(1+c_2+c_3)L_f m}{\eta_{k+2}} + \frac{28mL_f}{\eta_{k+1}} &\leq -\frac{3}{16\eta_{k+1}},\\
    \frac{L_\rho}{2} + \frac{(6+c_4)(1+c_1)L_f^2}{\rho\delta} + \frac{42(1+c_1)L_f^2}{\rho\delta} &\leq \frac{L_\rho}{2} + \frac{L_f}{2} \le L_{\rho},
\end{align*}
where the final inequality follows from \eqref{eq:Lrho}. Since the left-hand sides of these two inequalitites sum to $\beta_{1,k+1}+\beta_{2,k+1}$, \eqref{eq: hee4_3} holds if we can show that 
\begin{equation*}
    L_\rho \leq \frac{1}{8\eta_{k+1}} \iff \eta_{k+1} \leq \frac{1}{8L_\rho},
\end{equation*}
which is implied by $\eta \leq \frac{1}{11L_\rho}$ in \eqref{eq: params_sec1}.

As shown above, we have $w_k \le 0$ with the selected parameters. By substituting this bound together with  \eqref{eq: hee4_3}  into~\eqref{eq: gjt7}, we obtain
\begin{align}
\mathbb{E}Y_{k+1} &\leq \mathbb{E}Y_k  -\frac{1}{16\eta_{k+1}}\mathbb{E} \| x_{k+1} - x_k\|^2 - \frac{\eta_{k+1}}{2} \mathbb{E} \| g_k - \nabla f(x_k)\|^2  + v_k V^2,\label{eq: gjt8}
\end{align}
By adding to both sides $\frac{1}{32\eta_k}\mathbb{E}\| x_k - x_{k-1}\|^2 + \frac{\eta_{k}}{4} \mathbb{E} \| g_{k-1} - \nabla f(x_{k-1})\|^2$ and rearranging, we get
\begin{align}
 &\frac{1}{32\eta_{k+1}}\mathbb{E} \| x_{k+1} - x_k\|^2 + \frac{1}{32\eta_k} \mathbb{E} \| x_{k} - x_{k-1}\|^2 + \frac{\eta_{k+1}}{4} \mathbb{E} \| g_k - \nabla f(x_k)\|^2 + \frac{\eta_{k}}{4} \mathbb{E} \| g_{k-1} - \nabla f(x_{k-1})\|^2 \notag  \\
 &\leq \Big( \mathbb{E}Y_k + \frac{1}{32\eta_{k}}\| x_{k} - x_{k-1}\|^2 + \frac{\eta_{k}}{4}\| g_{k-1} - \nabla f(x_{k-1})\|^2 \Big) \notag \\
 &\quad- \Big( \mathbb{E}Y_{k+1} + \frac{1}{32\eta_{k+1}}\| x_{k+1} - x_k\|^2 + \frac{\eta_{k+1}}{4}\| g_k - \nabla f(x_k)\|^2 \Big)    + v_k V^2,\label{eq: gjt9}
\end{align}
From the definition $\alpha_k=c\eta_k^2$ with $\eta_k = \frac{\eta}{(k+k_0)^{1/3}\log(k+k_0)}$ in  \eqref{eq: params_sec1_app}, we have $\sum_{k=1}^\infty {\alpha_{k+1}^2}/{\eta_{k+1}} = O(1)$ and $\sum_{k=1}^{\infty} \alpha_k^2 = O(1)$. 
It follows from the definition of $v_k$ in Lemma~\ref{lem: skl4} that $\sum_{k=1}^{\infty} v_k = O(1)$.
Thus by summing the inequality \eqref{eq: gjt9} over $k=1,2,\dotsc,K$ and telescoping the right-hand side, we obtain
\begin{align}
&\sum_{k=1}^{K} \frac{1}{32\eta_{k+1}}\mathbb{E} \| x_{k+1} - x_k\|^2 + \frac{1}{32\eta_k}\mathbb{E} \| x_{k} - x_{k-1}\|^2 + \frac{\eta_{k+1}}{4} \mathbb{E} \| g_k - \nabla f(x_k)\|^2 + \frac{\eta_{k}}{4} \mathbb{E} \| g_{k-1} - \nabla f(x_{k-1})\|^2 \notag \\
&\leq \Big( \mathbb{E}Y_1 + \frac{1}{32\eta_1}\| x_1-x_0\|^2 + \frac{\eta_1}{4}\|g_0 - \nabla f(x_0)\|^2 \Big) \notag \\
&\quad-\Big(\mathbb{E} Y_{K+1} +\frac{1}{32\eta_{K+1}}\|x_{K+1} - x_K\|^2 + \frac{\eta_K}{4} \| g_K - \nabla f(x_K)\|^2 \Big)+ O(1) \notag \\
&\leq \Big( \mathbb{E}Y_1 + \frac{1}{32\eta_1}\| x_1-x_0\|^2 + \frac{\eta_1}{4}\|g_0 - \nabla f(x_0)\|^2 \Big) -\mathbb{E} Y_{K+1} +O(1). \label{eq: okj4}
\end{align}
For the terms on the left-hand side, we have for $k=1,2,\dotsc,K$, using $\eta_k \geq \eta_{k+1} \geq \eta_{K+1}$, that
\begin{align*}
    & \frac{1}{32\eta_{k+1}}\mathbb{E} \| x_{k+1} - x_k\|^2 + \frac{1}{32\eta_k}\mathbb{E} \| x_{k} - x_{k-1}\|^2 \\
    & =\frac{\eta_{k+1}}{32} \mathbb{E}\| \eta_{k+1}^{-1}(x_{k+1} - x_k) \|^2 + \frac{\eta_k}{32}\mathbb{E} \|\eta_k^{-1} (x_{k} - x_{k-1})\|^2 \\
    & \ge \frac{\eta_{K+1}}{32}   \left[ \mathbb{E}\| \eta_{k+1}^{-1}(x_{k+1} - x_k) \|^2 + \mathbb{E} \|\eta_k^{-1} (x_{k} - x_{k-1})\|^2 \right],
\end{align*}
and similarly
\[
\frac{\eta_{k+1}}{4} \mathbb{E} \| g_k - \nabla f(x_k)\|^2 + \frac{\eta_{k}}{4} \mathbb{E} \| g_{k-1} - \nabla f(x_{k-1})\|^2
\ge
\frac{\eta_{K+1}}{4} \left[ \mathbb{E} \| g_k - \nabla f(x_k)\|^2 +  \mathbb{E} \| g_{k-1} - \nabla f(x_{k-1})\|^2\right].
\]
Thus, by adding successive terms on the left-hand side of \eqref{eq: okj4}, and using these bounds, we have
\begin{align*}
&\sum_{k=1}^{K} \frac{1}{32\eta_{k+1}}\mathbb{E} \| x_{k+1} - x_k\|^2 + \frac{1}{32\eta_k}\mathbb{E} \| x_{k} - x_{k-1}\|^2 + \frac{\eta_{k+1}}{4} \mathbb{E} \| g_k - \nabla f(x_k)\|^2 + \frac{\eta_{k}}{4} \mathbb{E} \| g_{k-1} - \nabla f(x_{k-1})\|^2  \\
&\ge \eta_{K+1} \sum_{k=1}^{K+1}\left( \frac{1}{32} \mathbb{E} \| \eta_k^{-1}(x_{k} - x_{k-1}) \|^2 + \frac{1}{4} \mathbb{E} \| g_{k-1} - \nabla f(x_{k-1})\|^2\right).
\end{align*}
By substituting this lower bound into \eqref{eq: okj4}, joining the constant coefficients in the left-hand side to the $O(1)$ term on the right-hand side and using  $-Y_{K+1} = O(1)$ (by lower boundedness of the potential function to be shown below in Lemma~\ref{lem: potential_lb}), we have the result after dividing both sides by $\eta_{K+1}K$.
\end{proof}

\begin{corollary}\label{cor: sor4}
    Let the assumptions in~\eqref{eq: asp_case1} hold. With the parameter choices in~\eqref{eq: params_sec1_app}, Algorithm~\ref{algorithm:lalm_lincons} outputs $(\bar x, \bar\lambda)$ such that
\begin{equation*}
\mathbb{E}\| \nabla f(\bar x) + A^\top \bar \lambda \| \leq \varepsilon, \text{~and~} \mathbb{E}\| A\bar x - b\| \leq \varepsilon,
\end{equation*}
within $K = \tilde O(\varepsilon^{-3})$ iterations.
\end{corollary}
\begin{proof}
A useful preliminary result is as follows.
For $\hat k$ selected uniformly at random from $\{1,2,\dots, K\}$, we have
\begin{align}\label{eq: ehr}
&\mathbb{E}\left[ \| \eta_{\hat k+1}^{-1}(x_{\hat k+1} - x_{\hat k}) \|^2+\| \eta_{\hat k}^{-1}(x_{\hat k} - x_{\hat k-1}) \|^2 + \| g_{\hat k} - \nabla f(x_{\hat k}) \|^2 + \| g_{\hat k-1} - \nabla f(x_{\hat k-1}) \|^2 \right] \notag \\
&=\frac{1}{K}\sum_{k=1}^K \mathbb{E}\left[ \| \eta_{k+1}^{-1}(x_{k+1} - x_k) \|^2+\| \eta_{k}^{-1}(x_{k} - x_{k-1}) \|^2 + \| g_k - \nabla f(x_k) \|^2 + \| g_{k-1} - \nabla f(x_{k-1}) \|^2 \right] \notag \\
& \le \frac{2}{K} \sum_{k=1}^{K+1} \mathbb{E} \left[ \| \eta_{k}^{-1}(x_{k} - x_{k-1})\|^2 + \| g_{k-1} - \nabla f(x_{k-1})\|^2 \right] \notag \\
&= \tilde O(K^{-2/3}),
\end{align}
where the last step used the result of Theorem~\ref{th: main1}.
As a result, the right-hand side of \eqref{eq: ehr} is guaranteed to be less that $\varepsilon^2$ 
within $K = \tilde{O}(\varepsilon^{-3})$ iterations. 

To show the stationarity condition, we have by the definition of $x_{k+1}, \lambda_{k+1}$ in Algorithm~\ref{algorithm:lalm_lincons} and the triangle inequality that (see also \eqref{eq: yui4.5}) 
\begin{align*}
&& x_{k+1} &= x_k - \eta_{k+1} (g_k + A^\top \lambda_{k+1} + \rho A^\top A(x_k - x_{k+1})) \\
& \iff & x_{k+1} + \eta_{k+1}(g_k - \nabla f(x_k)) &= x_k - \eta_{k+1} (\nabla f(x_k) + A^\top \lambda_{k+1} + \rho A^\top A(x_k - x_{k+1})) \\
& \iff & \nabla f(x_{k+1}) + A^\top \lambda_{k+1} &= \eta_{k+1}^{-1}(x_k - x_{k+1}) + \rho A^\top A(x_{k+1} - x_k) \\
&&& \quad - (g_k - \nabla f(x_k)) + \nabla f(x_{k+1}) - \nabla f(x_k) \\
& \implies  & \| \nabla f(x_{k+1}) + A^\top \lambda_{k+1} \| &\leq  \| \eta_{k+1}^{-1}(x_k - x_{k+1}) \| +  \| g_k - \nabla f(x_k) \| + (\rho \| A\|^2+L_f) \| x_k -x_{k+1}\|,
\end{align*}
where, after taking expectation and using $k=\hat k$, all the terms on the right-hand side are smaller than $\varepsilon$ in expectation after $K = \tilde{O}(\varepsilon^{-3})$ iterations, due to~\eqref{eq: ehr} and Jensen's inequality.
Setting $\bar{x} = x_{\hat{k}+1}$ and $\bar\lambda = \lambda_{\hat{k}+1}$, we thus have $\E (\| \nabla f(\bar{x}) + A^\top \bar\lambda \|) \le \varepsilon$.

It remains to show that $\| A x_{\hat{k}+1} - b \| \le \varepsilon$ for this same value of $\hat{k}$.
For any $k$, since $\lambda_{k+1} - \lambda_k = \rho (Ax_{k+1}- b)$, $\lambda_{k+1} - \lambda_k$ is in the range of $A$.
Since $\delta$ is the smallest nonzero eigenvalue of $A^\top A$, we have (see also \eqref{eq:def.delta})
\begin{equation} \label{eq:eh1}
    \|A^\top(\lambda_{k+1} - \lambda_k) \|^2 \geq \delta\| \lambda_{k+1} - \lambda_k\|^2 = \rho^2 \delta \| Ax_{k+1} - b\|^2.
\end{equation}
By \eqref{eq: wer3}, together with $\eta_{k+1}^{-1} I \succ \eta_{k+1}^{-1} I -\rho A^\top A \succ 0 $, $0<\eta_{k+1} \leq\eta_k$, and Young's inequality, we have that
\begin{align*}
 \| A^\top (\lambda_k - \lambda_{k+1}) \|^2 
 &\leq 3 \| g_k - g_{k-1} \|^2 +  \frac{6}{\eta_{k+1}^2}\left(\|x_{k+1} - x_k\|^2 + \| x_k - x_{k-1}\|^2 \right) +  \frac{3}{\eta_{k+1}^2}\|x_k - x_{k-1}\|^2 \\
 & \leq 9\| g_k - \nabla f(x_k)\|^2 + 9 \| g_{k-1} - \nabla f(x_{k-1})\|^2 + 9\| \nabla f(x_k) - \nabla f(x_{k-1})\|^2 \\
 &\quad+ \frac{6}{\eta_{k+1}^2}\left(\|x_{k+1} - x_k\|^2 + \| x_k - x_{k-1}\|^2 \right)   + \frac{3}{\eta_{k+1}^2}\|x_k - x_{k-1}\|^2.
\end{align*}
As shown in Fact~\ref{fact: fact1} we have $\eta_k \leq 2\eta_{k+1}$ and consequently $\frac{1}{\eta_{k+1}} \leq \frac{2}{\eta_k}$. 
Combining these with Lipschitzness of $\nabla f$ and $L_f^2 \leq \eta_k^{-2}$ in the last inequality, we obtain
    \begin{align}
 \| A^\top(\lambda_{k+1} - \lambda_k) \|^2 &= O\Big(\| g_k - \nabla f(x_k)\|^2 + \| g_{k-1} -\nabla f(x_{k-1})\|^2 \notag \\
 &\quad + \| \eta_{k+1}^{-1} (x_{k+1} - x_k)\|^2 + \|\eta_{k}^{-1} (x_k - x_{k-1}) \|^2\Big).\label{eq: sry3}
\end{align}
Particularly, by \eqref{eq: ehr}, the last estimate implies
\[
 \mathbb{E}\| A^\top(\lambda_{\hat k+1} - \lambda_{\hat k}) \|^2=\tilde O(K^{-2/3}).
\]
In view of \eqref{eq:eh1}, this gives
\[
\mathbb{E}\| Ax_{\hat{k}+1}- b\|^2 \leq \frac{1}{\rho^2 \delta} \E \| A^\top (\lambda_{\hat{k}+1} - \lambda_{\hat{k}}) \| = \tilde O(K^{-2/3}).
\]
As a result, for $K = \tilde{O}(\varepsilon^{-3})$, by Jensen's inequality we have that $\E \| A \bar{x}-b \| \leq \varepsilon$ for $\bar{x} = x_{\hat{k}+1}$, as required.
\end{proof}

\subsection{Auxiliary results used in the analysis}
The following lemma is about control of the variance of the estimator, specializing the preliminary Lemma~\ref{lem: storm_estimator_genform} for the current purpose.
\begin{lemma}\label{lem: var_bd}
    Let the assumptions in~\eqref{eq: asp_case1} hold and $g_k$ be defined as Algorithm~\ref{algorithm:lalm_lincons}. For $c$ such that $\alpha_k=c\eta_k^2\leq 1$, $\eta_k$ as in~\eqref{eq: params_sec1_app} and any $m, c_1, \rho, \delta$, we have
    \begin{equation}\label{eq: lso3}
    \begin{aligned}
        &\eta_{k+1}  \mathbb{E}\| g_{k} - \nabla f(x_{k}) \|^2 + \left(\frac{m\alpha_k^2}{L_f\eta_{k+1}}+\frac{6(1+c_1)\alpha_{k}^2}{\delta\rho}\right)\mathbb{E} \| g_{k-1} - \nabla f(x_{k-1}) \|^2 \\
        &\leq \frac{2}{c\eta_k}\mathbb{E}\|g_k - \nabla f(x_k)\|^2 + \left( \frac{6(1+c_1)}{\rho\delta} + \frac{4m}{L_f\eta_{k-1}} \right) \| g_{k-1} - \nabla f(x_{k-1})\|^2 \\
        &\quad - \frac{2}{c\eta_{k+1}}\mathbb{E}\|g_{k+1} - \nabla f(x_{k+1})\|^2 - \left( \frac{6(1+c_1)}{\rho\delta} + \frac{4m}{L_f\eta_{k}} \right) \mathbb{E}\| g_{k} - \nabla f(x_{k})\|^2 \\
        &\quad +\frac{14L_f^2}{c\eta_{k+1}} \mathbb{E} \| x_{k+1} - x_k\|^2 + \left( \frac{42(1+c_1)L_f^2}{\rho\delta} + \frac{28mL_f}{\eta_k} \right) \mathbb{E}\| x_k - x_{k-1}\|^2 \\
        &\quad + \left( \frac{6\alpha_{k+1}^2}{c\eta_{k+1}}+\frac{18(1+c_1)\alpha_k^2}{\rho\delta}  + \frac{12m\alpha_k^2}{L_f\eta_k} \right) V^2.
    \end{aligned}
    \end{equation}
\end{lemma}
\begin{proof}
We first recall the result from Lemma~\ref{lem: storm_estimator_genform}.
By substituting $(\tilde G_{k}(x,\xi), G_k(x)) = (\tilde \nabla f(x, \xi), \nabla f(x))$ for all $k$, using Lipschitzness of $\tilde \nabla f(x, \xi)$ and the variance bound in \eqref{eq: asp_case1}, and taking total expectation, we have
\begin{equation}\label{eq: trw3}
\mathbb{E} \| g_{k+1} - \nabla f(x_{k+1})\|^2 \leq (1-\alpha_{k+1})^2\mathbb{E} \| g_{k} - \nabla f(x_{k})\|^2 + 7 L_f^2 \mathbb{E} \| x_{k+1} - x_k\|^2 + 3\alpha_{k+1}^2 V^2.
\end{equation}
Since $\alpha_{k+1} \leq 1$ by the assumption of the lemma, we have that $(1-\alpha_{k+1})^2 \leq 1-\alpha_{k+1}$ and therefore
\begin{equation}\label{eq: gfd4}
\alpha_{k+1} \mathbb{E} \| g_k - \nabla f(x_k)\|^2 \leq \mathbb{E} \| g_k - \nabla f(x_k)\|^2 - \mathbb{E} \| g_{k+1} - \nabla f(x_{k+1})\|^2 + 7 L_f^2 \mathbb{E}\| x_{k+1} - x_k\|^2 + 3\alpha_{k+1}^2 V^2.
\end{equation}
After multiplying the last inequality (replacing $k$ by $k-1$) with $\frac{6(1+c_1)}{\rho\delta}$ and using $\alpha_k^2 \leq \alpha_k$, we obtain
\begin{align}
\frac{6(1+c_1)\alpha_{k}^2}{\rho\delta} \mathbb{E} \| g_{k-1} - \nabla f(x_{k-1})\|^2 &\leq \frac{6(1+c_1)}{\rho\delta}\left(\mathbb{E} \| g_{k-1} - \nabla f(x_{k-1})\|^2 - \mathbb{E} \| g_{k} - \nabla f(x_{k})\|^2\right) \notag \\
&\quad+ \frac{42(1+c_1) L_f^2}{\rho\delta} \mathbb{E}\| x_{k} - x_{k-1}\|^2 + \frac{18(1+c_1)\alpha_{k}^2 V^2}{\rho\delta}.\label{eq: sde4}
\end{align}
By using $\alpha_k = c\eta_k^2$, we have from~\eqref{eq: gfd4} that 
\begin{align}
\eta_{k+1} \mathbb{E} \| g_k - \nabla f(x_k)\|^2 &\leq \frac{1}{c\eta_{k+1}}\bigg(\mathbb{E} \| g_k - \nabla f(x_k)\|^2 - \mathbb{E} \| g_{k+1} - \nabla f(x_{k+1})\|^2 \notag \\
&\quad+ 7 L_f^2 \mathbb{E}\| x_{k+1} - x_k\|^2 + 3\alpha_{k+1}^2 V^2\bigg).\label{eq: spo3}
\end{align}
Formula \eqref{eq: fact_ineq5} from  Fact~\ref{fact: fact1} sat that $\frac{1}{\eta_{k+1}} - \frac{1}{\eta_k} \leq \frac{\eta_{k+1}c}{2}$, so we have 
\begin{align*}
    \frac{1}{c\eta_{k+1}}  \| g_{k} -\nabla f(x_{k}\|^2 &= \frac{1}{c\eta_{k}}  \| g_{k} -\nabla f(x_{k}\|^2 + \left( \frac{1}{c\eta_{k+1}} -\frac{1}{c\eta_k}\right) \| g_{k} -\nabla f(x_{k}\|^2\\
    &\leq \frac{1}{c\eta_{k}}  \| g_{k} -\nabla f(x_{k}\|^2 + \frac{\eta_{k+1}}{2} \| g_{k} -\nabla f(x_{k}\|^2,
\end{align*}
which by substituting in~\eqref{eq: spo3} yields
\begin{align*}
    \eta_{k+1} \mathbb{E} \| g_k - \nabla f(x_k)\|^2 &\leq \frac{1}{c\eta_{k}} \mathbb{E} \| g_k - \nabla f(x_k)\|^2 + \frac{\eta_{k+1}}{2} \| g_k - \nabla f(x_k)\|^2 - \frac{1}{c\eta_{k+1}}\mathbb{E} \| g_{k+1} - \nabla f(x_{k+1})\|^2 \notag \\
&\quad+ \frac{7L_f^2}{c\eta_{k+1}} L_f^2 \mathbb{E}\| x_{k+1} - x_k\|^2 + \frac{3\alpha_{k+1}^2 V^2}{c\eta_{k+1}}.
\end{align*}
By moving the second term on the right-hand side to the left, then multiplying both sides by $2$, we obtain
\begin{align}
& \eta_{k+1} \mathbb{E} \| g_k - \nabla f(x_k)\|^2\notag \\
& \leq \frac{2}{c}\left(\frac{1}{\eta_{k}} \mathbb{E} \| g_k - \nabla f(x_k)\|^2 - \frac{1}{\eta_{k+1}}\mathbb{E} \| g_{k+1} - \nabla f(x_{k+1})\|^2\right) + \frac{14L_f^2}{c \eta_{k+1}}\mathbb{E}\| x_{k+1} - x_k\|^2 + \frac{6\alpha_{k+1}^2 V^2}{c \eta_{k+1}}.\label{eq: sde5}
\end{align}
By replacing $k$ by $k-1$ and multiplying both sides by $\frac{2mc}{L_f}$, we get
\begin{align*}
& \frac{2mc\eta_{k}}{L_f} \mathbb{E} \| g_{k-1} - \nabla f(x_{k-1})\|^2 \notag \\
& \leq \frac{4m}{L_f} \left( \frac{1}{\eta_{k-1}}\mathbb{E}\|g_{k-1} - \nabla f(x_{k-1})\|^2 - \frac{1}{\eta_k} \mathbb{E}\| g_k - \nabla f(x_k) \|^2 \right) + \frac{28mL_f}{\eta_{k}} \mathbb{E}\| x_k - x_{k-1} \|^2 + \frac{12m\alpha_k^2 V^2}{L_f\eta_{k}}.
\end{align*}
By using $\alpha_{k} \leq 1$ (thus $\alpha_k^2 \le \alpha_k$), $\alpha_k = c \eta_k^2$, and $\eta_k \le 2 \eta_{k+1}$ from \eqref{eq: fact_ineq6} in Fact~\ref{fact: fact1}, we have 
\[
\frac{m\alpha_k^2}{L_f\eta_{k+1}} \leq \frac{m\alpha_k}{L_f\eta_{k+1}} = \frac{mc\eta_k^2}{L_f\eta_{k+1}} \leq \frac{2mc \eta_k}{L_f},
\]
and so by replacing the left-hand side in the previous inequality, we obtain
\begin{align*}
    & \frac{m \alpha_k^2}{L_f \eta_{k+1}} \mathbb{E}\| g_{k-1} - \nabla f(x_{k-1}) \|^2 \notag \\
    &\leq \frac{4m}{L_f} \left( \frac{1}{\eta_{k-1}}\mathbb{E}\|g_{k-1} - \nabla f(x_{k-1})\|^2 - \frac{1}{\eta_k} \mathbb{E}\| g_k - \nabla f(x_k) \|^2 \right)+ \frac{28mL_f}{\eta_{k}} \mathbb{E}\| x_k - x_{k-1} \|^2 + \frac{12m\alpha_k^2 V^2}{L_f\eta_{k}}.
\end{align*}
The result follows by combining this bound with~\eqref{eq: sde4} and~\eqref{eq: sde5}.
\end{proof}

We use the following fact to simplify the coefficients appearing in the analysis (before and after this point). 
We do not try to optimize the constants in this result or other parts of the paper since our focus is on the $\varepsilon$-dependence in our bounds. 
(The constants in this lemma are certainly improvable, since, for simplicity, we make use of  loose inequalities to compare terms of the order $\log (x+1)$ and $x^{\alpha}$ for $\alpha > 1/3$.)
\begin{fact}\label{fact: fact1}
    Given $\eta_k = \frac{\eta}{(k+k_0)^{1/3}\log(k+k_0)}$,  it holds that
    \begin{subequations}\label{eq: fact_ineq}
    \begin{align}
        \frac{1}{2\rho}\left( \frac{1}{\eta_{k+2}} - \frac{1}{\eta_{k+1}} \right) & \leq \frac{c_1}{\rho m},\label{eq: fact_ineq1}\\
        \frac{\eta_{k+1}^{-2}-\eta_k^{-2}}{2} & \leq \frac{c_2 L_f}{\eta_{k+1}},\label{eq: fact_ineq2}\\
        \frac{ \eta_{k+1}^{-1}-\eta_k^{-1}}{2\eta_{k+1}} & \leq \frac{c_3 L_f}{\eta_{k+1}},\label{eq: fact_ineq3}\\
        3(\eta_{k}^{-1} - \eta_{k+1}^{-1})^2 & \leq c_4 L_f^2\label{eq: fact_ineq4}\\
        \frac{1}{\eta_{k+1}} - \frac{1}{\eta_k} &\leq \frac{\eta_{k+1} c}{2}\label{eq: fact_ineq5} \\ 
        \eta_{k+1} & \leq 2\eta_{k+2}, \label{eq: fact_ineq6}
    \end{align}
    \end{subequations}
     where $k_0 = \max\left( \left(\frac{10m}{3c_1\eta}\right)^2, \left(\frac{20}{3\eta c_2L_f}\right)^2 , \left(\frac{10}{3c_3L_f}\right)^2, \frac{400}{3\eta^2 c_4 L_f^2}, \left(\frac{20}{c\eta^2}\right)^4, \left(\frac{50}{3c\eta^2} \right)^6, 2 \right)$, for any absolute constants $c_1, c_2, c_3, c_4$ and any positive values of  $m, \eta, c$.
\end{fact}
\begin{proof}
    For~\eqref{eq: fact_ineq1}, by straightforward computation, we have 
    \begin{align}
        \frac{1}{\eta_{k+2}} - \frac{1}{\eta_{k+1}} &= \frac{(k+k_0+2)^{1/3}\log(k+k_0+2) - (k+k_0+1)^{1/3}\log(k+k_0+1)}{\eta} \notag \\
        &= \frac{\log(k+k_0+2)((k+k_0+2)^{1/3} - (k+k_0+1)^{1/3})}{\eta} \notag \\
        &\quad + \frac{(k+k_0+1)^{1/3}(\log(k+k_0+2)-\log(k+k_0+1)}{\eta}.\label{eq: nhy3}
    \end{align}
    Here, we have 
    \[
    \log(k+k_0+2)-\log(k+k_0+1) = \log\left( 1+\frac{1}{k+k_0+1}\right)\leq \frac{1}{k+k_0+1}
    \]
    and 
    \[
    (k+k_0+2)^{1/3} - (k+k_0+1)^{1/3} = \frac{1}{(k+k_0+2)^{2/3} + (k+k_0+2)(k+k_0+1)+(k+k_0+1)^{2/3}} \leq \frac{1}{3(k+k_0+1)^{2/3}}.
    \]
    With these, \eqref{eq: nhy3} yields
    \[
    \frac{1}{\eta_{k+2}} -\frac{1}{\eta_{k+1}} \leq \frac{\log(k+k_0+2)}{3(k+k_0+1)^{2/3}\eta} + \frac{1}{\eta (k+k_0+1)^{2/3}} \leq \frac{4\log(k+k_0+2)}{3(k+k_0+1)^{2/3}\eta}
    \]
    and hence the desired inequality is implied by $\frac{4\log(k+k_0+2)}{3(k+k_0+1)^{2/3}\eta} \leq \frac{2c_1}{m}$. By using $\log(k+k_0+2) \leq 5(k+k_0+1)^{1/6}$ for $k\geq 1$, the inequality is implied by $\frac{10m}{3c_1\eta} \leq (k+k_0+1)^{1/2}$ which is implied by $k_0 \geq \left(\frac{10m}{3c_1\eta}\right)^2$.

    For~\eqref{eq: fact_ineq2}, using the first bound in the previous paragraph and $\frac{1}{\eta_k} \leq \frac{1}{\eta_{k+1}}$, we have 
    \[
    \frac{1}{\eta_{k+1}^2} - \frac{1}{\eta_k^2} \leq \frac{2}{\eta_{k+1}} \left(\frac{1}{\eta_{k+1}} - \frac{1}{\eta_k}\right)  \leq \frac{4\log(k+k_0+1)}{3\eta(k+k_0)^{2/3}}\frac{2}{\eta_{k+1}},
    \]
    and the inequality we want to prove is implied by $\frac{4\log(k+k_0+1)}{3\eta(k+k_0)^{2/3}} \leq c_2 L_f$. 
    By using $\log(k+k_0+1)\leq 5(k+k_0)^{1/6}$, the assertion is implied by $\frac{20}{3\eta c_2L_f} \leq (k+k_0)^{1/2}$ which is implied by $k_0 \geq \left(\frac{20}{3\eta c_2L_f}\right)^2$.

    It is straightforward to show~\eqref{eq: fact_ineq3},~\eqref{eq: fact_ineq4} by using the same estimates, which we omit for brevity. 

    For~\eqref{eq: fact_ineq5}, as in the beginning of the proof, we have 
    \[
    \frac{1}{\eta_{k+1}} - \frac{1}{\eta_k} \leq \frac{1}{\eta(k+k_0)^{2/3}}\left( 1+\frac{\log(k+k_0+1)}{3}\right),
    \]
    and the desired inequality  is implied by
        \[
        \frac{1}{\eta(k+k_0)^{2/3}}\left( 1+\frac{\log(k+k_0+1)}{3}\right)\leq \frac{c \eta}{2(k+k_0+1)^{1/3}\log(k+k_0+1)},
        \]
    which in turn is implied by 
    \[
    2\log(k+k_0+1) + \frac{2(\log(k+k_0+1))^2}{3} \leq c\eta^2(k+k_0)^{1/3}.
    \]
    By using $\log(k+k_0+1)\leq 5(k+k_0)^{1/12}$, this inequality is implied by 
    \[
    10(k+k_0)^{1/12} + \frac{25(k+k_0)^{1/6}}{3} \leq c\eta^2(k+k_0)^{1/3},
    \]
    which in turn is implied by 
    \[
    10(k+k_0)^{1/12} \leq \frac{c\eta^2(k+k_0)^{1/3}}{2} \;\; \mbox{and} \;\; \frac{25(k+k_0)^{1/6}}{3} \leq  \frac{c\eta^2(k+k_0)^{1/3}}{2}.
    \]
    These bounds hold when $k_0 \geq \max\left(\left(\frac{20}{c\eta^2}\right)^4, \left(\frac{50}{3c\eta^2} \right)^6 \right) = \left(\frac{50}{3c\eta^2}\right)^6$.

    The last assertion $\eta_{k+1} \leq2\eta_{k+2}$ \eqref{eq: fact_ineq6} is implied by $(k+k_0+2)^{1/3} \log(k+k_0+2) \leq 2(k+k_0+1)^{1/3}\log(k+k_0+1)$ which is implied by $(k+k_0+2)^{1/3} \leq \frac{4}{3}(k+k_0+1)^{1/3}$ which holds for any $k_0$ and by $\log(k+k_0+2) \leq \frac{3}{2}\log(k+k_0+1)$ which holds, for example, when $k_0 \geq 2$.
\end{proof}

The following lemma is showing the lower boundedness of the potential function $Y_k$ defined in Lemma~\ref{lem: skl4}, which is important for ensuring that the complexity has the desired dependence on $\varepsilon$.
\begin{lemma}\label{lem: potential_lb}
Under the assumptions and the parameter choices of Theorem~\ref{th: main1} and the definition of $Y_k$ in~\eqref{eq: gjf4}, there exists $\underline y > -\infty$ such that $\mathbb{E}Y_k \geq \underline y$ for any $k$. In particular $\underline y = -2V^2\sum_{k=1}^\infty v_k > -\infty$ since 
\[
v_k = \frac{6(1+c_1)\alpha_{k}^2}{\delta\rho} + \frac{\alpha_k^2 m}{L_f\eta_{k+1}} + \frac{6\alpha_{k+1}^2}{c\eta_{k+1}} + \frac{18(1+c_1)\alpha_k^2}{\rho\delta} + \frac{12m\alpha_k^2}{L_f\eta_k}=O\left( \frac{1}{(k+k_0)(\log(k+k_0))^3}\right),
\]
due to
\[
\eta_k = \frac{\eta}{(k+k_0)^{1/3}\log(k+k_0)}, \quad 
\alpha_k = c\eta_k^2,
\]
for the  constants $\eta$ and $c$ from \eqref{eq: params_sec1_app}.
\end{lemma}
\begin{proof}
    Note that all the terms are nonnegative in the definition of $Y_k$ in Lemma \ref{lem: skl4}) with the exception of $L_\rho(x_k, \lambda_k)$. We therefore focus on the latter term.
    Our  argument extends that of~\cite[Lemma 3.5]{hong2016decomposing}, the important difference being that, due to the stochasticity of our setting, we do not have non-increasing potential, which is critical in the argument. 
    We show that our time-varying choices of step sizes still allow us to establish lower boundedness. 

    First, by using $f(x) \geq \underline f \geq 0$ and the definition of $\lambda_{k+1}$, which implies
    \[
    \langle \lambda_{k+1}, Ax_{k+1} - b\rangle = \rho^{-1}\langle \lambda_{k+1}, \lambda_{k+1} - \lambda_k \rangle = \frac{1}{2\rho}\left( \| \lambda_{k+1}\|^2 - \| \lambda_k\|^2 + \|\lambda_{k+1} -\lambda_k\|^2 \right),
    \]
 we can show that the following holds for any $K$:
    \begin{align*}
        \sum_{k=0}^K \mathbb{E} L_{\rho}(x_{k+1}, \lambda_{k+1}) &= \sum_{k=0}^K \mathbb{E} \left[f(x_{k+1}) + \langle \lambda_{k+1}, Ax_{k+1} - b \rangle + \frac{\rho}{2} \| Ax_{k+1} - b\|^2 \right] \\
        &\geq \frac{1}{\rho}\left( \mathbb{E}\| \lambda_{K+1} \|^2 - \| \lambda_0\|^2 \right) \geq -\frac{1}{\rho}\|\lambda_0\|^2.
    \end{align*}
    It follows that
    \begin{equation}\label{eq: y1_def}
        \sum_{k=1}^\infty \mathbb{E} L_\rho(x_k, \lambda_k) > -\infty \quad\Longrightarrow\quad \sum_{k=1}^\infty \mathbb{E} Y_k \geq \underline y_1 > -\infty,
    \end{equation}
    for some $\underline y_1$.
    
    Next, we use the bound \eqref{eq: gjt8} which states that 
    \begin{equation}\label{eq: sob4}
        \mathbb{E}Y_{k+1} \leq \mathbb{E}Y_k + v_k V^2,
    \end{equation}
    where  $v_k$ from  Lemma \ref{lem: skl4} is redefined in the statement of this lemma. 
    For the $O()$ estimate $v_k = O\left( \frac{1}{(k+k_0)(\log(k+k_0))^3} \right)$, we have used
$\eta_k = \frac{\eta}{(k+k_0)^{1/3}\log(k+k_0)}$,
$\alpha_k = c\eta_k^2$ for constants $\eta, c$ given in \eqref{eq: params_sec1_app}, together with
    with $k_0 \geq 2$. 
        We define $C:=V^2 \sum_{k=1}^\infty v_k$ which is finite due to the definition of $v_k$ with $k_0 \geq 2$.
    
    We consider three cases.
    \begin{enumerate}
        \item When $\mathbb{E}Y_k\geq 0$ for all $k$,  the assertion follows immediately.
        \item When  $\mathbb{E} Y_{k_2} < 0$ for some $k_2$ and $\mathbb{E} Y_{k} \geq -2C$ for all $k \ge k_2$, the assertion also follows.
        \item When there exists an index $k_1$ such that $\mathbb{E} Y_{k_1} < -2C$, assume without loss of generality that $k_1$ is the smallest index that satisfies this property. 
    By the definition of this case, $\mathbb{E}Y_k \geq -2C$ for $k<k_1$.
    Since $V^2 \sum_{k=1}^\infty v_k = C$, we have from \eqref{eq: sob4} that $\mathbb{E} Y_{k} < -C$ for all $k \ge k_1$. 
    However, this would cause a contradiction with \eqref{eq: y1_def}.
    \end{enumerate}
     This concludes the proof.
\end{proof}

\section{Stochastic Constraints: Problem~\eqref{eq:III}}\label{sec: qp_stoc_const_app}

\subsection{Variance control}
We start with a result for the variance of the estimator $g_k$. This result is essentially a corollary of Lemma~\ref{lem: storm_estimator_genform}. We characterize the precise constants for the bound of this lemma which are important for getting the order of complexity. For ease of reference, let us recall here the relevant quantities from \eqref{eq: htr_eqs}, \eqref{eq: asp_gen2}, \eqref{eq: asp_gen3}:
\begin{align*}
    \mathbb{E} \| \tilde \nabla f(x, \xi) - \tilde\nabla f(y, \xi) \|^2 &\leq \tilde L_{\nabla f}^2 \| x-y\|^2, \\ 
    \mathbb{E} \| \tilde \nabla c_i(x, \zeta) - \tilde\nabla c_i(y, \zeta) \|^2 &\leq \tilde L_{\nabla c}^2 \| x-y\|^2, \\
    \mathbb{E} | \tilde c_i(x, \zeta) - \tilde c_i(y, \zeta) \|^2 &\leq \tilde L_{c}^2 \| x-y\|^2,\\
    \mathbb{E}\| \tilde \nabla f(x,\xi) - \nabla f(x)\|^2 &\leq \sigma_f^2,\\
    \mathbb{E}\| \tilde \nabla c_i(x,\zeta) - \nabla c_i(x)\|^2 &\leq \sigma_{\nabla c}^2,\\
    \mathbb{E}\| \tilde c_i(x,\zeta) - c_i(x)\|^2 &\leq \sigma_c^2, \\
    \|\nabla c_i(x)\| &\leq C_{\nabla c} \\
    |c_i(x)| &\leq C_{c},\\
    \|\tilde \nabla c_i(x, \zeta)\| &\leq \tilde C_{\nabla c}, \\
    |\tilde c_i(x, \zeta)| &\leq \tilde C_{c},
\end{align*}
for any $i\in \{1,\dots,m\}$.
\begin{lemma}\label{cor: storm_qp}
Let the assumptions in~\eqref{eq: htr_eqs}, \eqref{eq: asp_gen2}, \eqref{eq: asp_gen3} hold and let the parameters of Algorithm~\ref{algorithm:stoc_lalm} be given as
\[
\eta_k = \frac{1}{9\tilde L \rho(k+1)^{3/5}}, \quad
\rho_k = \rho k^{1/5}, \quad \mbox{and} \;\; \alpha_{k+1} = 72 \tilde L^2 \rho_{k+1}^2 \eta_k^2 = \frac{72}{81} (k+1)^{-4/5},
\]
for some constant $\rho>1$, $\tilde L^2 = 4\tilde L_{\nabla f}^2 + 4m^2(\tilde C_c^2\tilde L_{\nabla c}^2 + \tilde C_{\nabla c}^2 \tilde L_c^2)$,
we have
\begin{align*}
   \eta_k\mathbb{E}\|g_k - \nabla Q_{\rho_{k}}(x_{k})\|^2 &\leq \frac{1}{72\tilde L^2 \rho_{k}^2 \eta_{k-1}} \mathbb{E}\|g_k - \nabla Q_{\rho_{k}}(x_{k})\|^2 - \frac{1}{72\tilde L^2 \rho_{k+1}^2 \eta_k}\mathbb{E} \|g_{k+1} - \nabla Q_{\rho_{k+1}}(x_{k+1})\|^2 \notag\\
    &\quad+\frac{7}{18\eta_k} \mathbb{E}\| x_{k+1} - x_k\|^2 
    + \frac{7 m^2 \tilde C_{\nabla c}^2 \tilde C_c^2}{12\tilde L^2 \rho_{k+1}^2 \eta_k} |\rho_{k+1} - \rho_k|^2 \\
    &\quad+\frac{\alpha_{k+1}^2}{12\tilde L^2 \rho_{k+1}^2 \eta_k} \left(\sigma_f^2 + 2m^2\rho_k^2\left(C_c^2 \sigma^2_{\nabla c} + \tilde C_{\nabla c}^2 \sigma_c^2 \right)\right).
\end{align*}
\end{lemma}
\begin{proof}
We apply Lemma~\ref{lem: storm_estimator_genform} with
\begin{subequations}\label{eq: ioe_eqs}
\begin{align}
G_{k}(x_k) &= \nabla Q_{\rho_k}(x_k) = \nabla f(x_k) + \rho_k \sum_{i=1}^m \nabla c_i(x_k) c_i(x_k) , \label{eq: ioe3}\\
G_{k+1}(x_{k+1}) &= \nabla Q_{\rho_{k+1}}(x_{k+1}) = \nabla f(x_{k+1}) + \rho_{k+1} \sum_{i=1}^m \nabla c_i(x_{k+1}) c_i(x_{k+1}) , \label{eq: ioe3.5}\\
\tilde G_{k}(x_k, B_{k+1}) & = \tilde\nabla Q_{\rho_k}(x_k,B_{k+1}) = \tilde \nabla f(x_k, \xi^0_{k+1}) + \rho_k \sum_{i=1}^m \tilde \nabla c_i(x_k, \zeta^{1}_{k+1}) \tilde c_i(x_k, \zeta^{2}_{k+1}),\label{eq: ioe4} \\
\tilde G_{k+1}(x_{k+1}, B_{k+1}) & = \tilde\nabla Q_{\rho_{k+1}}(x_{k+1},B_{k+1}) = \tilde \nabla f(x_{k+1}, \xi^0_{k+1}) + \rho_{k+1} \sum_{i=1}^m \tilde \nabla c_i(x_{k+1}, \zeta^{1}_{k+1}) \tilde c_i(x_{k+1}, \zeta^{2}_{k+1}).\label{eq: ioe4}
\end{align}
\end{subequations}
We note that $G_{k}(x_k)=\E_{B_{k+1}} \tilde G_k(x_k, B_{k+1})$ and $G_{k+1}(x_{k+1})=\E_{B_{k+1}} \tilde G_{k+1}(x_{k+1}, B_{k+1})$, as required by Lemma \ref{lem: storm_estimator_genform}. This estimation is due to $B_{k+1}$ being sampled after the computation of $x_{k+1}$, by the independence of $\zeta^{1}$ and $\zeta^{2}$ and by $\rho_k$ being a deterministic sequence.
We now estimate the error terms on the right-hand side of Lemma~\ref{lem: storm_estimator_genform}. 
Recall that $\mathbb{E}_k$ the expectation conditioning on all the history up to and including $x_{k+1}$.

For lighter notation, we drop the subscripts from the random variables $(\xi^0, \zeta^{1}, \zeta^{2})$ that define $B_{k+1}=(\xi^0_{k+1}, \zeta^{1}_{k+1}, \zeta^{2}_{k+1})$ in this lemma.

First, we estimate the second term on the right-hand side of the inequality in Lemma~\ref{lem: storm_estimator_genform}. 
By Young's inequality, we have
\begin{align}
&\mathbb{E}_k\| \tilde G_{k+1}(x_{k+1}, B_{k+1}) - \tilde G_{k+1}(x_k, B_{k+1}) \|^2 \notag \\
&\le 2\mathbb{E}_k \| \tilde \nabla f(x_{k+1}, \xi^0) - \tilde \nabla f(x_{k}, \xi^0) \|^2\notag \\
&\quad + 2\rho_{k+1}^2 \mathbb{E}_k\left\|  \sum_{i=1}^m \left(\tilde \nabla c_i(x_{k+1}, \zeta^{1}) \tilde c_i(x_{k+1}, \zeta^{2}) - \tilde \nabla c_i(x_{k}, \zeta^{1}) \tilde c_i(x_{k}, \zeta^{2})\right)\right\|^2.\label{eq: vpt4}
\end{align}
By using~\eqref{eq: htr_eqs} and~\eqref{eq: asp_gen3}, we have that
\begin{align}
&\mathbb{E}_k\Big\|  \sum_{i=1}^m \Big(\tilde\nabla  c_i(x_{k+1}, \zeta^{1}) \tilde c_i(x_{k+1}, \zeta^{2}) - \tilde\nabla c_i(x_{k}, \zeta^{1}) \tilde  c_i(x_{k}, \zeta^{2})\Big)\Big\|^2 \notag \\
& \leq m\sum_{i=1}^m \mathbb{E}_k\left\| \tilde\nabla c_i(x_{k+1}, \zeta^{1}) \tilde  c_i(x_{k+1}, \zeta^{2}) - \tilde\nabla c_i(x_{k}, \zeta^{1}) \tilde  c_i(x_{k}, \zeta^{2}) \right\|^2 \notag\\
& \leq 2m\sum_{i=1}^m \mathbb{E}_k\left\| \tilde\nabla c_i(x_{k+1}, \zeta^{ 1})\left(\tilde  c_i(x_{k+1}, \zeta^{2}) - \tilde c_i(x_{k}, \zeta^{2})\right) \right\|^2 \notag\\
&\quad+2m\sum_{i=1}^m \mathbb{E}_k\left\| \left(\tilde\nabla c_i(x_{k+1}, \zeta^{1})  - \tilde\nabla c_i(x_{k}, \zeta^{1})\right) \tilde c_i(x_{k}, \zeta^{2}) \right\|^2\notag \\
& \le 2m^2\left( \tilde C_c^2\tilde L_{\nabla c}^2 + \tilde C_{\nabla c}^2 \tilde L_{c}^2\right) \| x_{k+1} - x_k\|^2,\label{eq: uew3}
\end{align}
where the first inequality follows from the fact that for vectors $Y_i$, we have $\| \sum_{i=1}^m Y_i \|^2 \le (\sum_{i=1}^m \|Y_i\|)^2 \le m \sum_{i=1}^m \| Y_i \|^2 $ and the second by Young's inequality.
Using this in~\eqref{eq: vpt4} along with the first line in~\eqref{eq: htr_eqs}, we have
\begin{align}
\mathbb{E}_k\| \tilde G_{k+1}(x_{k+1}, B_{k+1}) - \tilde G_{k+1}(x_k, B_{k+1}) \|^2 &= \left( 2\tilde L_{\nabla f}^2 + 4\rho_{k+1}^2m^2\left( \tilde C_c^2\tilde L_{\nabla c}^2 + \tilde C_{\nabla c}^2 \tilde L_{c}^2\right) \right) \| x_{k+1} - x_k\|^2 \notag \\
&\le 4 \tilde{L}^2 \rho_{k+1}^2 \| x_{k+1} - x_k\|^2 
\label{eq: ngj4}
\end{align}
where the last line is by the definition of $\tilde L$, see also~\eqref{eq: stoc_lips}.

Second, we estimate the third term on the right-hand side in Lemma~\ref{lem: storm_estimator_genform}:
\begin{align}
\mathbb{E}_k\| \tilde G_{k+1}(x_k, B_{k+1}) - \tilde G_{k}(x_k, B_{k+1})\|^2 &= \mathbb{E}_k \left\| (\rho_{k+1} - \rho_k) \sum_{i=1}^m \tilde\nabla c_i(x_k, \zeta^{1}) \tilde c_i(x_k, \zeta^{2}) \right\|^2 \notag \\
&\leq m^2 \tilde C_{\nabla c}^2 \tilde C_c^2 |\rho_{k+1} - \rho_k|^2.\label{eq: ngj5}
\end{align}
Third, we estimate the fourth term on the right-hand side in Lemma~\ref{lem: storm_estimator_genform}. From Young's inequality, we have
\begin{align}
& \mathbb{E} \| G_{k}(x_k) - \tilde G_{k}(x_k, B_{k+1}) \|^2 \notag \\
& \le 2 \E \left\| \nabla f(x_k) - \tilde \nabla f(x_k, \xi^0) \right\|^2 
+ 2 \rho_k^2 \E \left\| \sum_{i=1}^m \nabla c_i(x_k) c_i(x_k) - \sum_{i=1}^m \tilde\nabla c_i(x_k, \zeta^{1}) \tilde c_i(x_k, \zeta^{2}) \right\|^2 \notag \\
&\leq 2\sigma_f^2 + 4m^2\rho_k^2\left(C_c^2 \sigma^2_{\nabla c} + \tilde C_{\nabla c}^2 \sigma_c^2 \right),\label{eq: ngj6}
\end{align}
where the estimations for the last inequality are similar to~\eqref{eq: uew3}.

By substituting the bounds \eqref{eq: ngj4}, \eqref{eq: ngj5}, \eqref{eq: ngj6} into \eqref{lem: storm_estimator_genform} and also using \eqref{eq: ioe_eqs}, we obtain 
\begin{align}
   \mathbb{E}_k \|g_{k+1} - \nabla Q_{\rho_{k+1}}(x_{k+1})\|^2 &\leq (1-\alpha_{k+1})^2\|g_k - \nabla Q_{\rho_{k}}(x_{k})\|^2 \notag\\
    &\quad+28 \tilde L^2 \rho_{k+1}^2 \| x_{k+1} - x_k\|^2 
    + 42 m^2 \tilde C_{\nabla c}^2 \tilde C_c^2 |\rho_{k+1} - \rho_k|^2 \notag\\
    &\quad+6\alpha_{k+1}^2 \left(\sigma_f^2 + 2m^2\rho_k^2\left(C_c^2 \sigma^2_{\nabla c} + \tilde C_{\nabla c}^2 \sigma_c^2 \right)\right).\label{eq: grl4}
\end{align}
In~\eqref{eq: grl4}, dividing all terms by $72\tilde L^2 \rho_{k+1}^2 \eta_k$  gives
\begin{align}
   \frac{1}{72\tilde L^2 \rho_{k+1}^2 \eta_k}\mathbb{E}_k \|g_{k+1} - \nabla Q_{\rho_{k+1}}(x_{k+1})\|^2 &\leq \frac{(1-\alpha_{k+1})^2}{72\tilde L^2 \rho_{k+1}^2 \eta_k}\|g_k - \nabla Q_{\rho_{k}}(x_{k})\|^2 \notag\\
    &\quad+\frac{7}{18\eta_k} \| x_{k+1} - x_k\|^2 
    + \frac{7 m^2 \tilde C_{\nabla c}^2 \tilde C_c^2}{12\tilde L^2 \rho_{k+1}^2 \eta_k} |\rho_{k+1} - \rho_k|^2 \notag\\
    &\quad+\frac{\alpha_{k+1}^2}{12\tilde L^2 \rho_{k+1}^2 \eta_k} \left(\sigma_f^2 + 2m^2\rho_k^2\left(C_c^2 \sigma^2_{\nabla c} + \tilde C_{\nabla c}^2 \sigma_c^2 \right)\right).\label{eq: vcb3}
\end{align}
We focus on the first term on the right-hand side and will show next
\begin{equation}
\frac{(1-\alpha_{k+1})^2}{72\tilde L^2 \rho_{k+1}^2 \eta_k}\|g_k - \nabla Q_{\rho_{k}}(x_{k})\|^2 \leq \left(\frac{1}{72\tilde L^2 \rho_{k}^2 \eta_{k-1}}-\eta_k\right) \|g_k - \nabla Q_{\rho_{k}}(x_{k})\|^2.\label{eq: gwq3}
\end{equation}
By the definitions of $\alpha_{k+1}, \eta_k, \rho_{k+1}$, we have that $\frac{-\alpha_{k+1}}{72\tilde L^2 \rho_{k+1}^2 \eta_k} = -\eta_k$ therefore~\eqref{eq: gwq3} follows after showing that
\begin{equation}
\frac{1-\alpha_{k+1}+\alpha_{k+1}^2}{72\tilde L^2 \rho_{k+1}^2 \eta_k} \leq \frac{1}{72\tilde L^2 \rho_k^2 \eta_{k-1}} \iff \frac{1}{ \rho_{k+1}^2 \eta_k} - \frac{1}{ \rho_k^2 \eta_{k-1}} \leq \frac{\alpha_{k+1}(1-\alpha_{k+1})}{ \rho_{k+1}^2 \eta_k}.  \label{eq: hrk3}
\end{equation}
Note that by definitions of $\eta_k$ and $\rho_k$, we have that 
\[
 \rho_{k+1}^2 \eta_k =  \rho^2 (k+1)^{2/5}\frac{1}{9\tilde L \rho (k+1)^{3/5}} = \frac{ \rho}{ 9\tilde L (k+1)^{1/5}}.
\]
By substituting the values of $ \rho_{k+1}^2\eta_k$ and $\alpha_{k+1}=72 \tilde L^2 \rho_{k+1}^2\eta_k^2=\frac{72}{81(k+1)^{4/5}}$, we find that~\eqref{eq: hrk3} is equivalent to 
\begin{equation*}
\frac{9\tilde L}{\rho} \left( (k+1)^{1/5} - k^{1/5} \right) \leq \frac{9\tilde L \alpha_{k+1}(1-\alpha_{k+1})(k+1)^{1/5}}{\rho} \iff (k+1)^{1/5} - k^{1/5}  \leq  \alpha_{k+1}(1-\alpha_{k+1})(k+1)^{1/5}.
\end{equation*}
First, we note that $(k+1)^{1/5} - k^{1/5} \leq \frac{1}{5k^{4/5}}$ and also $(1-\alpha_{k+1}) = 1-\frac{72}{81(k+1)^{4/5}} \geq \frac{4}{10}$  for $k\geq 1$.
Therefore,~\eqref{eq: hrk3} will be implied by
\begin{equation*}
\frac{1}{5k^{4/5}} \leq \frac{288}{810(k+1)^{3/5}},
\end{equation*}
which holds for $k\geq 1$. Thus,~\eqref{eq: hrk3}  and consequently~\eqref{eq: gwq3} hold for $k\geq 1$. Using~\eqref{eq: gwq3} to bound the first term on the right-hand side of~\eqref{eq: vcb3} and taking total expectation gives the result.
\end{proof}

\subsection{One iteration inequality}\label{subsec: one_it_stoch}

\begin{replemma}{eq: one_it_ineq1}
Let the assumptions in~\eqref{eq: htr_eqs}, \eqref{eq: asp_gen2}, \eqref{eq: asp_gen3}, \eqref{eq: asp_gen4} hold and let the parameters of Algorithm~\ref{algorithm:stoc_lalm} be given as
\[
\eta_k = \frac{1}{9\tilde L \rho(k+1)^{3/5}}, \quad
\rho_k = \rho k^{1/5}, \quad \mbox{and} \;\; \alpha_{k+1} = 72 \tilde L^2 \rho_{k+1}^2 \eta_k^2 = \frac{72}{81} (k+1)^{-4/5},
\]
for some constant $\rho>1$ and $\tilde L^2 = 4\tilde L_{\nabla f}^2 + 4m^2(\tilde C_c^2\tilde L_{\nabla c}^2 + \tilde C_{\nabla c}^2 \tilde L_c^2)$. Then, we have that
\begin{align*}
 \frac{\eta_k}{72}\E \left[\dist^2(\nabla f(x_{k+1}) + \rho_{k}\nabla c(x_{k+1})^\top c({x_{k+1}}), - N_X(x_{k+1}))\right] \leq  \E[Y_{k}- Y_{k+1} + |Q_{\rho_k}(x_{k+1}) - Q_{\rho_{k+1}}(x_{k+1})|] + \mathcal{E}_{k+1},
\end{align*}
where 
\begin{equation*}
Y_{k+1} = Q_{\rho_{k+1}}(x_{k+1}) + \frac{1}{72\tilde L^2 \rho_{k+1}^2 \eta_k} \| g_{k+1} - \nabla Q_{\rho_{k+1}}(x_{k+1}) \|^2
\end{equation*}
and
\begin{equation*}
\mathcal{E}_{k+1} = \frac{7 m^2 \tilde C_{\nabla c}^2 \tilde C_c^2}{12\tilde L^2 \rho_{k+1}^2 \eta_k} |\rho_{k+1} - \rho_k|^2 +\frac{\alpha_{k+1}^2}{12\tilde L^2 \rho_{k+1}^2 \eta_k} \left(\sigma_f^2 + 2m^2\rho_k^2\left(C_c^2 \sigma^2_{\nabla c} + \tilde C_{\nabla c}^2 \sigma_c^2 \right)\right).
\end{equation*}
\end{replemma}
\begin{remark}\label{rem: cons1}
By the definitions of $\eta_k$, $\rho_k$, we have that the first term of $\mathcal{E}_{k+1}$ is $O(k^{-7/5})$, the second term of $\mathcal{E}_{k+1}$ is $O(k^{-1})$, therefore $\sum_{k=1}^K \mathcal{E}_{k+1} = O(\log (K+1))$.
\end{remark}
\begin{proof}
By descent lemma applied on $x\mapsto Q_{\rho_k}(x)$, we have (by denoting the Lipschitz constant of $\nabla Q_{\rho_k}(x)$ as $L_{\rho_k}= \rho_k(L_{\nabla f} + m(C_c L_{\nabla c} + C_{\nabla c} L_c))$),
\begin{align}
    Q_{\rho_k}(x_{k+1}) &\leq Q_{\rho_k}(x_k) + \langle \nabla Q_{\rho_k}(x_k), x_{k+1} - x_k \rangle + \frac{L_{\rho_k}}{2} \| x_{k+1} - x_k\|^2 \notag \\
    &= Q_{\rho_k}(x_k) + \langle g_k, x_{k+1} - x_k \rangle + \langle \nabla Q_{\rho_k}(x_k) - g_k, x_{k+1} - x_k \rangle + \frac{L_{\rho_k}}{2} \| x_{k+1} - x_k\|^2 \notag \\
    &\leq Q_{\rho_k}(x_k) + \langle g_k, x_{k+1} - x_k \rangle +  \left(\frac{L_{\rho_k}}{2} +\frac{1}{2\eta_k}\right) \| x_{k+1} - x_k\|^2 + \frac{\eta_k}{2} \|\nabla Q_{\rho_k}(x_k) - g_k\|^2, \label{eq: storm_eq1_1}
\end{align}
where we added and subtracted $\langle g_k, x_{k+1} - x_k\rangle$ for the equality and then used Young's inequality.

By the definition of $x_{k+1}$ in Algorithm \ref{algorithm:stoc_lalm} and $x_k \in X$, we have 
\begin{equation*}
    \langle x_{k+1} - x_k + \eta_k g_k, x_k - x_{k+1} \rangle \geq 0 \iff \langle g_k, x_{k+1} - x_{k} \rangle \leq - \frac{1}{\eta_k} \| x_k - x_{k+1} \|^2.
\end{equation*}
By using this estimate in~\eqref{eq: storm_eq1_1}, and then splitting the last term, we have
\begin{align*}
     Q_{\rho_k}(x_{k+1}) &\leq Q_{\rho_k}(x_k) +\left( -\frac{1}{\eta_k} + \frac{L_{\rho_k}}{2} + \frac{1}{2\eta_k} \right) \| x_{k+1} - x_k\|^2+ \frac{\eta_k}{2} \| \nabla Q_{\rho_k}(x_k)-g_k\|^2 \\
    &= Q_{\rho_k}(x_k) +\left( -\frac{1}{2\eta_k} + \frac{L_{\rho_k}}{2} \right) \| x_{k+1} - x_k\|^2 +\eta_k \| \nabla Q_{\rho_{k}}(x_k)-g_k\|^2 - \frac{\eta_k}{2} \| \nabla Q_{\rho_{k}}(x_k)-g_k\|^2.
\end{align*}
We take expectation on this inequality and then use Lemma~\ref{cor: storm_qp} to bound the expectation of the third term on the right-hand side to get 
\begin{align*}
    & \E Q_{\rho_k}(x_{k+1}) + \frac{1}{72\tilde L^2 \rho_{k+1}^2 \eta_k} \mathbb{E} \| g_{k+1} - \nabla Q_{\rho_{k+1}}(x_{k+1}) \|^2 \\
    &\leq \E Q_{\rho_k}(x_{k}) + \frac{1}{72\tilde L^2 \rho_{k}^2 \eta_{k-1}} \E \| g_{k} - \nabla Q_{\rho_{k}}(x_{k}) \|^2
    -\frac{\eta_k}{2} \E \| g_k - \nabla Q_{\rho_{k}}(x_k)\|^2 \\
    &\quad+ \left( \frac{L_{\rho_k}}{2} + \frac{7}{18\eta_k} - \frac{1}{2\eta_k} \right) \E \| x_{k+1} - x_k\|^2 
      + \frac{7 m^2 \tilde C_{\nabla c}^2 \tilde C_c^2}{12\tilde L^2 \rho_{k+1}^2 \eta_k} |\rho_{k+1} - \rho_k|^2 \\
    &\quad+\frac{\alpha_{k+1}^2}{12\tilde L^2 \rho_{k+1}^2 \eta_k} \left(\sigma_f^2 + m^2\rho_k^2\left(C_c^2 \sigma^2_{\nabla c} + \tilde C_{\nabla c}^2 \sigma_c^2 \right)\right)
\end{align*}
First, by the definition of $\eta_k$ and by $L_{\rho_k} \leq \tilde L_{\rho_k} \leq \tilde L_{\rho_{k+1}} = \tilde L \rho(k+1)^{1/5} \leq \tilde L \rho (k+1)^{3/5}$ which is due to~\eqref{eq: stoc_lips}, Jensen's inequality and $\mathbb{E}[\tilde \nabla Q_{\rho_k}(x, B_{k+1})] = \nabla Q_{\rho_k}(x)$, we have  that $\frac{ L_{\rho_k}}{2} \leq \frac{\tilde L{\rho}(k+1)^{3/5}}{2} = \frac{1}{18\eta_k}$ since $\tilde L{\rho}(k+1)^{3/5} = \frac{1}{9\eta_k}$. We consequently have $\frac{L_{\rho_k}}{2} + \frac{7}{18\eta_k} - \frac{1}{2\eta_k} \leq \frac{4}{9\eta_k} - \frac{1}{2\eta_k}\leq -\frac{1}{18\eta_k}$. 
We then add to both sides $Q_{\rho_{k+1}}(x_{k+1})$, use the definitions of $Y_k$ and $\mathcal{E}_k$ along with $\frac{\eta_k}{2}\geq\frac{\eta_k}{18}$ to get
\begin{equation}
\frac{\eta_k}{18} \mathbb{E}\left[\eta_k^{-2}\|x_{k+1} - x_k\|^2 + \|g_k - \nabla Q_{\rho_k}(x_k)\|^2  \right]\leq \mathbb{E}\left[ Y_{k}-Y_{k+1} + |Q_{\rho_k}(x_{k+1}) - Q_{\rho_{k+1}}(x_{k+1})|\right] + \mathcal{E}_{k+1}.\label{eq: spe4}
\end{equation}
We will now show that
\begin{equation}\label{eq: gir3}
    \dist^2(\nabla f(x_{k+1}) + \rho_{k}\nabla c(x_{k+1})^\top  c({x_{k+1}}), - N_X(x_{k+1})) \leq 4(\eta_k^{-2}\| x_{k+1} - x_k\|^2 + \| g_k - \nabla Q_{\rho_{k}}(x_k)\|^2).
\end{equation}
By the definition of $x_{k+1}$, we have
\begin{align*}
    & 0\in x_{k+1} - x_k + \eta_k g_k + \partial i_X(x_{k+1}) \\
    & \iff \eta_k^{-1}(x_k - x_{k+1}) + ( \nabla Q_{\rho_{k}}(x_k) - g_k) + (\nabla Q_{\rho_{k}}(x_{k+1}) - \nabla Q_{\rho_{k}}(x_k)) \in \partial i_X(x_{k+1}) + \nabla Q_{\rho_{k}}(x_{k+1}) \\
    & \iff \eta_k^{-1}(x_k - x_{k+1}) + ( \nabla Q_{\rho_{k}}(x_k) - g_k) + (\nabla Q_{\rho_{k}}(x_{k+1}) - \nabla Q_{\rho_{k}}(x_k)) \\
    & \hspace{3in} \in \partial i_X(x_{k+1}) + \nabla f(x_{k+1}) + \rho_{k} \nabla c(x_{k+1})^\top c(x_{k+1}),
\end{align*}
where the last step also used the definition of $\nabla Q_{\rho_k}(x_{k+1})$.
This gives
\begin{align*}
& \dist^2(\nabla f(x_{k+1}) + \rho_{k}\nabla c(x_{k+1})^\top c({x_{k+1}}), - N_X(x_{k+1})) \\
&\leq 3 \eta_k^{-2} \| x_{k+1} - x_k\|^2 + \| g_k - \nabla Q_{\rho_{k}}(x_k) \|^2 + \|\nabla Q_{\rho_{k}}(x_k)-\nabla Q_{\rho_{k}}(x_{k+1})\|^2 \\
&\leq 4 \eta_k^{-2} \| x_{k+1} - x_k\|^2 + \| g_k - \nabla Q_{\rho_{k}}(x_k) \|^2,
\end{align*}
where the last step is due to $\nabla Q_{\rho_{k}}$ being $L_{\rho_{k}}$-Lipschitz, $L_{\rho_{k}} \leq \tilde L_{\rho_{k}}\leq \tilde L_{\rho_{k+1}}$ and hence $L_{\rho_k}\leq \tilde L\rho(k+1)^{1/5}< 9\tilde L \rho(k+1)^{3/5}=\eta_k^{-1}$ by the definition of $\eta_k$. Using~\eqref{eq: gir3} on~\eqref{eq: spe4} and taking total expectation gives the result.
\end{proof}
\subsection{Controlling the change of penalty parameters}
Using variable penalty parameters allows us to remove assumptions on initialization that was done in~\cite{shi2022momentum} for solving a special case of our problem. To handle the effect of the change on penalty parameters we have the next lemma that uses~\eqref{eq: asp_gen4} and Lemma~\ref{eq: one_it_ineq1}.
\begin{lemma}\label{lem: sps2}
Let the assumptions in~\eqref{eq: htr_eqs}, \eqref{eq: asp_gen2}, \eqref{eq: asp_gen3}, \eqref{eq: asp_gen4} hold and let the parameters of Algorithm~\ref{algorithm:stoc_lalm} be given as
\[
\eta_k = \frac{1}{9\tilde L \rho(k+1)^{3/5}}, \quad
\rho_k = \rho k^{1/5}, \quad \mbox{and} \;\; \alpha_{k+1} = 72 \tilde L^2 \rho_{k+1}^2 \eta_k^2 = \frac{72}{81} (k+1)^{-4/5},
\]
for some constant $\rho>1$, we have that
\begin{align*}
&\sum_{k=1}^K \mathbb{E} |Q_{\rho_k}(x_{k+1}) - Q_{\rho_{k+1}}(x_{k+1})| \\
&\leq b Q_{\rho_1}(x_1) - \frac{b\underline Q}{(K+1)^{3/5}} + \frac{9b}{72\tilde L \rho } \| g_1 - \nabla Q_{\rho_1}(x_1)\|^2 \\
&\quad+ \sum_{k=1}^K \frac{b(B_f+\rho C_c^2)}{k(k+1)^{2/5}} + \sum_{k=1}^K \frac{b \mathcal{E}_{k+1}}{(k+1)^{3/5}} + \sum_{k=1}^K  \frac{b \rho C_c^2}{k^{4/5}(k+1)^{3/5}} + \sum_{k=1}^K \frac{2C_{\nabla f}^2}{5\rho\delta^2 k^{6/5}},
\end{align*}
where $b_k = \frac{648\cdot 8 \tilde L}{5\delta^2(k+1)^{3/5}}$, $b = \frac{648\cdot 8 \tilde L}{5\delta^2}$ and $\mathcal{E}_{k+1}$ is as given in Lemma~\ref{eq: one_it_ineq1}.
\end{lemma}
\begin{remark}\label{rem: cons3}
In view of Remark~\ref{rem: cons1} and since $b_k=O(k^{-3/5})$, the right-hand side in this lemma is finite. 
\end{remark}
\begin{proof}
We start with the error term
\begin{align}\label{eq: fgr4}
|Q_{\rho_k}(x_{k+1}) - Q_{\rho_{k+1}}(x_{k+1})| = (\rho_{k+1} - \rho_{k}) \|c(x_{k+1})\|^2.
\end{align}
We have by the assumption in~\eqref{eq: asp_gen4} and triangle inequality that
\begin{align}
&\dist(\rho_{k}\nabla c(x_{k+1})^\top c(x_{k+1}), -N_X(x_{k+1})) \geq \rho_{k}\delta \| c(x_{k+1}) \| \notag \\
&\iff \|c(x_{k+1}) \| \leq \frac{1}{\rho_{k} \delta}\left( \| \nabla f(x_{k+1}) \| + \dist(\nabla f(x_{k+1}) + \rho_{k} \nabla c(x_{k+1})^\top c(x_{k+1}), -N_X(x_{k+1}) \right).\label{eq: bfh4}
\end{align}
Using this in~\eqref{eq: fgr4} gives
\begin{align} 
&|Q_{\rho_k}(x_{k+1}) - Q_{\rho_{k+1}}(x_{k+1})|\notag \\
&\leq \frac{2|\rho_k - \rho_{k+1}|}{\rho_{k}^2 \delta^2}\left( \| \nabla f(x_{k+1})\|^2 + \dist^2(\nabla f(x_{k+1}) + \rho_{k} \nabla c(x_{k+1})^\top c(x_{k+1}), -N_X(x_{k+1}) \right).\label{eq: hhp4}
\end{align}
As a result, we wish to bound 
\begin{align}
&\sum_{k=1}^K |Q_{\rho_k}(x_{k+1}) - Q_{\rho_{k+1}}(x_{k+1})| \notag \\
&\leq \sum_{k=1}^K \frac{2|\rho_k - \rho_{k+1}|}{\rho_{k}^2 \delta^2}\left( \| \nabla f(x_{k+1})\|^2 + \dist^2(\nabla f(x_{k+1}) + \rho_{k} \nabla c(x_{k+1})^\top c(x_{k+1}), -N_X(x_{k+1}) \right).\label{eq: xnw3}
\end{align}
For bounding the right-hand side of this inequality, we  use a crude bound that can be obtained by Lemma~\ref{eq: one_it_ineq1}.
By using~\eqref{eq: fgr4} with the uniform upper bound $\|c(x_{k+1})\|\leq C_c$ to bound the third term on the right-hand side of the result in Lemma~\ref{eq: one_it_ineq1}, we get
\begin{align}
 \frac{\eta_k}{72}\mathbb{E} \left[\dist^2(\nabla f(x_{k+1}) + \rho_{k}\nabla c(x_{k+1})^\top c({x_{k+1}}), - N_X(x_{k+1})) \right] \leq \mathbb{E}\left[ Y_{k}- Y_{k+1}\right] + \mathcal{E}_{k+1} + |\rho_k - \rho_{k+1}| C_c^2,\label{eq: sow4}
\end{align}
Let $b_k = \frac{648\cdot 8 \tilde L}{5\delta^2(k+1)^{3/5}}$ and $b=\frac{648\cdot 8}{5\delta^2}$.
First, note that for $k\geq 1$
\begin{equation*}
b_k\cdot \frac{\eta_k}{72} =\frac{8}{5\rho\delta^2(k+1)^{6/5}} \geq \frac{2}{5\rho\delta^2(k)^{6/5}} \geq \frac{2|\rho_k - \rho_{k+1}|}{\rho^2\delta^2k^{2/5} } = \frac{2|\rho_k-\rho_{k+1}|}{\rho_k^2\delta^2},
\end{equation*}
since $|\rho_{k+1} - \rho_k| \leq \frac{\rho}{5k^{4/5}}$ and $4(k)^{6/5}\geq (k+1)^{6/5}$ for $k\geq 1$.
Hence, after multiplying~\eqref{eq: sow4} by $b_k$, we get
\begin{align}
 &\frac{2|\rho_k - \rho_{k+1}|}{\rho_{k}^2 \delta^2}\mathbb{E}\left[\dist^2(\nabla f(x_{k+1}) + \rho_{k}\nabla c(x_{k+1})^\top c({x_{k+1}}), - N_X(x_{k+1}))\right]\notag  \\
 &\leq b_k(\mathbb{E} Y_{k}-\mathbb{E} Y_{k+1}) + b_k\mathcal{E}_{k+1} + b_k|\rho_k - \rho_{k+1}| C_c^2.\label{eq: sge3}
\end{align}
In view of~\eqref{eq: xnw3}, this gives
\begin{align}
 &\sum_{k=1}^K \mathbb{E}|Q_{\rho_k}(x_{k+1}) - Q_{\rho_{k+1}}(x_{k+1})|\notag\\
&\leq \sum_{k=1}^K\left(b_k(\mathbb{E} Y_{k}-\mathbb{E} Y_{k+1}) + b_k\mathcal{E}_{k+1} + b_k|\rho_k - \rho_{k+1}| C_c^2+\frac{2|\rho_k-\rho_{k+1}|C_{\nabla f}^2}{\rho_k^2\delta^2}\right),\label{eq: vcd3}
\end{align}
after also using $\|\nabla f(x_{k+1})\|^2 \leq C_{\nabla f}^2$.
By the definitions in Lemma~\ref{eq: one_it_ineq1}, we have $b_k\mathcal{E}_{k+1}= O\left(\frac{1}{k^{8/5}}\right)$, $b_k|\rho_k-\rho_{k+1}|=O\left(\frac{1}{k^{7/5}}\right)$, and $\frac{|\rho_k - \rho_{k+1}|}{\rho_k^2}=O\left(\frac{1}{k^{6/5}} \right)$. Hence, we now bound the term $b_k(\mathbb{E}Y_k-\mathbb{E}Y_{k+1})$. By the definition of $Y_{k+1}$ in Lemma~\ref{eq: one_it_ineq1}, we have
\begin{align}
    &\sum_{k=1}^K b_k(\mathbb{E}Y_k-\mathbb{E}Y_{k+1}) \notag\\
    &= \sum_{k=1}^K b_k \mathbb{E}\left[Q_{\rho_k}(x_k) - Q_{\rho_{k+1}}(x_{k+1})\right] \notag \\
    &\quad+ \sum_{k=1}^K b_k \mathbb{E}\left[ \frac{1}{72\tilde L^2 \rho_{k}^2 \eta_{k-1}} \| g_{k} - \nabla Q_{\rho_{k}}(x_{k}) \|^2- \frac{1}{72\tilde L^2 \rho_{k+1}^2 \eta_k} \| g_{k+1} - \nabla Q_{\rho_{k+1}}(x_{k+1}) \|^2 \right]\label{eq: sdw3}
\end{align}
First, we have
\begin{align*}
&\sum_{k=1}^K \frac{1}{(k+1)^{3/5}}\left(Q_{\rho_{k}}(x_k) - Q_{\rho_{k+1}}(x_{k+1})\right) \\
&= \sum_{k=1}^K \left(\frac{1}{k^{3/5}} Q_{\rho_{k}}(x_k) - \frac{1}{(k+1)^{3/5}}Q_{\rho_{k+1}}(x_{k+1}) \right)
+ \sum_{k=1}^K \left( \frac{1}{(k+1)^{3/5}}  - \frac{1}{k^{3/5}}\right) Q_{\rho_{k}}(x_{k}) \\
&\leq Q_{\rho_1}(x_1) - \frac{\underline Q}{(K+1)^{3/5}} + (B_f+\rho C_c^2) \sum_{k=1}^{K} \frac{1}{k(k+1)^{2/5}},
\end{align*}
where the last step is by $|Q_{\rho_{k+1}}(x_{k+1})| = |f(x_{k+1}) + \rho_{k+1} \| c(x_{k+1}\|^2| \leq B_f + \rho (k+1)^{1/5} C_c^2$ and $\left|\frac{1}{(k+1)^{3/5}} - \frac{1}{k^{3/5}} \right| \leq \frac{1}{(k+1)^{3/5} k}$ since $\frac{1}{k^{3/5}} - \frac{1}{(k+1)^{3/5}} = \frac{(k+1)^{3/5} - k^{3/5}}{(k+1)^{3/5} k ^{3/5}} = \frac{(k+1)^{3/5}k^{2/5} - k}{(k+1)^{3/5} k} \leq\frac{1}{(k+1)^{3/5} k}$ and $Q_{\rho}(x)\geq \underline Q > -\infty$ by \eqref{eq: asp_gen3}. After multiplying by $b$ and using $b_k = \frac{b}{(k+1)^{3/5}}$ on the last estimate gives
\begin{equation}\label{eq: pme3}
    \sum_{k=1}^K b_k\left(Q_{\rho_{k}}(x_k) - Q_{\rho_{k+1}}(x_{k+1})\right) \leq bQ_{\rho_1}(x_1) - \frac{\underline bQ}{(K+1)^{3/5}} + b(B_f+\rho C_c^2) \sum_{k=1}^{K} \frac{1}{k(k+1)^{2/5}}.
\end{equation}
Next, for the second term in~\eqref{eq: sdw3}, we have
\begin{align*}
&\sum_{k=1}^K \frac{1}{(k+1)^{3/5}}\left(\frac{1}{72\tilde L^2 \rho_{k}^2 \eta_{k-1}} \| g_{k} - \nabla Q_{\rho_{k}}(x_{k}) \|^2- \frac{1}{72\tilde L^2 \rho_{k+1}^2 \eta_k} \| g_{k+1} - \nabla Q_{\rho_{k+1}}(x_{k+1}) \|^2\right) \\
&\leq \sum_{k=1}^K \left( \frac{1}{k^{3/5}}\frac{1}{72\tilde L^2 \rho_{k}^2 \eta_{k-1}}\| g_{k} - \nabla Q_{\rho_{k}}(x_{k}) \|^2 -\frac{1}{(k+1)^{3/5}}\frac{1}{72\tilde L^2 \rho_{k+1}^2 \eta_{k}}\| g_{k+1} - \nabla Q_{\rho_{k+1}}(x_{k+1}) \|^2\right)\\
&\leq \frac{9}{72\tilde L \rho}\|g_1-\nabla Q_{\rho_1}(x_1)\|^2.
\end{align*}
Hence, after multiplying this estimate by $b$ and using $b_k = \frac{b}{(k+1)^{3/5}}$, we obtain
\begin{align}\label{eq: pme4}
    &\sum_{k=1}^K \frac{1}{(k+1)^{3/5}}\left(\frac{1}{72\tilde L^2 \rho_{k}^2 \eta_{k-1}} \| g_{k} - \nabla Q_{\rho_{k}}(x_{k}) \|^2- \frac{1}{72\tilde L^2 \rho_{k+1}^2 \eta_k} \| g_{k+1} - \nabla Q_{\rho_{k+1}}(x_{k+1}) \|^2\right) \notag \\
    &\leq \frac{9b}{72\tilde L \rho}\|g_1 - \nabla Q_{\rho_1}(x_1)\|^2.
\end{align}
We take expectations and then combine \eqref{eq: pme3} and \eqref{eq: pme4} in~\eqref{eq: sdw3} to have
\begin{align}
    \sum_{k=1}^K b_k(\mathbb{E}Y_k-\mathbb{E}Y_{k+1}) &\leq bQ_{\rho_1}(x_1) - \frac{\underline bQ}{(K+1)^{3/5}} + b(B_f+\rho C_c^2) \sum_{k=1}^{K} \frac{1}{k(k+1)^{2/5}} \notag \\&\quad+\frac{9b}{72\tilde L \rho}\|g_1 - \nabla Q_{\rho_1}(x_1)\|^2,
\end{align}
Using this estimate in \eqref{eq: vcd3} gives the result after also substituting the values of $\rho_k, b_k$.
\end{proof}

\subsection{Main theorem}
\begin{reptheorem}{eq: asb4}
Let the assumptions in~\eqref{eq: htr_eqs}, \eqref{eq: asp_gen2}, \eqref{eq: asp_gen3}, \eqref{eq: asp_gen4} hold.
Let
\[
\eta_k = \frac{1}{9\tilde L \rho(k+1)^{3/5}}, \quad
\rho_k = \rho k^{1/5}, \quad \mbox{and} \;\; \alpha_{k+1} = \frac{72}{81} (k+1)^{-4/5},
\]
for some constant $\rho>1$ and $\tilde{L}^2 = 4\tilde L_{\nabla f}^2 + 4m^2 (\tilde C_c^2 \tilde L_{\nabla c}^2 + \tilde C_{\nabla c}^2 \tilde L_c^2)$, we have that there exists $\lambda$ such that
\begin{align*}
    \mathbb{E}\left[\dist(\nabla f(x_{\hat k+1}) + \nabla c(x_{\hat k+1})^\top \lambda, -N_X(x_{\hat k +1}))\right] &\leq \varepsilon,\\
    \mathbb{E}\|c(x_{\hat k + 1})\|&\leq\varepsilon.
\end{align*}
with number of iterations bounded by $\tilde{O}(\varepsilon^{-5})$. 
\end{reptheorem}
\begin{proof}
We start by summing the one iteration inequality in Lemma~\ref{eq: one_it_ineq1}
\begin{align*}
 &\sum_{k=1}^K \frac{\eta_k}{72} \mathbb{E} \left[\dist^2(\nabla f(x_{k+1}) + \rho_{k}\nabla c(x_{k+1})^\top c({x_{k+1}}), - N_X(x_{k+1}))\right] \\
 &\leq \mathbb{E} Y_{1} + \sum_{k=1}^K |Q_{\rho_k}(x_{k+1}) - Q_{\rho_{k+1}}(x_{k+1})|  + \sum_{k=1}^K \mathcal{E}_{k+1}.
\end{align*}
We use that $\eta_k \geq \eta_K$ and divide both sides of the inequality by $K$ to derive
\begin{align*}
 &\frac{1}{K}\sum_{k=1}^K \mathbb{E} \left[\dist^2(\nabla f(x_{k+1}) + \rho_{k}\nabla c(x_{k+1})^\top c({x_{k+1}}), - N_X(x_{k+1}))\right] \\
 &\leq \frac{72}{K\eta_K} \Big(\mathbb{E} Y_{1} + \sum_{k=1}^K |Q_{\rho_k}(x_{k+1}) - Q_{\rho_{k+1}}(x_{k+1})| + \sum_{k=1}^K \mathcal{E}_{k+1}\Big).
\end{align*}
In view of Lemma~\ref{lem: sps2}, Remark~\ref{rem: cons1} and Remark~\ref{rem: cons3}, we have that the sums in the right-hand side are either finite or increase logarithmically in $K$ and since $\eta_K= O(K^{-3/5})$, we have
\begin{equation}\label{eq: foe3}
\frac{1}{K}\sum_{k=1}^K \mathbb{E} \left[\dist^2(\nabla f(x_{k+1}) + \rho_{k}\nabla c(x_{k+1})^\top c({x_{k+1}}), - N_X(x_{k+1}))\right]= O\left(\frac{\log (K+1)}{K^{2/5}}\right),
\end{equation}
along with $\|\nabla f(x)\|^2\leq C_{\nabla f}^2$ as per \eqref{eq: asp_gen3}.

This estimate along with $\rho_k=\rho k^{1/5}$ in ~\eqref{eq: bfh4} gives 
\begin{align*}
    \frac{1}{K}\sum_{k=1}^K\mathbb{E}\|c(x_{k+1}) \|^2 &\leq \frac{1}{K}\sum_{k=1}^K\frac{2}{\rho^2 k^{2/5}\delta^2}\mathbb{E}\left[ \| \nabla f(x_{k+1}) \|^2 + \dist^2(\nabla f(x_{k+1}) + \rho_{k} \nabla c(x_{k+1})^\top c(x_{k+1}), -N_X(x_{k+1}) \right]\\
    &= O\left(\frac{\log (K+1)}{K^{2/5}}\right),
\end{align*}
This inequality with \eqref{eq: foe3} gives
\begin{equation*}
    \frac{1}{K}\sum_{k=1}^K \mathbb{E} \left[\dist^2(\nabla f(x_{k+1}) + \rho_{k}\nabla c(x_{k+1})^\top c({x_{k+1}}), - N_X(x_{k+1})) + \| c(x_{k+1})\|^2\right] = O\left( \frac{\log (K+1)}{K^{2/5}} \right).
\end{equation*}
Hence the claims follow by using $\lambda = \rho_{\hat k} c(x_{\hat k+1})$ and Jensen's inequality.
\end{proof}

\section{Extensions}
Since the arguments in these parts are mostly the same as the previous section, the analyses in these two sections do not spell out all the details but mentions the changes compared to Section \ref{sec: qp_stoc_const_app}.
In this section, we will consider two extensions and show how they follow by minor adjustments on the analysis of the previous section.
\subsection{Dual variable updates}\label{subsec: ext1}
The algorithm in this case, written explicitly, is
\begin{subequations}
\begin{align}
    &x_{k+1} = P_{X}(x_k - \eta_k g_k) \label{eq: alg_vr2_step1}\\
    &\text{Sample }B_{k+1} = (\xi^0_{k+1}, \zeta^{1}_{k+1}, \zeta^{2}_{k+1}) \in \Xi \times \Zeta^{2} \text{~and set $\nabla Q_{\rho}(x, \lambda, B)$ as~\eqref{eq: sor12}} \\
    &g_{k+1} = \tilde\nabla Q_{\rho_{k+1}}(x_{k+1}, \lambda_{k+1}, B_{k+1}) + (1-\alpha_{k+1})(g_k - \tilde\nabla Q_{\rho_{k}}(x_{k}, \lambda_k, B_{k+1})),\\
    &\lambda_{k+2, i} = \lambda_{k+1, i} + \gamma_{k+1} \tilde c_i(x_{k+1}, \zeta^{2}_{k+1})~~\forall i=\{1,\dots,m\}\label{eq: alg_vr2_step4}
\end{align}
\end{subequations}
where we have
\begin{equation}\label{eq: sor12}
\tilde{\nabla}Q_{\rho}(x, \lambda, B) = \tilde{\nabla}f(x, \xi^{0}) + \sum_{i=1}^m \lambda_{i}\tilde \nabla c_i(x, \zeta^{1}) + \rho \sum_{i=1}^m \tilde \nabla c_{i}(x, \zeta^{1}) \tilde c_i(x, \zeta^{2}).
\end{equation}
\begin{reptheorem}{th: th_ext_1_main}
For the algorithm described in~\eqref{eq: alg_vr2_step1}-\eqref{eq: alg_vr2_step4} (as sketched in Section \ref{subsec: main_alm_stoc}), let 
\[
\eta_k = \frac{1}{9\tilde L \rho(k+1)^{3/5}}, \quad
\rho_k = \rho k^{1/5}, \quad \gamma_k = \frac{\gamma}{k(\log(k+1))^2|\tilde c_i(x_k, \zeta^{2}_k)|}, \quad \mbox{and} \;\; \alpha_{k+1} = 72 \tilde L^2 \rho_{k+1}^2 \eta_k^2 = \frac{72}{81} (k+1)^{-4/5},
\]
for some constant $\rho>1$ and $\tilde L^2 = \frac{3}{4} \tilde L_{\nabla f}^2 + \frac{3}{4}m^2(\|\lambda_1\|+ 4)^2\tilde L_{\nabla c}^2+\frac{3}{2}m^2(\tilde C^2\tilde L_{\nabla c}^2 + \tilde C^2_{\nabla c} \tilde L_c^2)$. Let also the assumptions in \eqref{eq: htr_eqs}, \eqref{eq: asp_gen2}, \eqref{eq: asp_gen3}, \eqref{eq: asp_gen4} hold. We have that there exists $\lambda$ such that
\begin{align*}
    \mathbb{E}\left[\dist(\nabla f(x_{\hat k+1}) +  \nabla c(x_{\hat k+1})^\top \lambda, -N_X(x_{\hat k +1}))\right] &\leq \varepsilon, \\
    \mathbb{E}\| c(x_{\hat k+1})\|^2 &\leq \varepsilon
\end{align*}
with number of iterations bounded by $\tilde{O}(\varepsilon^{-5})$. 
\end{reptheorem}
\begin{proof}
For convenience, let us denote $\tilde \gamma_{k} = \gamma_{k, i}|\tilde c_i(x_k, \zeta^{2}_k)|=\frac{\gamma}{k(\log(k+1))^2}$. 

$\bullet$ Modification of Lemma~\ref{cor: storm_qp}.

We apply Lemma~\ref{lem: storm_estimator_genform} with (\emph{cf.} Lemma~\ref{cor: storm_qp})
\begin{subequations}
\begin{align*}
G_{k}(x_k) &= \nabla Q_{\rho_k}(x_k, \lambda_k) = \nabla f(x_k) +\sum_{i=1}^m \nabla c_i(x_k)\lambda_{k, i} + \rho_k \sum_{i=1}^m \nabla c_i(x_k) c_i(x_k) , \\
G_{k+1}(x_{k+1}) &= \nabla Q_{\rho_{k+1}}(x_{k+1}, \lambda_{k+1}) = \nabla f(x_{k+1}) +\sum_{i=1}^m \nabla c_i(x_{k+1})\lambda_{k+1, i}+ \rho_{k+1} \sum_{i=1}^m \nabla c_i(x_{k+1}) c_i(x_{k+1}), \\
\tilde G_{k}(x_k, B_{k+1}) & = \tilde\nabla Q_{\rho_k}(x_k,\lambda_k, B_{k+1}) = \tilde \nabla f(x_k, \xi^0_{k+1}) + \sum_{i=1}^m\tilde c_i(
x_k, \zeta^{1})\lambda_{k, i} + \rho_k \sum_{i=1}^m \tilde \nabla c_i(x_k, \zeta^{1}_{k+1}) \tilde c_i(x_k, \zeta^{2}_{k+1}), \\
\tilde G_{k+1}(x_{k+1}, B_{k+1})  &= \tilde\nabla Q_{\rho_{k+1}}(x_{k+1},\lambda_{k+1}, B_{k+1})\\
&= \tilde \nabla f(x_{k+1}, \xi^0_{k+1}) + \sum_{i=1}^m \tilde \nabla c_i(x_{k+1}, \zeta^1) \lambda_{k+1, i} + \rho_{k+1} \sum_{i=1}^m \tilde \nabla c_i(x_{k+1}, \zeta^{1}_{k+1}) \tilde c_i(x_{k+1}, \zeta^{2}_{k+1}).
\end{align*}
\end{subequations}
where $\mathbb{E}_{B_{k+1}} \tilde G_k(x_k, B_{k+1}) = G_k(x_k)$ and $\mathbb{E}_{B_{k+1}} \tilde G_{k+1}(x_{k+1}, B_{k+1}) = G_{k+1}(x_{k+1})$ as before. We also define
\begin{align*}
    \tilde G_{k+1}(x_{k}, B_{k+1})  &= \tilde\nabla Q_{\rho_{k+1}}(x_{k},\lambda_{k+1}, B_{k+1})\\
&= \tilde \nabla f(x_{k}, \xi^0_{k+1}) + \sum_{i=1}^m \tilde \nabla c_i(x_{k}, \zeta^1) \lambda_{k+1, i} + \rho_{k+1} \sum_{i=1}^m \tilde \nabla c_i(x_{k}, \zeta^{1}_{k+1}) \tilde c_i(x_{k}, \zeta^{2}_{k+1}).
\end{align*}
As a result, the norm of dual vector $\lambda_k$ will affect the bounds in Lemma~\ref{cor: storm_qp}. First note that by the definition of $\gamma_k$, we have $\lambda_{k+1, i} = \lambda_{k, i} + \frac{\gamma}{k(\log(k+1))^2|\tilde c_i(x_{k}, \zeta^{2}_k)|}\tilde c_i(x_{k}, \zeta^{2}_k) $ and hence $\lambda_{k+1, i} = \lambda_{1, i} + \sum_{j=1}^k \frac{\gamma}{j(\log(j+1))^2|\tilde c_i(x_{j}, \zeta^{2}_j)|}\tilde c_i(x_{j}, \zeta^{2}_j)$ and hence
\begin{align}
    \|\lambda_{k+1}\|^2 &= \sum_{i=1}^m (\lambda_{k+1, i})^2 \notag \\
    &\leq \sum_{i=1}^m\left(2(\lambda_{1, i})^2 + 2\left|\sum_{j=1}^k\frac{\gamma}{j(\log(j+1))^2|\tilde c_i(x_j, \zeta_j^2)}\tilde c_i(x_j, \zeta_j^2)\right|^2 \right)\notag \\
    &\leq \sum_{i=1}^m (2(\lambda_{1, i})^2 + 32\gamma^2) \notag \\
    &\leq 2\|\lambda_1\|^2+32m\gamma^2,\label{eq: lpw2}
\end{align}
with $|\lambda_{k+1, i}|\leq |\lambda_{1, i}| + 4\gamma$,
where the second step in the inequality chain is by Young's inequality and the third by
\begin{align*}
\left| \sum_{j=1}^k \frac{\gamma}{j(\log(j+1))^2|\tilde c_i(x_{j}, \zeta^{2}_{j})|}\tilde c_i(x_{j}, \zeta^{2}_{j}) \right| &\leq \sum_{j=1}^k \left|\frac{\gamma}{j(\log(j+1))^2|\tilde c_i(x_{j}, \zeta^{2}_{j})|}\tilde c_i(x_{j}, \zeta^{2}_{j})\right| \\
&\leq 4\gamma.
\end{align*}
Note that instead of \eqref{eq: vpt4}, we now have
\begin{align}\label{eq: las2}
\mathbb{E}_k\| \tilde G_{k+1}(x_{k+1}, B_{k+1}) &- \tilde G_{k+1}(x_k, B_{k+1}) \|^2 \le 3\mathbb{E}_k \| \tilde \nabla f(x_{k+1}, \xi^0) - \tilde \nabla f(x_{k}, \xi^0) \|^2\notag \\
&\quad + 3\mathbb{E}_k\left\| \sum_{i=1}^m \left(\tilde \nabla c_i(x_{k+1}, \zeta^{1}_{k+1})-\tilde \nabla c_i(x_{k}, \zeta^{1}_{k+1})\right)\lambda_{k+1, i} \right\|^2 \notag \\
&\quad + 3\rho_{k+1}^2 \mathbb{E}_k\left\|  \sum_{i=1}^m \left(\tilde \nabla c_i(x_{k+1}, \zeta^{1}) \tilde c_i(x_{k+1}, \zeta^{2})  - \tilde \nabla c_i(x_{k}, \zeta^{1})\tilde c_i(x_{k}, \zeta^{2}) \right)\right\|^2.
\end{align}
Note also that, for the second term on the right-hand side of this inequality, we have
\begin{align}
    \mathbb{E}_k\left\| \sum_{i=1}^m \left(\tilde \nabla c_i(x_{k+1}, \zeta^{1}_{k+1})-\tilde \nabla c_i(x_{k}, \zeta^{1}_{k+1})\right)\lambda_{k+1, i} \right\| &\leq \sum_{i=1}^m \mathbb{E}_k\left\|  \left(\tilde \nabla c_i(x_{k+1}, \zeta^{1}_{k+1})-\tilde \nabla c_i(x_{k}, \zeta^{1}_{k+1})\right)\lambda_{k+1, i} \right\| \notag \\
    &\leq  (\|\lambda_1\|+4)\sum_{i=1}^m \mathbb{E}_k \|\tilde \nabla c_i(x_{k+1}, \zeta^{1}_{k+1})-\tilde \nabla c_i(x_{k}, \zeta^{1}_{k+1})\| \notag \\
    &\leq (\|\lambda_1\|+4)m\tilde L_{\nabla c}\|x_{k+1} - x_k\|^2,\label{eq: las3}
\end{align}
where the second inequality used $\max_i|\lambda_{k+1,i}|\leq(\|\lambda_1\|+4)$ and the last inequality used \eqref{eq: asp_gen2} with Jensen's inequality.

We reuse the estimation in Lemma \ref{cor: storm_qp} for the third term on the right-hand side of \eqref{eq: las2}, see \eqref{eq: uew3}. For the second term on the right-hand side of \eqref{eq: las2}, we use \eqref{eq: las3} and obtain (\emph{cf.} \eqref{eq: ngj4})
\begin{align}
&\mathbb{E}_k\| \tilde G_{k+1}(x_{k+1}, B_{k+1}) - \tilde G_{k+1}(x_k, B_{k+1}) \|^2 \notag \\
&\le \left(3\tilde L_{\nabla f}^2 + 3m^2\tilde L_{\nabla c}^2(\|\lambda_1\| + 4)^2  + 6\rho_{k+1}^2m^2\left( \tilde C_c^2\tilde L_{\nabla c}^2 + \tilde C_{\nabla c}^2 \tilde L_{c}^2\right)\right)  \| x_{k+1} - x_k\|^2 \notag \\&\leq  4\tilde L^2\rho_{k+1}^2 \| x_{k+1} - x_k\|^2,\label{eq: las5}
\end{align}
where 
\begin{equation*}
    \tilde L^2 = \frac{3}{4} \tilde L_{\nabla f}^2 + \frac{3}{4}m^2(\|\lambda_1\|+ 4)^2\tilde L_{\nabla c}^2+\frac{3}{2}m^2(\tilde C^2\tilde L_{\nabla c}^2 + \tilde C^2_{\nabla c} \tilde L_c^2).
\end{equation*}
Instead of \eqref{eq: ngj5} we have
\begin{align}
&\mathbb{E}_k\| \tilde G_{k+1}(x_k, B_{k+1}) - \tilde G_{k}(x_k, B_{k+1})\|^2 \notag \\
&\leq 2 \mathbb{E}_k \left\| (\rho_{k+1} - \rho_k) \sum_{i=1}^m \tilde \nabla c_i(x_k, \zeta^{1}) \tilde c_i(x_k, \zeta^{2})  \right\|^2 + 2\mathbb{E}_k\left\|\sum_{i=1}^m \nabla \tilde c_i(x_{k})(\lambda_{k+1, i} - \lambda_{k, i}) \right\|^2 \notag \\
&\leq 2m^2 \tilde C_{\nabla c}^2 \tilde C_c^2 |\rho_{k+1} - \rho_k|^2 + 2\tilde C_{\nabla c}^2 m^2 \tilde \gamma_k^2,\label{eq: las6}
\end{align}
where the last estimate is by \eqref{eq: asp_gen3} and $|\lambda_{k+1, i} - \lambda_{k, i}| = \tilde \gamma_k$ which is due to the definition of $\lambda_{k+1}$.

Moreover, instead of \eqref{eq: ngj6}, we have
\begin{align}
& \mathbb{E} \| G_{k}(x_k) - \tilde G_{k}(x_k, B_{k+1}) \|^2 \notag \\
& \le 3 \E \left\| \nabla f(x_k) - \tilde \nabla f(x_k, \xi^0) \right\|^2 + 3\left\| \sum_{i=1}^m(\tilde \nabla c_i(x_k, \zeta^2_{k+1})-\nabla c_i(x_k))\lambda_{k, i} \right\|^2 \notag \\
&\quad+ 3 \rho_k^2 \E \left\| \sum_{i=1}^m \nabla c_i(x_k) c_i(x_k)  - \sum_{i=1}^m \tilde \nabla c_i(x_k, \zeta^{1}) \tilde c_i(x_k, \zeta^{2})  \right\|^2 \notag \\
&\leq 3\sigma_f^2 + 3\sigma_{\nabla c}^2m^2(2\|\lambda_1\|^2 + 32m\gamma^2) + 6m^2\rho_k^2\left(C_c^2 \sigma^2_{\nabla c} + \tilde C_{\nabla c}^2 \sigma_c^2 \right),\label{eq: las7}
\end{align}
where we used \eqref{eq: lpw2} and \eqref{eq: asp_gen2} for the second term on the right-hand side and the estimation in \eqref{eq: ngj6} for the third term on the right-hand side.

By tracing the same calculations as Lemma~\ref{cor: storm_qp} (i.e., substituting \eqref{eq: las5}, \eqref{eq: las6}, \eqref{eq: las7} into \eqref{lem: storm_estimator_genform}, dividing by $72\tilde L^2 \rho_{k+1}^2 \eta_k$ and arguing the same way as \eqref{eq: gwq3} and the following estimations), we have
\begin{align*}
   \eta_k\mathbb{E}\|g_k - \nabla Q_{\rho_{k}}(x_{k}, \lambda_k)\|^2 &\leq \frac{1}{72\tilde L^2 \rho_{k}^2 \eta_{k-1}} \mathbb{E}\|g_k - \nabla Q_{\rho_{k}}(x_{k}, \lambda_k)\|^2 \notag \\
   &\quad- \frac{1}{72\tilde L^2 \rho_{k+1}^2 \eta_k}\mathbb{E} \|g_{k+1} - \nabla Q_{\rho_{k+1}}(x_{k+1}, \lambda_{k+1})\|^2 \notag\\
    &\quad+\frac{7}{18\eta_k} \mathbb{E}\| x_{k+1} - x_k\|^2 
    + \frac{7 m^2 \tilde C_{\nabla c}^2 \tilde C_c^2}{6\tilde L^2 \rho_{k+1}^2 \eta_k} |\rho_{k+1} - \rho_k|^2 + \frac{7\tilde \gamma_k^2m^2\tilde C_{\nabla c}^2}{6\tilde L^2\rho_{k+1}^2\eta_k}\\
&\quad+\frac{\alpha_{k+1}^2}{8\tilde L^2 \rho_{k+1}^2 \eta_k} \left(\sigma_f^2 + \sigma^2_{\nabla c}m^2(2\|\lambda_1\|^2+32m\gamma^2)+ 2m^2\rho_k^2\left(C_c^2 \sigma^2_{\nabla c} + \tilde C_{\nabla c}^2 \sigma_c^2 \right) \right).
\end{align*}
Note that the additions compared to Lemma~\ref{cor: storm_qp} are the constants in the last term and the fifth term (and as described above $\tilde L$ is different in this case). The fifth term is summable thanks to the definition of $\tilde\gamma_k$ hence the order of the bound is the same.

$\bullet$ Modification of Lemma~\ref{eq: one_it_ineq1}.

In Lemma~\ref{eq: one_it_ineq1}, the only change is that in addition to changing the penalty parameter, we also have to take into account the change in dual variable and also the effect of dual variable size on the Lipschitz constant. The latter is already reflected in the definition of $\tilde L$ earlier in this section.
For changing the dual variable, note that
\begin{align*}
|Q_{\rho_k}(x_{k+1}, \lambda_{k+1}) - Q_{\rho_k}(x_{k+1}, \lambda_{k})| &= \left|\sum_{i=1}^m (\lambda_{k+1, i} - \lambda_{k, i}) c_i(x_{k+1})\right|
\leq C_c \sum_{i=1}^m |\lambda_{k+1, i} - \lambda_k| \notag \\
&\leq mC_c\tilde \gamma_k,
\end{align*}
where we used triangle inequality, \eqref{eq: asp_gen3}, and the definition of $\lambda_{k+1, i}$ that gives $|\lambda_{k+1, i} - \lambda_{k, i}|=\gamma_k$.

Consequently, the result of Lemma~\ref{eq: one_it_ineq1} becomes
\begin{align}
&\frac{\eta_k}{72}\E \dist^2(\nabla f(x_{k+1}) + \nabla c(x_{k+1})^\top \lambda_k + \rho_{k}\nabla c(x_{k+1})^\top c({x_{k+1}}), - N_X(x_{k+1})) \notag \\
&\leq  \E[Y_{k}- Y_{k+1} + |Q_{\rho_k}(x_{k+1}, \lambda_{k+1}) - Q_{\rho_{k+1}}(x_{k+1},\lambda_{k+1})| ] + m C_c \tilde \gamma_k + \mathcal{E}_{k+1},\label{eq: boe3}
\end{align}
where 
\begin{equation*}
Y_{k+1} = Q_{\rho_{k+1}}(x_{k+1}, \lambda_{k+1}) + \frac{1}{72\tilde L^2 \rho_{k+1}^2 \eta_k} \| g_{k+1} - \nabla Q_{\rho_{k+1}}(x_{k+1}, \lambda_{k+1}) \|^2
\end{equation*}
and
\begin{align*}
    \mathcal{E}_{k+1} &= \frac{7 m^2 \tilde C_{\nabla c}^2 \tilde C_c^2}{6\tilde L^2 \rho_{k+1}^2 \eta_k} |\rho_{k+1} - \rho_k|^2 + \frac{7\gamma_k^2 m^2 \tilde C_{\nabla c}^2}{6\tilde L^2 \rho_{k+1}^2\eta_k} \\
    &\quad +\frac{\alpha_{k+1}^2}{8\tilde L^2 \rho_{k+1}^2 \eta_k} \left(\sigma_f^2 + \sigma^2_{\nabla c}m^2(2\|\lambda_1\|^2+32m\gamma^2)+ 2m^2\rho_k^2\left(C_c^2 \sigma^2_{\nabla c} + \tilde C_{\nabla c}^2 \sigma_c^2 \right) \right).
\end{align*}
An important remark here is that the order of the dominant term in $\mathcal{E}_k$ is still $O\left(\frac{1}{k}\right)$ as before.

$\bullet$ Modification of Lemma~\ref{lem: sps2}. 

In this case, the dual variable update will change two estimations. First is~\eqref{eq: hhp4} where we will now have
\begin{align}
&|Q_{\rho_k}(x_{k+1}, \lambda_{k+1}) - Q_{\rho_{k+1}}(x_{k+1}, \lambda_{k+1})| \notag \\
&\leq \frac{3|\rho_k - \rho_{k+1}|}{\rho_{k}^2 \delta}\big( \| \nabla f(x_{k+1})\|^2 + C_c^2 \| \lambda_k\|^2 \notag \\
&\quad+ \dist^2(\nabla f(x_{k+1}) +\nabla c(x_{k+1})^\top \lambda_k + \rho_{k} \nabla c(x_{k+1})^\top c(x_{k+1}), -N_X(x_{k+1}) \big),
\end{align}
where we note that $\|\lambda_k\|^2$ is finite as per \eqref{eq: lpw2} and hence do not change the order of the dominant terms in this bound.

The second is~\eqref{eq: sow4} where now we will have, in view of \eqref{eq: boe3}, that
\begin{align}
&\frac{\eta_k}{72}\mathbb{E} \dist^2(\nabla f(x_{k+1}) + \nabla c(x_{k+1})^\top \lambda_k  + \rho_{k}\nabla c(x_{k+1})^\top c({x_{k+1}}), - N_X(x_{k+1})) \notag \\
&\leq \mathbb{E}\left[ Y_{k}- Y_{k+1}\right] + \mathcal{E}_{k+1} + |\rho_k - \rho_{k+1}| C_c^2
 + mC_c\tilde \gamma_k ,\label{eq: asq3}
\end{align}
where the additional error term $\tilde \gamma_k$ is summable by definition and the rest of the proof of Lemma~\ref{lem: sps2} would be the same to get $\sum_{k=1}^K |Q_{\rho_k}(x_{k+1}, \lambda_{k+1}) - Q_{\rho_{k+1}}(x_{k+1}, \lambda_{k+1})|=O(1)$.

By combining these modified results as in Theorem~\ref{eq: asb4} we deduce the result.
\end{proof}

\subsection{Deterministic functional constraints}\label{subsec: ext2}
In this section, we consider the case when the constraints are not given in the expectation form. In this case, we will set the parameters accordingly to get the complexity $\tilde O(\varepsilon^{-4})$.

\begin{reptheorem}{th: th_determ_nonl}
For Algorithm~\ref{algorithm:stoc_lalm}, set
\begin{align}
\eta_k &= \frac{1}{9\tilde L \rho(k+1)^{1/2}}, \quad
\rho_k = \rho k^{1/4}, \notag \\
\alpha_{k+1} &=  \frac{72}{81(k+1)^{1/2}},\notag
\end{align}
for some  $\rho>1$ and $\tilde{L}^2 = 4\tilde L_{\nabla f}^2 + 4m^2 (\tilde C_c^2 \tilde L_{\nabla c}^2 + \tilde C_{\nabla c}^2 \tilde L_c^2)$. Let the assumptions in \eqref{eq: htr_eqs}, \eqref{eq: asp_gen2}, \eqref{eq: asp_gen3}, \eqref{eq: asp_gen4} hold with a deterministic $c(x)$. We have that there exists $\lambda$ such that
\begin{align*}
    \mathbb{E}\left[\dist(\nabla f(x_{\hat k+1}) +  \nabla c(x_{\hat k+1})^\top \lambda, -N_X(x_{\hat k +1}))\right] &\leq \varepsilon,\\
    \mathbb{E}\|c(x_{\hat k + 1})\|&\leq\varepsilon,
\end{align*}
with number of iterations bounded by $\widetilde{O}(\varepsilon^{-4})$.
\end{reptheorem}
\begin{proof}
Since the orders of the parameter choices differ in this case, we will show the changes in the analysis of Section~\ref{sec: qp_stoc_const}.

First, in Lemma~\ref{cor: storm_qp}, the main change will be that we use full gradients for the constraints. In particular, we have
\begin{align*}
G_{k}(x) &= \nabla Q_{\rho_k}(x) = \nabla f(x) + \rho_k \sum_{i=1}^m c_i(x) \nabla c_i(x),\\
\tilde G_{k}(x, \xi) &= \tilde \nabla Q_{\rho_k}(x, \xi) = \tilde \nabla f(x, \xi) + \rho_k \sum_{i=1}^m c_i(x) \nabla c_i(x).
\end{align*}
As a result, the main change will be in the variance term, i.e.:
\begin{equation*}
\mathbb{E} \| G_{k}(x_k) - \tilde G_{k}(x_k, \xi_{k+1}) \|^2 = \mathbb{E} \| \nabla f(x) - \tilde \nabla f(x, \xi) \|^2 \leq \sigma_f^2.
\end{equation*}

Since in this case, we have $72\tilde L^2 \rho_{k+1}^2 \eta_k = \frac{72\tilde L \rho}{9}$, it follows that~\eqref{eq: gwq3} holds by $\alpha_{k+1} \leq 1$ due to $1-\alpha_{k+1} + \alpha_{k+1}^2 \leq 1$ (note that this is sufficient as per \eqref{eq: hrk3} due to the term $72\tilde L^2\rho_{k+1}^2\eta_k$ being independent of $k$ in this case). This gives, instead of the result of Lemma~\ref{cor: storm_qp}, that
\begin{align}
   \eta_k\mathbb{E}\|g_k - \nabla Q_{\rho_{k}}(x_{k})\|^2 &\leq \frac{1}{72\tilde L^2 \rho^2 } \mathbb{E}\|g_k - \nabla Q_{\rho_{k}}(x_{k})\|^2 - \frac{1}{72\tilde L^2 \rho^2 }\mathbb{E} \|g_{k+1} - \nabla Q_{\rho_{k+1}}(x_{k+1})\|^2 \notag\\
    &\quad+\frac{7}{18\eta_k} \| x_{k+1} - x_k\|^2 
    + \frac{7 m^2 \tilde C_{\nabla c}^2 \tilde C_c^2}{12\tilde L^2 \rho^2} |\rho_{k+1} - \rho_k|^2 +\frac{\alpha_{k+1}^2\sigma_f^2}{12\tilde L^2 \rho ^2}.
\end{align}
The numerical estimations in Lemma~\ref{eq: one_it_ineq1} are true with the new parameters since we still have $ L_{\rho_k} \leq \tilde L_{\rho_k} \leq \tilde L_{\rho_{k+1}} \leq \tilde L \rho_{k+1}(k+1)^{1/4}=\frac{1}{9\eta_k}$ and hence $\frac{L_{\rho_k}}{2} + \frac{7}{18\eta_k} - \frac{1}{2\eta_k} \leq -\frac{1}{18\eta_k}$. Moreover, for the estimations at the end of the proof of Lemma~\ref{eq: one_it_ineq1}, we still have that $L_{\rho_k}\leq \eta_k^{-1}$ and as a result we have (\emph{cf.} the result of Lemma~\ref{eq: one_it_ineq1})
\begin{align*}
 \frac{\eta_k}{72}\E \dist^2(\nabla f(x_{k+1}) + \rho_{k}\nabla c(x_{k+1})^\top c({x_{k+1}}), - N_X(x_{k+1})) \leq  \E[Y_{k}- Y_{k+1} + |Q_{\rho_k}(x_{k+1}) - Q_{\rho_{k+1}}(x_{k+1})|] + \mathcal{E}_{k+1},
\end{align*}
where 
\begin{equation*}
Y_{k+1} = Q_{\rho_{k+1}}(x_{k+1}) + \frac{1}{72\tilde L^2 \rho^2} \| g_{k+1} - \nabla Q_{\rho_{k+1}}(x_{k+1}) \|^2
\end{equation*}
and
\begin{equation*}
\mathcal{E}_{k+1} = \frac{7 m^2 \tilde C_{\nabla c}^2 \tilde C_c^2}{12\tilde L^2 \rho^2} |\rho_{k+1} - \rho_k|^2 +\frac{\alpha_{k+1}^2\sigma_f^2}{12\tilde L^2 \rho^2}.
\end{equation*}
Note that as Remark \ref{rem: cons1}, we have that $\sum_{k=1}^K\mathcal{E}_{k+1}=O(\log (K+1))$ by the definitions of $\alpha_{k+1}$ and $\rho_k$ since $|\rho_k - \rho_{k+1}|\leq \frac{\rho}{4k^{3/4}}$ and $\alpha_k = \frac{72}{81(k+1)^{1/2}}$.

For Lemma~\ref{lem: sps2}, we use $b_k = \frac{ 72\cdot 18\tilde L}{\delta(k+1)^{3/4}}$ to have
\begin{equation*}
b_k \cdot \frac{\eta_k}{72} = \frac{2}{\rho \delta (k+1)^{5/4}} 
\geq \frac{1}{2\rho\delta(k)^{5/4}} 
\geq \frac{2|\rho_k - \rho_{k+1}|}{\rho_k^2\delta},
\end{equation*}
by also using $|\rho_k - \rho_{k+1}|\leq \frac{\rho}{4k^{3/4}}$ and $\rho_k = \rho k^{1/4}$. We note also that, in the same way as Lemma~\ref{lem: sps2}, we have $\left| \frac{1}{(k+1)^{3/4}} - \frac{1}{k^{3/4}} \right| \leq \frac{1}{(k+1)^{3/4}k}$. With these estimations and by repeating the same arguments as Lemma~\ref{lem: sps2}, we get
\begin{align*}
&\sum_{k=1}^K \mathbb{E} |Q_{\rho_k}(x_{k+1}) - Q_{\rho_{k+1}}(x_{k+1})| \\
&\leq b Q_{\rho_1}(x_1) - \frac{b\underline Q}{(K+1)^{3/4}} + \frac{b}{72\tilde L^2 \rho^2} \| g_1 - \nabla Q_{\rho_1}(x_1)\|^2 \\
&\quad+ \sum_{k=1}^K \frac{b(B_f+\rho C_c^2)}{k(k+1)^{1/2}} + \sum_{k=1}^K \frac{b \mathcal{E}_{k+1}}{(k+1)^{3/4}} + \sum_{k=1}^K  \frac{b \rho C_c^2}{k^{3/2}} + \sum_{k=1}^K \frac{2C_{\nabla f}^2}{4\rho\delta k^{5/4}},
\end{align*}
where $b_k = \frac{72\cdot 18\tilde L}{\delta(k+1)^{3/4}}$, $b = \frac{72\cdot 18\tilde L}{\delta}$. As $\mathcal{E}_k=O\left(\frac{1}{k}\right)$, the right-hand side on this inequality is finite. We can then combine these inequalities the same way as Theorem~\ref{eq: asb4} and use the definitions of $\eta_k = \frac{1}{9\tilde L \rho (k+1)^{1/2} }$ and $\rho_k=\rho k^{1/4}$ to get the result.
\end{proof}
\end{document}